\documentclass [10pt, final]{article}

\usepackage{amssymb}
\usepackage{amsfonts}
\usepackage{graphicx}
\usepackage{amsmath,  amsthm}
\usepackage{mathrsfs}
\usepackage{frcursive}
\usepackage{bbm}
\usepackage{hyperref}

\usepackage[normalem]{ulem}
\usepackage[margin=3cm]{geometry}

\usepackage[notref,notcite]{showkeys}        
\usepackage{enumitem}

\usepackage{xcolor}

\numberwithin{equation}{section}

\def\[{\left[}
\def\]{\right]}
\def\({\left(}
\def\){\right)}

\newcommand{\R}{\mathbb{R}}

\newtheorem{theorem}{Theorem}[section]

\newtheorem{assumption}[theorem]{Assumption}

\newtheorem{corollary}[theorem]{Corollary}

\newtheorem{lemma}[theorem]{Lemma}

\newtheorem{proposition}[theorem]{Proposition}

\theoremstyle{definition}
\newtheorem{definition}[theorem]{Definition}
\newtheorem{remark}[theorem]{Remark}

\renewcommand{\eqref}[1]{(\ref{#1})}
\newcommand\dd{\mathrm{d}}

\begin{document}
	
\begin{center}
	\hypersetup{hidelinks}
	\vspace{1cm}
	\renewcommand{\thefootnote}{\fnsymbol{footnote}}
	\begin{minipage}{0.95\textwidth}
		\centering
		%\LARGE{\bf Existence and multiplicity of principal eigenpairs for nonlocal operators with sign-changing advection in a periodic environment.}\bigskip
		%\LARGE{\bf Multiplicity of principal eigenpairs for advective nonlocal operators.}\bigskip
		\LARGE{\bf On advective nonlocal operators: multiplicity of principal eigenpairs.}\bigskip
	\end{minipage}

	\Large
	 Arnaud Ducrot$^{a}$, Quentin Griette$^{a}$, and Xing Liang$^{b}$. \medskip \\
	\bigskip

	\normalsize
	\today
	\medskip

	{\it $^a$ Universit\'e Le Havre Normandie, Normandie Universit\'e, LMAH, 76600 Le Havre, France.}\\
	\medskip

	{\it $^b$ School of Mathematical Sciences, University of Science and Technology of China, Hefei, Anhui 230026, China. 
	}
	\hypersetup{hidelinks=false}
\end{center}

	\begin{abstract}
		We study the existence and multiplicity of principal eigenvalues and eigenfunctions for a periodically heterogeneous nonlocal  dispersal model with advection. The operator we consider is resolvent-positive but not resolvent-compact; therefore, the classical Krein-Rutman theory cannot be applied directly. 
		When the advection coefficient has a constant sign, we prove the existence and uniqueness of the principal eigenvalue and the corresponding normalized eigenfunction. 
		In sharp contrast, when the advection does not have a constant sign, the problem is more involved and leads to surprising results. 
		Depending on the coefficients of the equation, the principal eigenproblem can either have a unique normalized solution or a continuum of solutions, at the boundary of which there exists a principal eigenvector with a singular measure component. 
		In the latter situation, all the constructed eigenvalues are embedded in the continuous spectrum of our operator.
		We completely characterize the eigenvalues associated with positive eigenvectors, even when the eigenvector is a Radon measure. 
		Finally, we discuss an application to a nonlinear KPP-type equation with nonlocal dispersal, which possesses a continuum of nontrivial stationary solutions, a different behavior from the classical KPP equation with local diffusion.

		\vspace{0.2in}\noindent \textbf{Key words}. Non-local operator, sign-changing advection, principal eigenvalue, continuous spectrum, resolvent-positive operator. 
		
		\vspace{0.1in}\noindent \textbf{2020 Mathematical Subject Classification}. 
		Primary:
		47H07, %Monotone and positive operators on ordered Banach spaces or other ordered topological vector spaces
		35P05, %General topics in linear spectral theory for PDEs
		45C05, %Eigenvalue problems for integral equations 
		47A75, %Eigenvalue problems for linear operators
		47G20. %Integro-differential operators
		Secondary:
		92D25. %Population dynamics (general)
		
	\end{abstract}

\section{Introduction}

The aim of this article is to investigate the existence and multiplicity of the positive $1$-periodic solutions to the equation
\begin{equation}\label{eq:main}
	a(x)\varphi'(x)+r(x)\varphi(x)+\int_{\mathbb{S}^1} K(x,y)\varphi(y)dy=\lambda \varphi(x),
\end{equation}
where $\lambda$ and $\varphi(x)>0$ are the unknowns, 
$K$ is a positive continuous kernel acting on $\mathbb{S}^1\times \mathbb{S}^1$ and  $a(x)$ and $r(x)$ are bounded one-periodic functions. A solution $(\lambda, \varphi)$ of \eqref{eq:main} with 
%$\varphi(x)\gneqq 0$ 
$\varphi(x)\geq 0$ and $\varphi(x)\not\equiv 0$
is called a \textit{principal eigenpair}, $\lambda$ is a \textit{principal eigenvalue} and $\varphi$ is a \textit{principal eigenfunction}.
Here, crucially, we do not assume that the advection coefficient $a(x)$ has a constant sign. 
\medskip

Principal eigenpairs are an important tool in the analysis of nonlinear reaction-diffusion equations, whether the diffusion is local or not. They have been used to determine the local  stability of equilibria \cite{Rawal-Shen-2012, Lam-Lou-2016}, the existence of positive stationary solutions \cite{Rawal-Shen-2012, Alfaro-Griette-2018, Griette-2019}, and principal eigenvalues are involved in the computation of the spreading speed in such equations \cite{Freidlin-Gertner-1979, Weinberger-2002, Berestycki-Hamel-Nadirashvili-2005}. In population dynamics models, they are useful to establish persistence thresholds or to compute the basic reproduction number of a population \cite{Diekmann-Heesterbeek-Metz-1990, van_den_Driessche-Watmough-2002}. 
\medskip

Reaction-diffusion equations involving local or nonlocal diffusion have attracted a lot of interest in recent years, notably through the growing interest in the propagation properties of population dynamics models \cite{Thieme-1977, Diekmann-1979, Weinberger-1982, Weinberger-2002}. 
\medskip

\noindent \textbf{Principal eigenvalues of local diffusion operators.} In the case of local diffusion operators set on bounded domains $\Omega\subset\mathbb{R}^N$ of the type $\mathcal{L}_Du=\Delta u+a(x)\cdot \nabla u + r(x) u$, the existence and uniqueness of a principal eigenpair can be established by the Krein-Rutman Theorem, see e.g. \cite{Aronson-Weinberger-1975, Berestycki-Nirenberg-Varadhan-1994, Birindelli-1995}. 
In the case of unbounded operators, the existence of a principal eigenvector can prove to be difficult. 
One can rely on different generalized notions of principal eigenvalues \cite{Berestycki-Nirenberg-Varadhan-1994,Berestycki-Rossi-2006} such as
\begin{equation*}
	\lambda_1(\mathcal{L}, \Omega):=\inf\{\lambda\in\mathbb{R}\,:\,\exists \Phi \in C^2(\Omega)\cap C^1_{loc}(\Omega), \Phi >0 \text{ and } \mathcal{L}\Phi\leq \lambda\Phi\}, 
\end{equation*}
or 
\begin{equation*}
	\lambda_1'(\mathcal{L}, \Omega):=\sup\{\lambda\in\mathbb{R}\,:\,\exists \Phi \in C^2(\Omega)\cap C^1_{loc}(\Omega)\cap W^{2,\infty}(\Omega), \Phi >0 \text{ and } \mathcal{L}\Phi\geq \lambda\Phi \text{ in } \Omega, \Phi=0 \text{ on } \partial\Omega\}, 
\end{equation*}
for the Dirichlet problem. This indicates the possibility of more intricate dynamics: it turns out that the sign of $\lambda_1(\mathcal{L}, \Omega)$ is related to the existence of positive solutions of semilinear problems, while $\lambda_1'(\mathcal{L}, \Omega)$ is related to the nonexistence of such solutions. Since then, these notions have been generalized to more complex settings, such as time-heterogeneous equations \cite{Nadin-2009}.

%[Link with maximum principle]
\medskip

\noindent \textbf{Principal eigenvalues for nonlocal operators without advection.} For equations involving nonlocal dispersal and advection, early results have been established by Coville \cite{Coville-2010,Coville-2013}, Rawal and Shen \cite{Rawal-Shen-2012}, Ding and Liang \cite{Ding-Liang-2015}. Even without advection, the existence of principal eigenfunctions in the case of nonlocal operators is not always guaranteed, and many references in the existing literature focus on sufficient conditions for the existence. For operators acting on a bounded set $\Omega\subset\mathbb{R}^N$ such as $\mathcal{L}_{NL}\Phi = \int_{\overline{\Omega}}K(x, y)\Phi(y)\dd y + r(x)\Phi(x)$, Coville \cite{Coville-2010} proved that a sufficient condition for the existence of a continuous principal eigenfunction is
\begin{equation}\label{eq:260522a}
	\frac{1}{r_\infty-r(x)}\not\in L^1(\overline{\Omega}), \text{ where }r_\infty:=\sup_{x\in \overline{\Omega}} r(x).
\end{equation}
When \eqref{eq:260522a} does not hold, he later constructed some counterexamples \cite{Coville-2013} in which a principal eigenvector can be constructed in the space of measures and showed that this principal eigenvector is not absolutely continuous with respect to the Lebesgue measure. In this situation, the first author \cite{Griette-2019} later established a sharp condition for the existence of a continuous principal eigenpair based on the spectral radius $\kappa $ of the operator
\begin{equation*}
	\mathcal{K}\Phi(x):=\int_{\overline{\Omega}} K(x, y) \dfrac{\Phi(y)}{r_\infty-r(y)}\dd y.
\end{equation*}
More precisely, when $\kappa>1$, there exists a unique continuous eigenfunction; when $\kappa=1$, there exists a unique principal eigenfunction that belongs to $L^1$; and when $\kappa<1$, there exists a family of principal eigenvectors that all possess a singularity with respect to the Lebesgue measure.  Moreover, the principal eigenpairs all share the same eigenvalue $\lambda_1$. By adapting the method presented in the current manuscript, it would actually be possible to characterize all principal eigenpairs for this problem in the sense of measures.
\medskip

\noindent \textbf{Principal eigenvalues for nonlocal operators with advection.} Recently, Coville and Hamel \cite{Coville-Hamel-2020} studied the eigenvalue problem associated with the operator $\mathcal{L}_{CH}\Phi(x):=a(x)\Phi'(x) + \int_{\Omega}K(x,y)\dd y+r(x)$, where $\Omega\subset \mathbb{R}$ is an open interval and the advection coefficient $a(x)$ does not change sign. They proved, in particular, the existence of a principal eigenpair  associated with this operator,  involving a regular eigenfunction. They also studied the relations between the constructed eigenvalue and different notions of generalized principal eigenvalue. To the extent of our knowledge, the study of Coville and Hamel is the unique reference dealing with a principal eigenpair for nonlocal operators with non-constant advection in the current literature. 

Our method also applies to this case, in that we are in a position to prove the existence and uniqueness of a principal eigenpair in the case of an advection coefficient with a constant sign. We state our result in subsection \ref{subsec:constant-sign}.

\medskip

\noindent \textbf{Freidlin-G\"{a}rtner formula.}
Apart from establishing a criterion for the stability or instability of equilibria, an interesting application of the principal eigenpair is the so-called \textit{Freidlin-G\"{a}rtner formula} \cite{Freidlin-Gertner-1979}, which links a family of principal eigenproblems to the spreading speed associated with space or time heterogeneous operators. In the case of local diffusion operators with periodically heterogeneous coefficients such as the operator $\mathcal{L}_D$ defined above, Berestycki, Hamel and Nadirashvili \cite{Berestycki-Hamel-Nadirashvili-2005} prove that the spreading speed $c^*_{KPP}$ of KPP-type problems can be obtained as
\begin{equation*}
	c^*_{KPP}=\inf_{\lambda>0} \dfrac{k(\lambda)}{\lambda}, 
\end{equation*}
where $k(\lambda)$ is the principal eigenvalue of the operator $\mathcal{L}_\lambda\Phi(x):= e^{\lambda x}\mathcal{L}_D\big[ e^{-\lambda x} \Phi(x)\big]$. This formula is presently well-established in different contexts, including the case of a non-local diffusion \cite{Weinberger-2002, Shen-Zhang-2012, Coville-Davila-Martinez-2013, Liang-Zhou-2022a, Liang-Zhou-2022b}.

\medskip

\noindent \textbf{Forced waves.} Our problem \eqref{eq:main} is particularly relevant to the analysis of the existence and stability of forced waves. Early occurrences of such models, which can be used to model the propagation and adaptation of a species under climate change, go back to Berestycki et al \cite{Berestycki-Diekmann-Nagelkerke-Zegeling-2009}, Berestycki and Fang \cite{Berestycki-Fang-2018}, Bouhours and Giletti \cite{Bouhours-Giletti-2019}, among others. Let $\tilde{a}(x)$ be an advection coefficient (having a constant sign or not), and consider the reaction-diffusion equation
\begin{equation*}
	u_t(t, x) = \tilde{a}(x)u_x(t, {x})+ \int_{\mathbb{R}}\tilde{K}(x,y)\big(u(t, y)-u(t, x)\big)\dd y + f\big(x-ct, u(t, x)\big),
\end{equation*}
where $\tilde{K}(x+1, y)=\tilde{K}(x, y)$ is a nonlocal dispersal kernel and $f(x+1, u)=f(x, u)$. By the change of variables $\xi=x-ct$, one reaches the equation
\begin{equation*}
	u_t(t, \xi) = {\big(\tilde{a}(\xi+ct)+c\big)u_{\xi}(t, \xi)}+ \int_{\mathbb{R}}\tilde{K}(\xi+ct,y+ct)\big(u(t, y)-u(t, \xi)\big)\dd y + f\big(\xi, u(t, \xi)\big).
\end{equation*}
Hence, we recover a nonlocal diffusion equation where the advection coefficient is replaced by $a_t(x)=\tilde{a}(x+ct)-c$. In the context of KPP-type nonlinearities, the existence of a pulsating wave is strongly linked with the sign of the principal eigenvalue of the linearized operator around the unstable state $u\equiv 0$ \cite{Xin-2000, Berestycki-Hamel-2002, Coville-Davila-Martinez-2013, Alfaro-Griette-2018}. Since $c$ is arbitrary, we can easily construct some models where  the study of a nonlocal diffusion with a sign-changing kernel becomes necessary. 

The difference with our model lies in the dependency on the new advection coefficient $a_t(x)$ and the new kernel $  K_t(\xi,y)=\tilde{K}(x+ct, y+ct)$.   We plan to investigate such a situation in future works. 
\medskip

\textbf{Eigenvalues of order-preserving operators.} Order-preserving operators are known to enjoy particular spectral properties, especially concerning the extremity of the spectrum (spectral bound or spectral radius). For a positive and irreducible matrix $A$, the Perron-Frobenius Theorem \cite[Theorem 4.25 p.138]{Ducrot-Griette-Liu-Magal-2022} ensures that the spectral radius $r(A)$ is a simple eigenvalue that is associated with a positive eigenvector. This theorem admits generalizations in the setting of Banach spaces; particularly when the operator is compact and preserves a cone with a non-empty interior, the Krein-Rutman Theorem (sometimes called the Jentzch-Perron Theorem as in \cite[Corollary 4.2.14 p.273]{Meyer--Nieberg-1991}; we also refer to Schaefer \cite{Schaefer-1974}) ensures that the spectral radius  is a simple eigenvalue that is isolated,  associated with a unique positive normalized  eigenvector. Birindelli \cite{Birindelli-1995} adapted the Krein-Rutman Theorem in the context of cones with empty interiors, under a  stronger condition on the operator, so that it can be applied to the resolvent of elliptic operators on $L^n(\Omega)$, where $\Omega$ is a bounded domain of $\mathbb{R}^n$. Thieme \cite{Thieme-1998a, Thieme-1998b} introduced a method close to ours in the context of integrated semigroups to study the spectral radius of integrated semigroups and  establishes some local exponential stability criteria for nonlinear equations; in the process, he proved the existence of a principal eigenpair for some class of age-structure models. 
However, to the best of our knowledge, there exist no previous results on the multiplicity of principal eigenvalues for resolvent-positive operators in the literature.
\medskip

The structure of our article is as follows. In section \ref{sec:main-results}, we state our main results and prove the existence and uniqueness of a principal eigenpair for \eqref{eq:main} when $a(x)$ does not change sign, whether it is strictly positive or has a non-degenerate zero of order two. 
In Section \ref{sec:generation}, we recall some classical results about semigroups and prove that our operator generates a positive and strongly continuous semigroup in $L^1_{per}$.
In section \ref{sec:construction}, we deal with the more involved situation of  an advection term $a(x)$ with two zeros of order 1: depending on the other parameters, we prove either the existence and uniqueness of a principal eigenpair or prove the existence of a continuum of eigenpairs, at the boundary of which there exists a particular eigenpair with a  singular measure eigenvector.
As a result, we completely describe the set of positive solutions of \eqref{eq:main} in the space of Radon measures. 
Finally, in section \ref{sec:KPP}, we present an application to a nonlinear KPP-type equation, where a continuum of stationary solutions exists, corresponding to the constructed continuum of eigenpairs for the linearized problem. This showcases the differences between the local and non-local dispersal operators with diffusion.
	
	\section{Main results}
	\label{sec:main-results}
We consider $1$-periodic functions $a:\R\to \R$ and $r:\R\to\R$, of class $C^1$ and $C^0$ respectively.
We also consider a continuous function $K:\R\times [0,1]\to\R$ and we assume that
$$
K(x+1,y)=K(x,y),\;\forall (x,y)\in [0,1]\times [0,1]\text{ and }K(x,y)>0,\;\forall (x,y)\in [0,1]^2.
$$	
We denote $L^1_{\rm per}(\R)$ the set of the functions that are measurable, $1$-periodic and integrable on $(0,1)$, and we define the linear operators
\begin{equation*}%\label{defB}
B\in \mathcal L\left(L^1_{\rm per}(\R)\right),\;B\varphi\, (x)=\int_0^1 K(x,y)\varphi(y)dy,\;\forall \varphi\in L^1_{\rm per}(\R),
\end{equation*}
and $A:D(A)\subset L^1_{\rm per}(\R)\to L^1_{\rm per}(\R)$ with
\begin{equation}\label{defA}
\begin{split}
&D(A)=\left\{\varphi\in L^1_{\rm per}(\R):\;a\varphi'\in L^1_{\rm per}(\R)\right\},\\
&A\varphi=a(x)\varphi'+r(x)\varphi,\;\forall \varphi\in D(A).
\end{split}
\end{equation} 
We also define 
$A_0:D(A_0)\subset C^0_{per}(\R)\to C^0_{per}$ with
\begin{equation}\label{defA0}
\begin{split}
    &D(A_0)=\left\{\varphi\in C^0_{per}:\;a\varphi'\in C^0_{per}\right\},\\
&A_0\varphi=a(x)\varphi'+r(x)\varphi,\;\forall \varphi\in D(A_0).
\end{split}
\end{equation} 
We will denote $L^1_{per,+}(\R)$ the positive cone of $L^1_{per}(\R)$, which consists of all non-negative functions in this space. Similarly, $C^0_{per, +}=L^1_{per,+}\cap C^0_{per} $ will denote the positive cone of $C^0_{per}$.
\medskip

	Our basic assumption reads as follows. 

\begin{assumption}\label{ASS-10}
	The functions $a(x)$ and $r(x)$ are $1$-periodic with $ a\in C^1_{per}$ and $r \in C^0_{per}$. The kernel $K(x, y)$ is a continuous function defined on $\mathbb{S}^1\times \mathbb{S}^1$, which is strictly positive: $K(x, y)>0$ for all $(x, y)\in \mathbb{S}^1\times \mathbb{S}^1$.
\end{assumption}
We start by defining the notion of principal eigenpair, which is the main subject of this article.
\begin{definition}[Principal eigenpair]
	A \textit{principal eigenpair} of the operator $A+B$ is a pair $(\lambda, \varphi)$ with $\varphi\in D(A)$, $\varphi\geq 0$ and $\varphi\not\equiv 0$, such that $\varphi$ solves \eqref{eq:main} almost everywhere. A \textit{normalized principal eigenpair} $(\lambda, \varphi)$ is a principal eigenpair that satisfies moreover $\int_{\mathbb{S}^1} \varphi(y)\dd y=1$.
\end{definition}
Our method to obtain principal eigenvalues of $A+B$ is to regularize the principal eigenproblem by using the resolvent of the unbounded operator involved in the sum $A+B$: we transform $(A+B)\varphi=\lambda\varphi$ into $(\lambda I-A)\varphi =  B\varphi$ and reach
\begin{equation*}
	\varphi=(\lambda I-A)^{-1}B\varphi. 
\end{equation*}
Hence the operator $T(\lambda)=(\lambda I-A)^{-1}B$, which is a priori defined for $\lambda\in \rho(A)$ (the resolvent set of $A$) is of particular significance in our analysis: if $T(\lambda)$ has an eigenvector $\varphi$ with the eigenvalue 1, then this eigenvector is a solution of our problem \eqref{eq:main} with the eigenvalue $\lambda$.

The existence and multiplicity of principal eigenvalues for the operator $A+B$ depend on the number of zeros of $a(x)$ and the order of these zeros (understood as the order of the first non-vanishing derivative). We start with the simplest case, which is the case when $a(x)$ has a constant sign; the existence of a principal eigenpair has been established in the literature  (see e.g. \cite{Coville-Hamel-2020}) by a different method, but we believe that this case is an illuminating example for our method.

\subsection{The advection does not vanish}
\label{subsec:constant-sign}
In this case, the argument is relatively quick, so we present it here. Suppose that $a(x)>0$ and fix two arbitrary numbers $X_0\in [0, 1)$ and $X_1\in [0, 1)$. For any $f\in L^1_{per}$, we can compute the resolvent of $A$ as follows:
\begin{align*}
	(\lambda I-A)\varphi = f&\Leftrightarrow \big(a(x)\varphi\big)' + \big(r(x)-a'(x)-\lambda\big) \varphi = -f(x) \\
	&\Leftrightarrow \left(a(x) e^{\int_{X_0}^x \frac{\sigma(z)-\lambda}{a(z)}\dd z}\varphi\right)' e^{-\int_{X_0}^x \frac{\sigma(z)-\lambda}{a(z)}\dd z} = -f(x)\\
	&\Leftrightarrow a(x)\varphi(x)  e^{\int_{X_0}^x \frac{\sigma(z)-\lambda}{a(z)}\dd z} - a(X_1)\varphi(X_1) e^{\int_{X_0}^{X_1} \frac{\sigma(z)-\lambda}{a(z)}\dd z}=-\int_{X_1}^xf(y)e^{\int_{X_0}^y \frac{\sigma(z)-\lambda}{a(z)}\dd z} \dd y, 
\end{align*}
where $\sigma(x)=r(x)-a'(x)$. Since we want $\varphi$ to be $1$-periodic, there is a unique possible value for $\varphi(X_1)$, which is 
\begin{equation}\label{eq:no0-phiX1}
	\varphi(X_1)=\dfrac{-\int_{X_1}^{X_1+1}f(y)e^{\int_{X_0}^y \frac{\sigma(z)-\lambda}{a(z)}\dd z} \dd y}{a(X_1)e^{\int_{X_0}^{X_1} \frac{\sigma(z)-\lambda}{a(z)}\dd z}\left(e^{\int_{X_1}^{X_1+1} \frac{\sigma(z)-\lambda}{a(z)}\dd z}-1\right)},
\end{equation}
provided $e^{\int_{X_1}^{X_1+1} \frac{\sigma(z)-\lambda}{a(z)}\dd z}-1\neq 0$, or equivalently $\int_{0}^{1} \frac{\sigma(z)-\lambda}{a(z)}\dd z\neq 0$. We conclude that the resolvent set of $\big(A, D(A)\big)$ is precisely 
\begin{equation*}%\label{eq:no0-res-set}
	\rho(A)=\left\{\lambda\in\mathbb{C}\,:\, \int_{0}^{1} \frac{\sigma(z)-\lambda}{a(z)}\dd z\not\in 2\pi i\mathbb{Z}\right\}, 
\end{equation*}
and the resolvent is 
\begin{equation*}%\label{eq:no0-res}
	R(\lambda; A)f(x):= (\lambda I-A)^{-1}f(x)=\dfrac{-\int_{X_1}^xf(y)e^{\int_{X_0}^y \frac{\sigma(z)-\lambda}{a(z)}\dd z} \dd y + a(X_1)\varphi(X_1)e^{\int_{X_0}^{X_1} \frac{\sigma(z)-\lambda}{a(z)}\dd z} }{a(x) e^{\int_{X_0}^{x} \frac{\sigma(z)-\lambda}{a(z)}\dd z}}, 
\end{equation*}
for any $\lambda\in\rho(A)$, where $\varphi(X_1)$ is defined by \eqref{eq:no0-phiX1}. 
We can check that 
\begin{equation}\label{eq:260512a}
	R(\lambda;A)f=\dfrac{-\int_{X_1}^x f(y)e^{\int_{X_0}^y\frac{\sigma(z)-\lambda}{a(z)}\dd z}\dd y+\frac{1}{1-e^{\int_{X_1}^{X_1+1}\frac{\sigma(z)-\lambda}{a(z)}\dd z}}\int_{X_1}^{X_1+1}f(y)e^{\int_{X_0}^{y}\frac{\sigma(z)-\lambda}{a(z)}\dd z}\dd y}{a(x)e^{\int_{X_0}^x\frac{\sigma(z)-\lambda}{a(z)}\dd z}}.
\end{equation}
Hence if 
\begin{equation}\label{eq:260512b}
	\lambda>\lambda_{crit}:=\frac{1}{\int_0^1 a^{-1}(z)\dd z}\int_0^1 \frac{\sigma(z)}{a(z)}\dd z,
\end{equation}
then we have $\lambda\in \rho(A)$ and $\frac{1}{1-e^{\int_{X_1}^{X_1+1}\frac{\sigma(z)-\lambda}{a(z)}\dd z}}> 1$, and we deduce from \eqref{eq:260512a} that $R(\lambda;A)f\geq 0$ whenever $f\geq 0$; in other words, $R(\lambda;A)$ is a positive operator. Define 
\begin{equation*}
	T(\lambda)\varphi := R(\lambda; A)B\varphi.
\end{equation*}
Whenever \eqref{eq:260512b} holds true, the operator $T(\lambda)$ is positive, irreducible, and compact, hence possesses a unique normalized principal eigenpair $(\kappa(\lambda), \varphi_\lambda)$ with $\varphi_{\lambda}>0$; moreover $\kappa(\lambda)=r\big(T(\lambda)\big)$, the spectral radius of $T(\lambda)$.
It follows from classical arguments (that can be adapted from the proof of proposition \ref{prop:eigen-0} below) that $\lambda\mapsto \kappa(\lambda)$ is continuous,  strictly decreasing, $\kappa(\lambda) \to 0$ as $\lambda\to+\infty$ and $\kappa(\lambda)\to +\infty$ as {$\lambda\to \lambda_{crit}^+$}. 
Hence there exists a unique value $\lambda=\lambda_1^*>\lambda_{crit}$ such that $\kappa(\lambda)=1$. We finally note that
\begin{equation*}
	(A+B)\varphi=\lambda\varphi \Leftrightarrow (\lambda I-A)\varphi=B\varphi \Leftrightarrow \varphi=R(\lambda; A)B\varphi=T(\lambda)\varphi, 
\end{equation*}
and we conclude that there always exists a principal eigenpair $(\lambda_1^*, \varphi_{\lambda_1^*})$,  which corresponds to the principal eigenvalue of $T(\lambda)$ satisfying $\kappa(\lambda)=1$. By adapting the proof of Lemma \ref{lem:estimates-princeig} below, we can prove that there cannot exist a principal eigenpair $(\lambda,\varphi)$ with $\lambda\leq \lambda_{crit}$. Hence $(\lambda_1^*, \varphi_{\lambda_1^*})$ is the unique principal eigenpair for the operator $A+B$. We have proved the following Theorem.
\begin{theorem}
	Let Assumption \ref{ASS-10} hold true and suppose moreover that $a(x)>0$ for all $x\in\mathbb{R}$. Then $A+B$ has a unique normalized principal eigenpair $\big(\lambda_1^*, \varphi_{\lambda_1^*}\big)$, we have $\lambda_1^*>\lambda_{crit}$, and $\lambda_1^*$ is the unique solution of the equation $\kappa(\lambda)=1$. 
\end{theorem}
By the change of variable $x\mapsto -x$, we obtain a similar statement if $a(x)<0$ for all $x\in\mathbb{R}$. We omit the details.

\subsection{The advection has two zeros of order 1}
This section contains the core of our analysis. We make the following assumption. 
\begin{assumption}\label{ASS-20}
	The functions $a(x)$,  $r(x)$ and $K(x, y)$ satisfy assumption \ref{ASS-10}. There exists $0<x_0<x_1<1$ such that 
\begin{equation*}
	a(x_0)=a(x_1)=0,\, a(x)>0\text{ if } x\in (x_0, x_1)\text{ and } a(x)<0 \text{ if } x\in(x_1, 1+x_0), 
\end{equation*}
and $r$ is once differentiable  and $ a $ is twice differentiable at $x_0$ and $x_1$, i.e., $r'(x)$, $a'(x)$ and $a''(x)$ exist for $x\in \{x_0, x_1\}$. Finally, 
	\begin{equation*}
		a'(x_0)>0,\;\;a'(x_1)<0.
	\end{equation*}
\end{assumption}
\noindent We will frequently use the notations $\lambda_1^0:= r(x_0)-a'(x_0)$,  $\lambda_1^1:=r(x_1)-a'(x_1)$, $\lambda_1^m:=\max(\lambda_1^0, \lambda_1^1)$ and $\lambda_1^{ext}:=\max(r(x_1), \lambda_1^0)$.

First we need to study the resolvent of $A$. 
\begin{proposition}[The resolvent of $A$]\label{prop:resolvent}
	Let {$\lambda_1^m=\max\big(r(x_0)-a'(x_0), r(x_1)-a'(x_1)\big)$}. Then the resolvent set of $A$ contains the ray $(\lambda_1^m, +\infty)$, and $\lambda_1^m$ is the minimal number to have this property: the spectral bound of $A$, $s(A)$,  corresponds to $\lambda_1^m$. 
	\begin{equation*}
		\rho(A)\supset (\lambda_1^m, +\infty) \text{ and } s(A)=\lambda_1^m. 
	\end{equation*}
\end{proposition}
\noindent The proof of Proposition \ref{prop:resolvent} will be given in Proposition \ref{prop:resolvent-circle}, as a consequence of the analysis in subsection \ref{sec:inversion-hyperbolic}. \medskip

It turns out that we can use the regularizing character of $B$ to extend the definition of $T(\lambda)$ outside of its ``natural'' domain of existence. 
To that aim we arbitrarily fix $X_0\in (x_0, x_1)$ and $ X_0'\in (x_1, 1+x_0)$ and  define  
\begin{equation}\label{eq:E_lambdaS1}
	\mathcal{E}_\lambda(x)= \mathcal{E}_\lambda^{X_0, X_0'}(x):=
	\begin{cases}
		\exp\left(\int_{X_0}^x \dfrac{r(z)- a'(z)-\lambda}{a(z)}\dd z\right),  &\text{ if } x\in (x_0, x_1), \vspace{3pt}\\ 
		\exp\left(\int_{X_0'}^x \dfrac{r(z)-a'(z)-\lambda}{a(z)}\dd z\right),  &\text{ if } x\in (x_1,1+ x_0). 
	\end{cases}
\end{equation}
When there are no ambiguities we will use the notation $\mathcal{E}_\lambda(x)$ instead of $\mathcal{E}_\lambda^{X_0, X_0'}(x)$.
\begin{proposition}[The operator $T(\lambda)$ and the function $\kappa(\lambda)$]\label{prop:Tlambda}
	Let $\lambda_1^{ext}=\max\big(r(x_1), r(x_0)-a'(x_0)\big)$. For any $\lambda\in (\lambda_1^{ext}, +\infty)$, we define the operator
	\begin{equation*}\label{eq:Tlambda}
		T(\lambda) \varphi(x):= \dfrac{\displaystyle -\int_{x_1}^x (K\star \varphi)(y) \mathcal{E}_\lambda(y)\dd y}{\displaystyle a(x) \mathcal{E}_\lambda(x)}, \text{ for any } x\not\in \{x_0, x_1\} .
	\end{equation*} 
	Then $T(\lambda)$ maps $L^1_{per}$ into $D(A)\subset L^1_{per}$, and satisfies $(\lambda I-A)T(\lambda) = B$. Moreover, $T(\lambda)$ is nonnegative, bounded, compact and irreducible. 

	Let $\kappa(\lambda):=r\big(T(\lambda)\big)$, the spectral radius of $T$. Then $\kappa(\lambda)$ is continuous, strictly decreasing and superconvex in $\lambda$; we have
	\begin{equation*}\label{eq:limits_kappa}
		\lim_{\lambda\to (\lambda_1^{ext})^+}\kappa(\lambda)=+\infty \text{ and } \lim_{\lambda\to +\infty}\kappa(\lambda)=0.
	\end{equation*}
	Finally, we have $\ker\big(\kappa(\lambda)I-T(\lambda)\big)=\mathrm{span}(\varphi_1^*)$, where $\varphi_1^*$ is a strictly positive vector satisfying $\int\varphi_1^*=1$.
\end{proposition}
\noindent The proof of Proposition \ref{prop:Tlambda} will be given in Proposition \ref{prop:eigen-0}.\medskip

Our first result on the existence of a principal eigenpair is a corollary of Proposition \ref{prop:Tlambda}. Indeed, by the continuity and strict monotonicity of $\kappa(\lambda)$, there always exists a unique solution of the equation $\kappa(\lambda)=1$, which is associated with a positive eigenvector. 
\begin{corollary}[Existence of a principal eigenpair]
	Let  $\lambda_1^*$ be the unique solution of the equation $\kappa(\lambda)=1$, and $\varphi_1^*\in D(A)$ be the associated eigenvector of $T(\lambda_1^*)$ satisfying $\int_{\mathbb{S}^1}\varphi_1^*=1$. Then $(\lambda_1^*, \varphi_1^*)$ is a normalized principal eigenpair of $A+B$.
\end{corollary}
\begin{definition}\label{def:minimal-eigenpair}
	We call $\lambda_1^*$ the \textit{minimal principal eigenvalue} and $(\lambda_1^*, \varphi_1^*)$ the \textit{minimal principal eigenpair} of $A+B$.
\end{definition}
But there is more: in some situations, $(\lambda_1^*, \varphi_1^*)$ is the unique principal eigenvalue of $A+B$, while in some other cases there exist many more (in fact, a continuum). Actually, we can completely characterize the solutions to the principal eigenproblem depending on the relative position of $\lambda_1^*$ compared to $\lambda_1^1:=r(x_1)-a'(x_1)$.
\begin{theorem}[Existence and multiplicity of eigenpairs]\label{thm:existence}
	Let $\lambda_1^1=r(x_1)-a'(x_1)$, and $(\lambda_1^*, \varphi_1^*)$ be the minimal principal eigenpair of $A+B$. The following alternative holds. 
	\begin{enumerate}
		\item If $\lambda_1^*\geq  \lambda_1^1$, then $(\lambda_1^*, \varphi_1^*)$ is the unique normalized principal eigenpair of $A+B$.
		\item If $\lambda_1^*<\lambda_1^1$, then for each $\lambda\in (\lambda_1^*, \lambda_1^1)$,  there exists   $\overline{\gamma}(\lambda)>0$ and a continuous family of positive vectors $\varphi_1^{\lambda, \gamma}\in D(A)$, indexed by $\gamma\in[0, \overline{\gamma}(\lambda)]$, such that $(\lambda, \varphi_1^{\lambda, \gamma})$ is a normalized principal eigenpair of $A+B$. 
			The family  $\big(\varphi_1^{\lambda, \gamma}\big)$ is composed of distinct eigenvectors: for $(\lambda, \gamma)\neq (\lambda', \gamma')$ we have $\varphi_1^{\lambda, \gamma}\neq \varphi_1^{\lambda', \gamma'}$.
			Finally, we have
			\begin{equation*}
				\overline{\gamma}(\lambda)\xrightarrow[\lambda\to (\lambda_1^*)^+]{} 0 \text{ and } \overline{\gamma}(\lambda)\xrightarrow[\lambda\to (\lambda_1^1)^-]{} 0. 
			\end{equation*}
	\end{enumerate}
\end{theorem}
\noindent Theorem \ref{thm:existence} is a direct consequence of Theorem \ref{thm:eigenpairs-L1} in subsection \ref{subsec:construction}. \medskip

The family of eigenpairs $(\lambda, \varphi_1^{\lambda, \gamma})$ can actually be characterized as the principal eigenpairs of a family of positive operators, that are constructed in the proof. Finally, we can show that no other solutions exist, except a singular eigenvalue when $\lambda_1^*<\lambda_1^1$, that we identify.
\begin{theorem}[Characterization of principal eigenpairs]\label{thm:characterization}
	Let $(\lambda, \varphi)\in \mathbb{R}\times \mathcal{M}_{per, +}$ be a solution of \eqref{eq:main} in the sense of distributions, where $\mathcal{M}_{per,+}$ is the set of nonnegative periodic Radon measures. Suppose that $\int_{\mathbb{S}^1} \varphi=1$. The following alternative holds. 
	\begin{enumerate}
		\item If $\lambda_1^*\geq \lambda_1^1$, then $\lambda=\lambda_1^*$ and $\varphi=\varphi_1^*$. In other words, $(\lambda, \varphi)$ is the principal eigenpair given by Theorem \ref{thm:existence}.
		\item If $\lambda_1^*<\lambda_1^1$, then either $\lambda\in [\lambda_1^*, \lambda_1^1)$ and in this case $\varphi=\varphi_1^{\lambda, \gamma}$ for some $\gamma\in [0, \overline{\gamma}(\lambda)]$;  or, $\lambda=\lambda_1^1$, and in that case we have $\varphi=\varphi_1^s=:\varphi_{ac}+\varphi_1\delta_{x_1}$, where $(\varphi_{ac}, \varphi_1)\in D(A)\times \mathbb{R}$ solves the problem
			\begin{equation}\label{eq:singular-eigenvect}
				a(x)\varphi_{ac}'+r(x)\varphi_{ac}+K\star \varphi_{ac} - \lambda_1^1 \varphi_{ac} = -\varphi_1 K(x, x_1), \text{ and }{\varphi_1} + \int \varphi_{ac}(y)\dd y = 1.
			\end{equation}
			Conversely, \eqref{eq:singular-eigenvect} has a unique positive solution, and $(\lambda_1^1, \varphi_1^s)$ with $\varphi_1^s=\varphi_{ac}+\varphi_1\delta_{x_1}$ is a solution of \eqref{eq:main} in the sense of distributions.
	\end{enumerate}
\end{theorem}
\noindent Theorem \ref{thm:characterization} follows from Theorem \ref{thm:eigenpair-singular} proved in subsection \ref{subsec:construction} below.

\subsection{The advection has a constant sign and a zero of order two}

In this section we consider the possibility of a single zero of order two. 
Our assumption is as follows. Up to shifting the origin, we assume that the unique zero of $a(x)$ is at $x=0$.
\begin{assumption}\label{ASS-30}
	The functions $a(x)$,  $r(x)$ and $K(x, y)$ satisfy assumption \ref{ASS-10}. We assume that $a(x)>0$ for all $x\in(0, 1)$ and that $a(0)=0$. Moreover, 
	$r$ is twice differentiable  and $ a $ is four times differentiable at $x=0$, i.e., $r'(0)$, $r''(0)$, $a'(0)$,  $a''(0)$, $a'''(0)$ and  $a^{(4)}(0)$ exist. {Finally, $x=0$ is a zero of order two of $a(x)$, meaning that $a''(0)>0$.}
\end{assumption}
Our argument is less complex than in the case of two zeros of order one, so we present it here. We compute the resolvent of $A$ as follows: 
\begin{equation*}
	(\lambda I-A)\varphi=f\Leftrightarrow a(x)\varphi'+(r(x)-\lambda)\varphi=-f(x) 
\end{equation*}
so 
\begin{equation}\label{eq:varphi-deg2}
	\varphi(x) = \dfrac{-\int_{X_1}^x f(y) \mathcal{E}_\lambda(y)\dd y + a(X_1)\varphi(X_1)\mathcal{E}_\lambda(X_1)}{a(x)\mathcal{E}_\lambda(x)}, 
\end{equation}
where   $\mathcal{E}_\lambda$ is defined by 
\begin{equation*}\label{eq:Elambda-deg2}
	\mathcal{E}_\lambda(x)= \exp\left(\int_{X_0}^x\dfrac{\sigma(x)-\lambda}{a(z)}\dd x\right), \text{ for all } x\in(0, 1), 
\end{equation*}
and $X_0, X_1\in [0, 1)$ are two constants that are chosen arbitrarily. We use a Taylor expansion of the integrand to find out the local behavior of $\mathcal{E}_\lambda(x)$ for $x$ close to $0$:
\begin{align}
	\nonumber\dfrac{\sigma(x)-\lambda}{a(x)}&=\frac{\sigma(0)-\lambda+\sigma'(0)x+o(x)}{\frac{1}{2}a''(0)x^2+\frac{1}{6} a'''(0)x^3+o(x^3)} = \frac{p^0_\lambda}{x^2}+\frac{p^1_\lambda}{x}+p^2_{\lambda}+o(1),
\end{align} 
where
\begin{gather*}
	p^0_\lambda:= \frac{2(r(0)-\lambda)}{a''(0)}, \qquad p^1_\lambda := \frac{2}{a''(0)}\left(r'(0)-{a''(0)}-\big(r(0)-\lambda\big)\frac{a'''(0)}{3a''(0)}\right), \\ 
	p^2_\lambda:=\frac{2}{a''(0)}\left[\frac{r''(0)-{a'''(0)}}{2}-(r'(0)-a'(0))\frac{a'''(0)}{3a''(0)}+\big(r(0)-\lambda\big)\left(-\frac{a^{(4)}(0)}{12a''(0)}+\frac{a'''(0)^2}{9a''(0)^2}\right)\right].
\end{gather*}
Integrating between $X_0$ and $x$, we obtain
\begin{equation*}
	\int_{X_0}^x \dfrac{\sigma(z)-\lambda}{a(z)}\dd z = -\frac{p^0_\lambda}{x}+p^1_\lambda\ln x + \mathcal{O}(1), 
\end{equation*}
and finally
\begin{equation*}\label{eq:260514a}
	\mathcal{E}_\lambda(x) = |x|^{p^1_{\lambda}} e^{-\frac{p^0_\lambda}{x} +\mathcal{O}(1)}.
\end{equation*}
We will now determine the value of $\varphi(X_1)$, when possible, to complete the computation of the resolvent $R(\lambda;A)$. We distinguish two cases. \medskip

\noindent$\bullet$ \textbf{Case 1}: $\lambda>\sigma(0)$. In this case, $\mathcal{E}_\lambda(x)$ converges to $0$ when $x\to 1^-$ (using the periodicity of the problem), so the numerator in \eqref{eq:varphi-deg2} has to converge to $0$ when $x\to 1^-$ (for otherwise $\varphi$ could not belong to $C^0_{per}$ nor even $L^1_{per}$). So 
\begin{equation*}
	a(X_1)\varphi(X_1)\mathcal{E}_{\lambda}(X_1) =  \int_{X_1}^1f(y) \mathcal{E}_\lambda(y)\dd y,
\end{equation*}
and we conclude 
\begin{equation}\label{eq:260512d}
	\varphi(x) = \dfrac{\int_{x}^1 f(y) \mathcal{E}_\lambda(y)\dd y }{a(x)\mathcal{E}_\lambda(x)}. 
\end{equation}
Let us prove that $\varphi\in L^1_{per}$. We have:
\begin{align}\label{eq:260512c}
	\int_0^1 |\varphi(x)|\dd x &\leq  \int_0^1 \dfrac{\int_x^1 |f(y)|\mathcal{E}_\lambda(y)\dd y}{a(x)\mathcal{E}_\lambda(x)} \dd x = \int_0^1|f(y)|\mathcal{E}_\lambda(y)\int_0^y \frac{1}{a(x)\mathcal{E}_\lambda(x)}\dd x\dd y.
\end{align}
Now, 
\begin{equation*}
	\mathcal{E}_\lambda(y)\int_0^y \frac{1}{a(x)\mathcal{E}_\lambda(x)}\dd x = \dfrac{\int_0^y \frac{e^{-\int_{X_0}^x \frac{\sigma(x)-\lambda}{a(z)}\dd z}}{a(y)} \dd x}{e^{-\int_{X_0}^y \frac{\sigma(z)-\lambda}{a(z)}\dd z}}, 
\end{equation*}
and since both the numerator and denominator converge to $0$ when $y\to 0$, we have by L'Hospital rule 
\begin{equation*}
	\lim_{y\to 0^+} \mathcal{E}_\lambda(y)\int_0^y \frac{1}{a(x)\mathcal{E}_\lambda(x)}\dd x  = \lim_{y\to 0^+} \dfrac{\frac{e^{-\int_{X_0}^y \frac{\sigma(z)-\lambda}{a(z)}\dd z}}{a(y)}}{-\frac{\sigma(y)-\lambda}{a(y)}e^{-\int_{X_0}^y \frac{\sigma(z)-\lambda}{a(z)}\dd z}} = \frac{1}{\lambda-\sigma(0)}.
\end{equation*}
This proves that $y\mapsto \mathcal{E}_\lambda(y)\int_0^y \frac{1}{a(x)\mathcal{E}_\lambda(x)}\dd x$ is a bounded function, hence the right-hand side in \eqref{eq:260512c} is finite, and $\varphi$ is in $L^1_{per}$ with a norm controlled by $\Vert f\Vert_{L^1_{per}}$ (up to a constant). Therefore $\lambda I-A $ admits a bounded inverse $R(\lambda; A)f =\varphi$. Moreover, it follows from \eqref{eq:260512d} that $R(\lambda; A)$, in that case, is a positive operator. 
\medskip

\noindent$\bullet$ \textbf{Case 2}: $\lambda>\sigma(0)$. In this case, $\mathcal{E}_\lambda(x)$ converges to $0$ when $x\to 0^+$, so the numerator in \eqref{eq:varphi-deg2} has to converge to $0$ when $x\to 0^+$. So 
\begin{equation*}
	a(X_1)\varphi(X_1)\mathcal{E}_{\lambda}(X_1) =  \int_{X_1}^0f(y) \mathcal{E}_\lambda(y)\dd y,
\end{equation*}
and we conclude 
\begin{equation}\label{eq:260512e}
	\varphi(x) = \dfrac{-\int_{0}^x f(y) \mathcal{E}_\lambda(y)\dd y }{a(x)\mathcal{E}_\lambda(x)}. 
\end{equation}
Similar computations as in Case 1 show that $R(\lambda; A)f=\varphi$ is a bounded operator; however, from \eqref{eq:260512e}, it is a \textit{negative} operator. 

\medskip

As before, we define the operator
\begin{equation*}
	T(\lambda) \varphi_\lambda:= (\lambda I-A)^{-1}B\varphi. 
\end{equation*}
This operator is well-defined, positive, irreducible and compact for all $\lambda>r(0)$. Hence there exists a unique normalized principal eigenpair $(\kappa(\lambda), \varphi_{\lambda})$, such that 
\begin{equation*}
	T(\lambda)\varphi_{\lambda} = \kappa(\lambda)\varphi_{\lambda}. 
\end{equation*}
The existence of a principal eigenpair is related to the existence of a value $\lambda\geq r(0)$ such that $\kappa(\lambda)=1$. We have:
\begin{equation*}
	\kappa(\lambda)=\int_{\mathbb{S}^1} T(\lambda)\varphi = \int_0^1 \dfrac{\int_x^1 K\star \varphi (y)\mathcal{E}_\lambda(y)\dd y}{a(x)\mathcal{E}_\lambda(x)} \dd x \begin{cases} 
		\leq \displaystyle \sup_{\bar x,\bar y} K(\bar x, \bar y) \int_0^1 \dfrac{\int_x^1 \mathcal{E}_\lambda(y)\dd y}{a(x)\mathcal{E}_\lambda(x)} \dd x, \\ 
		\geq \displaystyle \inf_{\bar x, \bar y}  K(\bar x, \bar y) \int_0^1 \dfrac{\int_x^1 \mathcal{E}_\lambda(y)\dd y}{a(x)\mathcal{E}_\lambda(x)} \dd x .
	\end{cases}
\end{equation*}
Hence the finiteness of $\kappa(\lambda)$ as $\lambda\to r(0)$ is strongly tied to the behavior of the integral 
\begin{equation*}
	\mathcal{I}(\lambda):=\int_0^1 \dfrac{\int_x^1 \mathcal{E}_\lambda(y)\dd y}{a(x)\mathcal{E}_\lambda(x)} \dd x, 
\end{equation*}
as $\lambda\to r(0)^+$. Let $I_\lambda(x):= \dfrac{\int_x^1 \mathcal{E}_\lambda(y)\dd y}{a(x)\mathcal{E}_\lambda(x)} $, then as $\lambda\to r(0)^+$ the function $I_\lambda(x)$ converges to 
\begin{align*}
	I_{r(0)}(x)&=\dfrac{\int_x^1 \mathcal{E}_{r(0)}(y)\dd y}{a(x)\mathcal{E}_{r(0)}(x)} = \dfrac{\int_x^1 \mathcal{E}_{r(0)}(y)\dd y}{x^{2+p^1_{\lambda}} }e^{\mathcal{O}(1)}, 
\end{align*}
locally uniformly in $x\in (0, 1]$, where $p^1:=p^1_{r(0)}=2\frac{r'(0)}{a''(0)}$. If $2+p^1\geq 1$, then $I_{r(0)}(x)$ is not integrable and $\mathcal{I}(\lambda)\to +\infty$. If $2+p^1<1$, then we have
\begin{align*}
	I_{r(0)}(x)&= \dfrac{\int_x^1 \mathcal{E}_{r(0)}(y)\dd y}{x^{2+p^1_{\lambda}} }e^{\mathcal{O}(1)}\geq \dfrac{\int_x^1 y^{p^1}-1\dd y}{x^{2+p^1_{\lambda}} }e^{\mathcal{O}(1)} = \dfrac{\frac{1}{p^1+1}(x^{p^1+1}-1)-(x-1)}{x^{2+p^1_{\lambda}} }e^{\mathcal{O}(1)} \\ 
	&= \left(\frac{1}{p^1+1}\,\frac{1}{x} -\frac{1}{x^{1+p^1}} + 1-\frac{1}{p^1+1}\right)e^{\mathcal{O}(1)}.
\end{align*}
In that case, we also find that $I_\lambda(x)\not\in L^1$ and therefore $\mathcal{I}(\lambda)\to +\infty$ as $\lambda\to r(0)^+$.  Thus, in any case we obtain $\kappa(\lambda)\to +\infty$ as $\lambda\to r(0)^+$.
\medskip

Next we  prove that $\kappa(\lambda)\to 0$ as $\lambda\to+\infty$. We establish the local behavior of $I_\lambda(x)$ in the vicinity of $x=0$. We select $\lambda>S:=\sup_{z\in[0, 1]}\sigma(z)$  and $\alpha>0$ such that $a(x)\leq \alpha x^2$ for all $x\in [0, 1]$. We have
\begin{align*}
	I_\lambda(x)&= \frac{1}{a(x)} \int_{x}^1 e^{\int_{x}^y \frac{\sigma(z)-\lambda}{a(z)}\dd z}\dd y\leq \frac{1}{a(x)}\int_x^1 e^{\int_x^y\frac{S-\lambda}{\alpha z^2}\dd z}\dd y = \frac{1}{a(x)}\int_x^1 e^{-\frac{S-\lambda}{\alpha y}+\frac{S-\lambda}{\alpha x}}\dd y  \\
	&=\frac{1}{a(x)} e^{\frac{S-\lambda}{\alpha x}}\int_x^1 e^{-\frac{S-\lambda}{\alpha y}}\dd y.
\end{align*}
We will use the formula $\int \frac{1}{\ln^2 t} = \mathrm{Li}(t)-\frac{t}{\ln t}$ to estimate the last integral, where $\mathrm{Li}(x)$ is the Eulerian logarithmic integral $\mathrm{Li}(t):= \int_2^t \frac{1}{\ln s}\dd s$ for $t>1$. We have, by the change of variables $z= e^{-\frac{S-\lambda}{\alpha y}}$, 
\begin{align*}
	\int_x^1 e^{-\frac{S-\lambda}{\alpha y}}\dd y&=\frac{S-\lambda}{\alpha}\int_{e^{-\frac{S-\lambda}{\alpha x}}}^{e^{-\frac{S-\lambda}{\alpha}}}z \frac{1}{z \ln^2(z)}\dd z = \frac{\lambda-S}{\alpha} \int_{e^{\frac{\lambda-S}{\alpha}}}^{e^{\frac{\lambda-S}{\alpha x}}} \frac{1}{ \ln^2(z)}\dd z \\ 
	&=\frac{\lambda-S}{\alpha}\left[\mathrm{Li}\left(e^{\frac{\lambda-S}{\alpha x}}\right)-\mathrm{Li}\left(e^{\frac{\lambda-S}{\alpha}}\right)-\frac{e^{\frac{\lambda-S}{\alpha x}}}{\ln e^{\frac{\lambda-S}{ \alpha x}}} + \frac{e^{\frac{\lambda-S}{\alpha}}}{\ln e^{\frac{\lambda-S}{\alpha}}}\right]\\
	&=\frac{\lambda-S}{\alpha}\left[\mathrm{Li}\left(e^{\frac{\lambda-S}{\alpha x}}\right)-\alpha x\frac{e^{\frac{\lambda-S}{\alpha x}}}{\lambda-S} -\mathrm{Li}\left(e^{\frac{\lambda-S}{\alpha}}\right)+ \alpha \frac{e^{\frac{\lambda-S}{\alpha}}}{\lambda-S}\right]. 
\end{align*}
Recalling the asymptotic expansion \cite[p.228]{Abramowitz-Stegun-1972} $\mathrm{Li}(x)=\frac{x}{\ln(x)}\left[1+\frac{1}{\ln x}+ \mathcal{O}\left(\frac{2}{\ln^2 x}\right)\right]$ we have %+\mathcal{O}\left(\frac{1}{\ln^3 x}\right)\right] $ we have
\begin{align*}
	I_\lambda(x)&\leq \frac{1}{a(x)}\left[\frac{\lambda - S}{\alpha }e^{\frac{S-\lambda}{\alpha x}} \mathrm{Li}\left(e^{-\frac{S-\lambda}{\alpha x}}\right) -  x-\frac{S-\lambda}{\alpha} e^{\frac{S-\lambda}{\alpha x}}\mathrm{Li}\left(e^{-\frac{S-\lambda}{\alpha}}\right) + e^{\frac{S-\lambda}{\alpha}\left(\frac{1}{x}-1\right)}\right] \\ 
	&=\frac{1}{a(x)}\left[ \frac{\alpha}{\lambda-S} x^2 +\mathcal{O}\left(\frac{\alpha^2}{(\lambda-S)^2} x^3\right) - e^{\frac{S-\lambda}{\alpha } \left(\frac{1}{x}-1\right)}\left(\frac{\alpha^2}{\big(S-\lambda)^2} +\mathcal{O}\left(\frac{\alpha^3}{\big(S-\lambda)^3}\right)\right)\right].
\end{align*}
Since $\frac{x^2}{a(x)}$ is a bounded function and $ e^{\frac{S-\lambda}{\alpha } \left(\frac{1}{x}-1\right)} = o(x^2)$,  we conclude that $I_\lambda(x)$ converges uniformly to $0$ on $[0,1]$. This proves that $\mathcal{I}(\lambda)\to 0$, and therefore $\kappa(\lambda)\to 0$, as $\lambda\to+\infty$.
\medskip 

The proof of non-existence of a principal eigenpair $(\lambda, \varphi_\lambda)$ with $\lambda\leq r(0)$, as well as the non-existence of a singular principal eigenpair, follow the same step as in the case of two zeros of order one, hence we omit it. We have reached the following result. 
\begin{theorem}
	Let Assumption \ref{ASS-30} hold true. Then $A+B$ has a unique normalized principal eigenpair $\big(\lambda_1^*, \varphi_{\lambda_1^*}\big)$, we have $\lambda_1^*>r(0)$, and $\lambda_1^*$ is the unique solution of the equation $\kappa(\lambda)=1$. 
\end{theorem}

%\subsection{How to deal with multiple zeros of arbitrary order}

\section{Properties of the semigroup}
\label{sec:generation}
\subsection{Preliminaries}

In this section we recall important properties of semigroups, including positive semigroups and some elements of their spectral theory, that will be used in the sequel.

Let $(X,\|\cdot\|)$ be a Banach space and $A:D(A)\subset X\to X$ a closed linear operator, the generator of a strongly continuous semigroup on $X$, denoted by $\{T_A(t)\}_{t\geq 0}$.\\
We define below various important spectral quantities.
\begin{definition}
We denote by $\sigma(A)$ and $\rho(A)$,the spectrum and the resolvent set of $A$, respectively.\\
We also define the spectral bound of $A$ by:
\begin{equation*}
s(A)=\sup\left\{{\rm Re}\,\lambda,\;\lambda\in\sigma (A)\right\}.
\end{equation*}
We also define the growth rate of the semigroup $\{T_A(t)\}_{t\geq 0}$ by
\begin{equation*}
\omega_0(A)=\lim_{t\to\infty}\frac{1}{t}\ln \|T_A(t)\|_{\mathcal L(X)}\in [-\infty,\infty).
\end{equation*}
\end{definition}
To go further and connect the above quantities, we introduce the essential norm and the essential growth rate as follows.
We denote by $\kappa$ the Kuratovski measure of non compactness in $X$. Next if $L\in \mathcal L(X)$ is a given bounded linear operator, its essential norm is defined by
\begin{equation*}
\|L\|_{\rm ess}=\kappa\left(B_X(0,1)\right),
\end{equation*}
where $B_X(0,1)$ denotes the unit ball in $X$ centered at the origin. 
Using this definition we define the essential growth rate of the semigroup $\{T_A(t)\}_{t\geq 0}$ as follows.	
\begin{definition}
We define the essential growth rate of the semigroup $\{T_A(t)\}_{t\geq 0}$ by
\begin{equation*}
\omega_{0,ess}(A)=\lim_{t\to\infty}\frac{1}{t}\ln \|T_A(t)\|_{\rm ess}\in [-\infty,\omega_0(A)].
\end{equation*}
\end{definition}

Next according Webb \cite{Webb-1987}, Engel and Nagel \cite{Engel-Nagel-2000}, Magal and Ruan \cite{Magal-Ruan-2018}, one has
\begin{lemma}\label{LE1}
	Suppose that  $\omega_{0,ess}(A)<\omega_0(A)$. Then $\omega_0(A)=s(A)$. Moreover, 
for each $\eta>\omega_{0,ess}(A)$, define 
$$
\Sigma_\eta:=\{\lambda\in \sigma(A):\;{\rm Re}\,(\lambda)\geq \eta\}.
$$
Then for all $\eta>\omega_{0,ess}(A)$ such that $\Sigma_\eta$ is finite or empty, then each $\lambda_0\in \Sigma_\eta$ is an isolated eigenvalue of $A$ and a pole of the resolvent $\lambda\mapsto (\lambda-A)^{-1}$ with finite algebraic multiplicity.
\end{lemma}

We then recall some spectral properties of positive semigroups on Banach lattices, typically in $L^1-$ or $C^0-$spaces.
From now on in this section $(X,\|\cdot\|)$ denotes a Banach lattice with partial order denoted by $\leq$. We also denote by $X^+$ its positive cone, that is
$$
X^+=\left\{x\in X:\;0\leq x\right\}.
$$
Then the first result we aim at recalling reads as follows.
\begin{lemma}\label{LE2}
Let $A:D(A)\subset X\to X$ be the infinitesimal generator of a strongly continuous semigroup $\{T_A(t)\}_{t\geq 0}$ on the Banach lattice $X$. We assume that $T_A$ is positive, in the sense that $T_A(t)X^+\subset X^+$ for all $t\geq 0$, and that $s(A)>-\infty$. Then
\begin{equation*}
s(A)\in \sigma(A).
\end{equation*}
\end{lemma}

\begin{proposition}\label{PROP-eig}
Let $A:D(A)\subset X\to X$ be the infinitesimal generator of a strongly continuous semigroup $\{T_A(t)\}_{t\geq 0}$ on the Banach lattice $X$ such that $T_A$ is positive and irreducible. If $s(A)\in \sigma(A)$ is a pole of the resolvent, then there exists $y\in X^+$, which is strictly positive, such that
$$
\ker \left(A-s(A)I\right)={\rm span}(y).
$$ 
\end{proposition}
	
As a corrolary, coupling Lemma \ref{LE1}, Lemma \ref{LE2} and Proposition \ref{PROP-eig}, we have the following result.

\begin{corollary}\label{CORO-eig}
Let $A:D(A)\subset X\to X$ be the infinitesimal generator of a strongly continuous semigroup $\{T_A(t)\}_{t\geq 0}$ on the Banach lattice $X$ such that $T_A$ is positive and irreducible. 
Assume furthermore that $\omega_{0,ess}(A)<\omega_0(A)$. Then, there exists $y\in X^+$, which is strictly positive, such that
$$
Ay=s(A)y.
$$ 
\end{corollary}

\subsection{Generation and asymptotic regularity}

In this section we describe some basic properties of the semigroup generated by $A+B:D(A)\subset L^1_{per}(\R)\to L^1_{per}(\R)$; 
we will assume throughout this subsection that Assumption \ref{ASS-20} holds true. In particular, the advection has two zeros of order 1.
We collect the main results of this section in the next theorems.	
	
\begin{theorem}\label{THEO1}
	Let Assumption \ref{ASS-20} hold true.
The closed linear operator $A:D(A)\subset L^1_{per}(\R)\to L^1_{per}(\R)$ is the infinitesimal generator of a strongly continuous semigroup, denoted by $\{T_A(t)\}_{t\geq 0}$, that satisfies the following properties:
\begin{enumerate}
\item[(i)] The semigroup $\{T_A(t)\}_{t\geq 0}$ is positive, in the sense that $T_A(t)L^1_{per,+}(\R)\subset L^1_{per,+}(\R)$.  
\item[(ii)] The growth rate $\omega_0(A)$ of the semigroup $\{T_A(t)\}_{t\geq 0}$ satisfies
\begin{equation*}
\omega_0(A)\leq \max_{i=0,1} \left(r(x_i)-a'(x_i)\right) = \lambda_1^m.
\end{equation*}
\end{enumerate} 
\end{theorem}
	
\begin{proof}[Proof of Theorem \ref{THEO1} (i)]	
To prove the above theorem, we introduce the characteristic curves $\pi(t,s;x)$, for $x\in\R$ and $(t,s)\in\R^2$, as the solution of the following ODE
$$
\partial_t \pi(t,s;x)=-a(\pi(t,s;x)),\;\;\pi(s,s;x)=x.
$$
Recall that $a$ is $1$-periodic, so that for any $x\in \R$, $k\in\mathbb Z$ we have 
$$
\pi(t,s;x+k)=\pi(t,s;x)+k,\;\;\forall (t,s)\in\R^2.
$$
Moreover, since $a$ is of class $C^1$, so is $\pi$ with respect to the variable $x\in\R$. 
Using these characteristic curves, we can solve the Cauchy problem
$$
\frac{du(t)}{dt}=Au(t),\;u(0)=u_0\in L^1_{per}(\R),
$$
that reads as
$$
\begin{cases}
\partial_t u(t,x)=a(x)\partial_x u(t,x)+r(x)u(t,x),\\
u(t,0)=u(t,1),\\
u(0,.)=u_0\in L^1_{per}(\R).
\end{cases}
$$
Setting $v(t,x)=u(t,\pi(t,0;x))$ we obtain
$$
\partial_t v(t,x)=r(\pi(t,0;x))v(t,x),
$$
that is
$$
v(t,x)=u(t,\pi(t,0;x))=u_0(x)\exp\left[\int_0^t r(\pi(s,0;x)ds\right].
$$
	Equivalently, by using the change of variable $y=\pi(t,0;x)$ (which is equivalent to $x=\pi(0,t;y)$), we obtain
$$
u(t,y)=u_0(\pi(0,t;y))\exp\left[\int_0^t r(\pi(s,t;y)ds\right].
$$	
As a consequence of the above analysis, one may consider the family of linear operator $\{T_A(t)\}_{t\geq 0}\subset \mathcal L\left(L^1_{per}(\R)\right)$ given by
$$
T_A(t)\varphi\,(y)=u_0(\pi(0,t;y))\exp\left[\int_0^t r(\pi(s,t;y)ds\right],\:\:\forall t\geq 0,\;\forall \varphi\in L^1_{per}(\R).
$$
Then $\{T_A(t)\}_{t\geq 0}$ is a strongly continuous and positive semigroup on $L^1_{per}(\R)$. %{(\bf Also true in $C^0_{per}$!!)}\\
\end{proof}

	Now, in order to estimate the growth rate of $T_A$, we need to understand some averaging properties of the characteristic curve and we shall need the following lemma. {Note that this Lemma can be viewed as a simple example of subadditive ergodic theorem, see for instance \cite[Proposition 2.2]{Jenkinson-2019} or \cite{Morro-Sant'Anna-Varandas-2020}.}
\begin{lemma}\label{LE-averaging}
Recalling that Assumption \ref{ASS-20} is satisfied. Then for all continuous and $1-$periodic function $\Phi:\R\to \R$ one has
$$
\limsup_{t\to\infty} \sup_{x\in [0,1]} \frac{1}{t}\int_0^t \Phi(\pi(\ell,0;x))d\ell\leq \max\{\Phi(x_0),\Phi(x_1)\}.
$$
\end{lemma} 

%{\bf To be proved!!}\\

Before proving this lemma, let us complete the proof of Theorem \ref{THEO1} $(ii)$.
\begin{proof}[Proof of Theorem \ref{THEO1} (ii)]
Let $u_0\in L^1_{per}(\R)$ be given.
Then for any $t>0$ one has
$$
\|T_A(t)u_0\|_{L^1(0,1)}=\int_0^1 |u_0(\pi(0,t;y))|\exp\left[\int_0^t r(\pi(s,t;y)ds\right]dy.
$$	
Next we use the change of variable $z=\pi(0,t;y)$ or equivalently $y=\pi(t,0;z)$, hence $dy=\partial_x\pi(t,0;z)dz$.
%However one may compute	the derivative in $x$ by differentiating the characteristic curves with respect to $x$ as follows: 
	For all $(t,s)\in\R$ and $x\in\R$ we have
$$	
\partial_t \partial_x\pi(t,s;x)=-a'(\pi(t,s;x))\partial_x\pi(t,s;x),\;\;\partial_x\pi(s,s;x)=1,
$$
and integrating the above ODE, it rewrites as 
$$
\partial_x\pi(t,s;x)=\exp\left[-\int_s^t a'(\pi(\ell,s;x))d\ell\right],\;\forall (t,s)\in\R,\;\forall x\in\R.
$$
As a consequence we obtain
$$
\|T_A(t)u_0\|_{L^1(0,1)}=\int_0^1 |u_0(z)|\exp\left[\int_0^t \left(r-a'\right)(\pi(\ell,0;z)d\ell\right]dz
$$	
so that for any $t>0$
$$
\|T_A(t)\|_{\mathcal L\left(L^1_{\rm per}(\R)\right)}\leq \exp\left[\sup_{z\in [0,1]}\int_0^t \left(r-a'\right)(\pi(\ell,0;z)d\ell\right].
$$	
Since the function $r-a'$ is continuous and $1$-periodic, Lemma \ref{LE-averaging} applies and yields $(ii)$, which completes the proof of Theorem \ref{THEO1}. 
\end{proof}
	
It remains to prove Lemma \ref{LE-averaging}.
%{\bf REFs and proof}

\begin{proof}[Proof of  Lemma \ref{LE-averaging}]
Without loss of generality, suppose that $\Phi(x)>0$. Denote by $\tilde \Phi:=\max\limits_{x\in [0,1]}\Phi(x)$ and $\bar \Phi:= \max\{\Phi(x_0),\Phi(x_1)\}$. Since $\Phi$ is continuous, for any $\epsilon>0$, there is some neighborhood of $I_\epsilon$of $\{x_0,x_1\}$, such that for any $x\in I_\epsilon$, $\Phi(x)<\bar \Phi+\epsilon$. For any $x\in [0,1]$, denote by $T(x,\epsilon):=\{t\geq 0| \pi(t,0;x)\not\in I_\epsilon\}$. Then $t\geq 0$ and $t\not\in T(x,\epsilon)$ means that $\pi(t,0;x)\in I_\epsilon$. Moreover, $T(x,\epsilon)$ is a bounded interval (which may be empty) and there is some $L>0$ independent of $x$ such that the length of $T(x,\epsilon)$ is less than $L$, taking the convention that the length of an empty subset of the real line is zero. 

We have:
$$
\int_0^t \Phi(\pi(\ell,0;x))d\ell = \int_{T(x,\epsilon)} \Phi(\pi(\ell,0;x))d\ell +  \int_{[0,t]\setminus T(x,\epsilon)} \Phi(\pi(\ell,0;x))d\ell \leq L \tilde \Phi+t (\bar\Phi+\epsilon).
$$
And hence $$
\limsup_{t\to\infty} \sup_{x\in [0,1]} \frac{1}{t}\int_0^t \Phi(\pi(\ell,0;x))d\ell\leq \bar \Phi+\epsilon.
$$

Finally, since $\epsilon$ is arbitrary, the lemma is proved.
\end{proof}

\begin{theorem}\label{THEO2}
The closed linear operator $A+B:D(A)\subset L^1_{per}(\R)\to L^1_{per}(\R)$ is the infinitesimal generator of a strongly continuous semigroup, denoted by $\{T_{A+B}(t)\}_{t\geq 0}$, that satisfies the following properties:
\begin{enumerate}
\item[(i)] The semigroup $\{T_{A+B}(t)\}_{t\geq 0}$ is positive and irreducible on $L^1_{per,+}(\R)$, in the sense that for all 
$\varphi\in L^1_{per,+}(\R)\setminus\{0\}$ and $\varphi^*\in L^\infty_{per,+}(\R)\setminus\{0\}$ there exists $t>0$ such that
$$
\left\langle \varphi^*,T_{A+B}(t)\varphi\right\rangle>0.
$$
\item[(ii)] The essential growth rate of the semigroup $\{T_{A=B}(t)\}_{t\geq 0}$, denoted by $\omega_{0,ess}(A+B)$ and defined by
$$
\omega_{0,ess}(A+B)=\lim_{t\to\infty} \frac{1}{t}\ln \kappa\left(T_{A+B}(t)\right),
$$ 
where we recall that $\kappa\left(\cdot\right)$ denotes he Kuratovsky measure of non-compactness,
satisfies
\begin{equation*}
\omega_{0,ess}(A+B)\leq \omega_0(A)\leq \max_{i=0,1} \left(r(x_i)-a'(x_i)\right)=\lambda_1^m.
\end{equation*}
\end{enumerate} 
\end{theorem}

\begin{remark} It is possible to deduce from Theorem \ref{THEO2}  the existence of a spectral gap for $A+B$ in the special situation where $\lambda_1^*>\lambda_1^m $, where $\lambda_1^*$ is the minimal principal eigenvalue of $A+B$ defined in Definition  \ref{def:minimal-eigenpair}. We refer to Schaefer \cite[Proposition 5.6 p. 332]{Schaefer-1974} for a result on the uniqueness of the spectral radius as an eigenvalue in the peripheral spectrum of a strongly positive linear operator, and Webb \cite{Webb-1987} for the theory of the relationship between the growth rate and essential growth rates of semigroups.
\end{remark}

\section{Construction of principal eigenpairs}
\label{sec:construction}
\subsection{Inverting the hyperbolic part on the interval}
\label{sec:inversion-hyperbolic}

In this section  we restrict the analysis to the interval $(x_0, x_1)$. We will use the following set of assumptions. 
\begin{assumption}\label{as:interval}
	The function $a(x)$ is Lipschitz continuous  on $[x_0, x_1]$, positive on $(x_0, x_1)$, and twice differentiable at $x_0$ and $x_1$. The function $r(x)$ is continuous on $[x_0, x_1]$ and differentiable at $x_0$ and $x_1$. Moreover we assume that 
	\begin{equation*}
		a'(x_0)>0, \qquad a'(x_1)<0.
	\end{equation*}
\end{assumption}
	In this section we will use the slightly modified operator $A$ defined by 
\begin{equation}\label{defAint}
\begin{split}
	&D(A)=\left\{\varphi\in L^1\big([x_0, x_1]\big):\;a\varphi'\in L^1\big([x_0, x_1]\big)\right\},\\
&A\varphi=a(x)\varphi'+r(x)\varphi,\;\forall \varphi\in D(A).
\end{split}
\end{equation} 
and also define 
$A_0:D(A_0)\subset C^0\big([x_0, x_1]\big)\to C^0\big([x_0, x_1]\big)$ with
\begin{equation}\label{defA0int}
\begin{split}
    &D(A_0)=\left\{\varphi\in C^0\big([x_0, x_1]\big):\;a\varphi'\in C^0\big([x_0, x_1]\big)\right\},\\
&A_0\varphi=a(x)\varphi'+r(x)\varphi,\;\forall \varphi\in D(A_0).
\end{split}
\end{equation} 
In this subsection we focus on the invertibility of the hyperbolic part, namely, we are interesting in solving the equation 
\begin{equation}\label{eq:resolvent-A}
	(\lambda I-A)\varphi = f
\end{equation}
with $f\in L^1\big([x_0, x_1]\big)$ or $f\in C^0\big([x_0, x_1]\big)$ and $\lambda\in\mathbb{R}$. We introduce the function $\sigma(x):=r(x)-a'(x)$. 
If $\varphi$ is a solution of \eqref{eq:resolvent-A}, we have
\begin{gather*}
	\big(a(x)\varphi\big)' + \big(\sigma(x)-\lambda\big)\varphi = -f(x), \\ 
	\left(a(x)e^{\int_{X_0}^x \frac{\sigma(z)-\lambda}{a(z)}\dd z}\varphi\right)' e^{-\int_{X_0}^x \frac{\sigma(z)-\lambda}{a(z)}\dd z} = -f(x), \\ 
	a(x)e^{\int_{X_0}^x \frac{\sigma(z)-\lambda}{a(z)}\dd z}\varphi(x) - a(X_1)e^{\int_{X_0}^{X_1} \frac{\sigma(z)-\lambda}{a(z)}\dd z}\varphi(X_1) = -\int_{X_1}^x f(y)e^{\int_{X_0}^y \frac{\sigma(z)-\lambda}{a(z)}\dd z}\dd y,
\end{gather*}
where $X_0\in (x_0, x_1)$ and $X_1\in (x_0, x_1)$ are arbitrarily chosen.  We obtain the formula
\begin{equation}\label{eq:resolvent-A-formula-0}
    \varphi(x) =  
	\dfrac{- \int_{X_1}^x f(y)e^{\int_{X_0}^y \frac{\sigma(z)-\lambda}{a(z)}\dd z}\dd y +a(X_1)e^{\int_{X_0}^{X_1} \frac{\sigma(z)-\lambda}{a(z)}\dd z}\varphi(X_1)  }{a(x)e^{\int_{X_0}^x \frac{\sigma(z)-\lambda}{a(z)}\dd z}} , \text{ for all } x\in(x_0, x_1).
\end{equation}	
In the formula \eqref{eq:resolvent-A-formula-0}, there are two possible singularities (at $x=x_0$ and $x=x_1$) and one degree of liberty, the value of $\varphi(X_1)$.
\medskip 

We start by a lemma concerning the function 
\begin{equation}\label{eq:E_lambda}
	\mathcal{E}_\lambda^{X_0}(x):=\exp\left(\int_{X_0}^x \dfrac{\sigma(z)-\lambda}{a(z)}\dd z\right), 
\end{equation}
which is of particular importance. When the context is clear, we will frequently write $\mathcal{E}_\lambda$ instead of $\mathcal{E}_\lambda^{X_0}$, for simplicity. We also define 
\begin{equation}\label{eq:alphas}
	\alpha_0(\lambda):= \dfrac{\sigma(x_0)-\lambda}{a'(x_0)} \text{ and } \alpha_1(\lambda):= \dfrac{\sigma(x_1)-\lambda}{a'(x_1)}. 
\end{equation}
Again, we will frequently write $\alpha_0$ and $\alpha_1$, omitting the parameter $\lambda$ for simplicity, when the context is clear. 
\begin{lemma}\label{lem:E_lambda}
	Let Assumption \ref{as:interval} hold true.
	Let $X_0\in (x_0, x_1)$ be fixed and $\lambda\in\mathbb{R}$. Let  $\mathcal{E}_\lambda(x)=\mathcal{E}_\lambda^{X_0}(x)$ be the function defined in \eqref{eq:E_lambda}.  Then 
	for each $\Lambda>0$ there exists a constant $C_\Lambda>1$ such that 
	\begin{subequations}\label{eq:251204d}
		\begin{equation}\label{eq:251204da}
			\frac{1}{C_\Lambda}(x_1-x)^{\alpha_1}\leq \mathcal{E}_\lambda(x) \leq C_\Lambda (x_1-x)^{\alpha_1}, \text{ as } x\to x_1^-, 
		\end{equation}
		and 
		\begin{equation}\label{eq:251204db}
			\frac{1}{C_\Lambda}(x-x_0)^{\alpha_0}\leq \mathcal{E}_\lambda(x) \leq C_\Lambda (x-x_0)^{\alpha_0} , \text{ as } x\to x_0^+, 
		\end{equation}
	\end{subequations}
	for all  $\lambda\in[-\Lambda, \Lambda] $.
\end{lemma}
We will often use the notation $\mathcal{E}_\lambda(x) = (x-x_1)^{\alpha_1}e^{\mathcal{O}(1)}$ in place of the two inequalities \eqref{eq:251204da}, and similarly $\mathcal{E}_\lambda(x) = (x-x_0)^{\alpha_0}e^{\mathcal{O}(1)}$ in place of \eqref{eq:251204db}. More generally, whenever the context is clear, the notation $e^{\mathcal{O}(1)}$ will be used in place of a positive function that is uniformly bounded above and below.
\begin{proof}[Proof of Lemma \ref{lem:E_lambda}]
We have, as $x\to x_1$:
	\begin{align*}
		\dfrac{\sigma(x)-\lambda}{a(x)} &= \dfrac{\sigma(x_1)-\lambda+\sigma'(x_1)(x-x_1)+o(x-x_1)}{a'(x_1)(x-x_1)\big[1+\frac{a''(x_1)}{2a'(x_1)}(x-x_1)+o(x-x_1)\big]} \\
		&= \dfrac{\sigma(x_1)-\lambda}{a'(x_1) (x-x_1)} + \frac{\sigma'(x_1)}{a'(x_1)}-\dfrac{\sigma(x_1)-\lambda}{a'(x_1)}\dfrac{a''(x_1)}{2a'(x_1)} + o(1) \\ 
		&=:\dfrac{\alpha_1}{x-x_1}+\mathcal{O}(1 ).
	\end{align*}
Then,
	\begin{equation*}
		e^{\int_{X_0}^x\frac{\sigma(z)-\lambda}{a(z)}\dd z} = e^{\alpha_1(\ln(x_1-x)-{\ln(x_1-X_0)})+\mathcal{O}(x-x_1)} = (x_1-x)^{\alpha_1} e^{\mathcal{O}(1)}. 
	\end{equation*}
	A similar analysis when $x\to x_0$ proves \eqref{eq:251204d}.
\end{proof}

Our first result is as follows. 
\begin{proposition}\label{prop:resolvent-A-L1}
	Let  Assumption \ref{as:interval} hold and $A$ be the operator defined in \eqref{defAint}. Recall that $\lambda_1^0=r(x_0)-a'(x_0)$, $\lambda_1^1 = r(x_1)-a'(x_1)$, and $\lambda_1^{m}=\max\big(\lambda_1^0, \lambda_1^1\big)$.
	\begin{enumerate}
		\item The resolvent set of $A$, $\rho(A)$, contains the ray $(\lambda_1^m, +\infty)$, and we have
			\begin{equation*}\label{eq:res-l>>1}
				R(\lambda; A)f=(\lambda I -A)^{-1} f (x)= \dfrac{-1}{a(x)}\displaystyle \int_{x_1}^x f(y) e^{\int_{x}^y \frac{\sigma(z)-\lambda}{a(z)}\dd z}\dd y, \text{ for any } x\not\in \{x_0, x_1\} 
			\end{equation*}
			for any $f\in L^1\big([x_0, x_1]\big)$ and $\lambda>\lambda_1^m$.
		\item When $\lambda\to \lambda_1^m$ we have 
		    \begin{equation}\label{eq:251124a}
			\lim_{\lambda\to (\lambda_1^m)^+} \Vert (\lambda I-A)^{-1}\Vert_{\mathcal{L}(L^1)} = +\infty,
		    \end{equation}
		    and the spectral bound of $A$ is $s(A)=\lambda_1^m$.
	\end{enumerate}
\end{proposition}
\begin{proof}
	Let us consider assertion 1. Let $f\in L^1\big([x_0, x_1]\big)$, then it follows from classical ODE theory that the unique solution of \eqref{eq:resolvent-A} satisfying $\varphi(X_1)=\varphi_1 $ is given by the formula \eqref{eq:resolvent-A-formula-0}, where $X_0, X_1\in(x_0, x_1)$. Let us prove that 
	\begin{equation}\label{eq:251121a}
		\varphi(X_1)=\dfrac{e^{-\int_{X_0}^{X_1} \frac{\sigma(z)-\lambda}{a(z)}\dd z}}{a(X_1)}\int_{X_1}^{x_1} f(y)e^{-\int_{X_0}^y \frac{\sigma(z)-\lambda}{a(z)}\dd z}\dd y = \dfrac{\int_{X_1}^{x_1}f(y)\mathcal{E}_\lambda(y)\dd y}{a(X_1)\mathcal{E}_\lambda(X_1)} 
	\end{equation}
	is the only possible value for $\varphi(X_1)$. Indeed, suppose by contradiction that \eqref{eq:251121a} does not hold. We have by Lemma \ref{lem:E_lambda}: 
	\begin{equation*}\label{eq:251121e}
		\mathcal{E}_\lambda(y)=e^{\int_{X_0}^y\frac{\sigma(z)-\lambda}{a(z)}\dd z}= (x_1-y)^{\alpha_1} e^{\mathcal{O}(1)}, 
	\end{equation*}
	which is bounded in the neighborhood of $x_1$ because  $\alpha_1=\dfrac{\sigma(x_1)-\lambda}{a'(x_1)}>0$ (where $\alpha_1$ has been defined in \eqref{eq:alphas}; recall that  $\lambda>\lambda_1^m\geq \lambda_1^1=\sigma(x_1)$ and $a'(x_1)<0$). In particular, writing
	\begin{equation*}
		\varphi(x) = 	\dfrac{ -\int_{X_1}^x f(y)e^{\int_{X_0}^y \frac{\sigma(z)-\lambda}{a(z)}\dd z}\dd y +a(X_1)e^{\int_{X_0}^{X_1} \frac{\sigma(z)-\lambda}{a(z)}\dd z}\varphi(X_1)  }{a(x)e^{\int_{X_0}^x \frac{\sigma(z)-\lambda}{a(z)}\dd z}} =: \dfrac{N(x)}{D(x)}, 
	\end{equation*}
	then $N(x)$ is a continuous function up to $x=x_1$ and $N(x_1)\neq 0$ because \eqref{eq:251121a} does not hold; on the other hand, $D(x)=-a'(x_1)(x_1-x)^{\alpha_1+1}e^{\mathcal{O}(1)} + o((x_1-x)^{\alpha_1+1})$ so $\frac{1}{D(x)}\not\in L^1$. From that we deduce that $\varphi\not\in L^1\big([x_0, x_1]\big)$. The contradiction proves \eqref{eq:251121a}.

	Hence we have identified the unique value of $\varphi(X_1)$ such that $\varphi\in L^1\big([x_0, x_1]\big)$. Notice that  \eqref{eq:251121a} leads to the expression
	\begin{equation}\label{eq:251121d}
		\varphi(x)=-	a(x)^{-1} \int_{x_1}^x f(y)e^{\int_x^y \frac{\sigma(z)-\lambda}{a(z)}\dd z}\dd y,
	\end{equation}
	that is independent of the choice of $X_0$.
	\medskip

	Let us prove that $\varphi(x)$,  defined by the formula \eqref{eq:251121d}, belongs to $L^1\big([x_0, x_1]\big)$. The following computation can be justified by Fubini's Theorem.  Choose $X_0<x_1$ sufficiently close to $x_1$ so that $\sigma(x)-\lambda \leq -m<0$ for $x\in(X_0, x_1)$; then
	\begin{align*}
	    \int_{X_0}^{x_1}|\varphi(x)|\dd x&= \int_{X_0}^{x_1} a(x)^{-1} \int_{x}^{x_1} |f(y)| e^{\int_x^y \frac{\sigma(z)-\lambda}{a(z)}\dd z}\dd y\dd x=\int_{X_0}^{x_1} |f(y)|\int_{X_0}^y\frac{1}{a(x)} e^{\int_x^y \frac{\sigma(z)-\lambda}{a(z)}\dd z}\dd x\dd y \\
	    &=\int_{X_0}^{x_1} |f(y)|\int_{X_0}^y-\frac{\sigma(x)-\lambda}{a(x)} e^{-\int_y^x \frac{\sigma(z)-\lambda}{a(z)}\dd z}\times \frac{-1}{\sigma(x)-\lambda}\dd x\dd y \\
	    &\leq\frac{1}{m}\int_{X_0}^{x_1}|f(y)| \left[e^{-\int_y^x \frac{\sigma(z)-\lambda}{a(z)}\dd z}\right]_{x=X_0}^{x=y} \dd y = \frac{1}{m}\int_{X_0}^{x_1} |f(y)|\left(1-e^{\int_{X_0}^y \frac{\sigma(z)-\lambda}{a(z)}\dd z}\right)\dd y \\ 
	    &=\frac{1}{m}\int_{X_0}^{x_1} |f(y)|\left(1-\dfrac{(x_1-y)^{\alpha_1}}{(x_1-X_0)^{\alpha_1}}e^{\mathcal{O}(1)}\right)\dd y  \leq \frac{1}{m}\Vert f\Vert_{L^1_{per}}, 
	\end{align*}
    where we recall that $\alpha_1=\frac{\sigma(x_1)-\lambda}{a'(x_1)}>0$. % and $C>0$ is a generic constant name that may change in the following computations. 
	Hence we conclude that ${\varphi}\in L^1$ in the neighborhood of $x_1$. Next we consider the integrability of $\varphi(x)$ in the neighborhood of $x_0$. Lemma \ref{lem:E_lambda} leads to 
    \begin{equation}\label{eq:251121f}
	    \mathcal{E}_\lambda(y)=e^{\int_{X_0}^y\frac{\sigma(z)-\lambda}{a(z)}\dd z} = (y-x_0)^{\alpha_0} e^{\mathcal{O}(1)}, 
    \end{equation}
    where $\alpha_0:=\dfrac{\sigma(x_0)-\lambda}{a'(x_0)}<0$.
	With $X_0\in (x_0, x_1)$, we have
    \begin{align}
	    \nonumber
	    \int_{x_0}^{X_0}|\varphi(x)|\dd x&\leq  \int_{x_0}^{X_0}\dfrac{\int_x^{x_1} |f(y)|e^{\int_{X_0}^y \frac{\sigma(z)-\lambda}{a(z)}\dd z}\dd y}{a(x)e^{\int_{X_0}^x \frac{\sigma(z)-\lambda}{a(z)}\dd z}}\dd x \\
	    \nonumber
	    &= \int_{x_0}^{x_1} |f(y)|e^{\int_{X_0}^y \frac{\sigma(z)-\lambda}{a(z)}\dd z} \int_{x_0}^{\min(y, X_0)} e^{\mathcal{O}(1)} (x-x_0)^{-(1+\alpha_0)}\dd x\dd y\\
	    \nonumber
	    &\leq C \int_{x_0}^{x_1}|f(y)|e^{\int_{X_0}^y \frac{\sigma(z)-\lambda}{a(z)}\dd z} \dfrac{\big(\min(y, X_0)-x_0)^{-\alpha_0}}{-\alpha_0}\dd y \\ 
	    \nonumber
	    &\leq C\left(\int_{x_0}^{X_0}|f(y)| (y-x_0)^{\alpha_0}e^{\mathcal{O}(1)}(y-x_0)^{-\alpha_0}\dd y + \Vert f\Vert_{L^1}\right) \\ 
	    \label{eq:251126b}
	    &\leq C\Vert f\Vert_{L^1}, 
    \end{align}
	where $C>0$ is a positive constant. 	Thus we have proved that $\varphi\in L^1$; $(\lambda I-A)^{-1}$ is indeed a bounded operator on $L^1$, and corresponds to the resolvent of $\lambda I-A$ for $\lambda >\lambda_1^m$.
	\medskip

	Next we consider assertion 2, the blow-up when $\lambda\to (\lambda_1^m)^+$. Choose $X_0\in (x_0, x_1)$ and suppose first that $\lambda_1^m=\lambda_1^0\geq\lambda_1^1$. Notice that the $e^{\mathcal{O}(1)}$ constant in \eqref{eq:251121f} is locally bounded with respect to $\lambda$. We have, with $f\equiv 1$:
	\begin{align}
		\nonumber
	    \int_{x_0}^{X_0} (\lambda I-A)^{-1} \mathbbm{1} \dd x&= \int_{x_0}^{X_0} \dfrac{\int_x^{x_1} e^{\int_{X_0}^y \frac{\sigma(z)-\lambda}{a(z)}\dd z}\dd y}{a(x)e^{\int_{X_0}^x \frac{\sigma(z)-\lambda}{a(z)}\dd z}}\dd x= \int_{x_0}^{x_1} e^{\int_{x_0}^y \frac{\sigma(z)-\lambda}{a(z)}\dd z} \int_{x_0}^{\min(y, X_0)}  \dfrac{e^{\mathcal{O}(1)}}{(x-x_0)^{(1+\alpha_0)}}\dd x\dd y\\
	    \nonumber
	    &\geq C\int_{x_0}^{X_0} e^{\int_{X_0}^y \frac{\sigma(z)-\lambda}{a(z)}\dd z} \left[ \frac{(x-x_0)^{-\frac{\sigma(x_0)-\lambda}{a'(x_0)}}}{-\frac{\sigma(x_0)-\lambda}{a'(x_0)}}\right]_{x=x_0}^{x=y}\dd y\\
	    \nonumber
	    &\geq C \dfrac{-a'(x_0)}{\sigma(x_0)-\lambda}\int_{x_0}^{X_0}(y-x_0)^{\frac{\sigma(x_0)-\lambda}{a'(x_0)}}e^{\mathcal{O}(1)} (y-x_0)^{-\frac{\sigma(x_0)-\lambda}{a'(x_0)}}\dd y \\ 
		\label{eq:251125c}
	    &\geq C\dfrac{-a'(x_0)}{\sigma(x_0)-\lambda} (X_0-x_0) \xrightarrow[\lambda\to \sigma(x_0)^+]{}+\infty, 
	\end{align}
	which, recalling $\lambda_1^0=\sigma(x_0)$, proves the blow-up of $\Vert (\lambda I-A)^{-1}\Vert_{\mathcal{L}(L^1)}$ as $\lambda\to (\lambda_1^m)^+$. Suppose now that $\lambda_1^m=\lambda_1^1>\lambda_1^0$.  
	%again $f\equiv 1$, we have
	%$f(x)=\frac{1-\alpha}{2^{\alpha}|x-x_1|^{\alpha}}$ so that $\Vert f\Vert_{L^1_{per}}=\int_{x_1-\frac{1}{2}}^{x_1+\frac{1}{2}}f(x)\dd x = 1$. 
	Then, with $\alpha_1=\frac{\sigma(x_1)-\lambda}{a'(x_1)}>0$,
	\begin{align*}
		\int_{X_0}^{x_1}\varphi(x)\dd x&= \int_{X_0}^{x_1} a(x)^{-1} \int_{x}^{x_1} f(y) e^{\int_x^y \frac{\sigma(z)-\lambda}{a(z)}\dd z}\dd y\dd x=\int_{X_0}^{x_1} \int_{X_0}^y\frac{1}{a(x)}f(y) e^{\int_x^y \frac{\sigma(z)-\lambda}{a(z)}\dd z}\dd x\dd y \\
		&=\int_{X_0}^{x_1}f(y) e^{\int_{X_0}^y \frac{\sigma(z)-\lambda}{a(z)}\dd z}\int_{X_0}^y\dfrac{e^{\mathcal{O}(1)}}{(x_1-x)^{\alpha_1+1}}\dd x\dd y \\
		&\geq C \int_{X_0}^{x_1} f(y)(x_1-y)^{\alpha_1} \left[-\dfrac{(x_1-x)^{-\alpha_1}}{-\alpha_1}\right]_{x=X_0}^{x=y}\dd y \\
		&=C\int_{X_0}^{x_1}f(y)\frac{1}{\alpha_1} \left( \dfrac{(x_1-y)^{\alpha_1}}{(x_1-X_0)^{\alpha_1}}-1\right)\dd y = C\int_{X_0}^{x_1}f(y)\frac{1}{\alpha_1} \left( e^{\alpha_1\ln\left(\frac{x_1-y}{x_1-X_0}\right)+o(\alpha_1)}-1\right)\dd y \\
		&=C\int_{X_0}^{x_1}f(y)\left(\ln\left(\frac{x_1-y}{x_1-X_0}\right)+o(\alpha_1)\right)\dd y\xrightarrow[\lambda\to \lambda_1^1]{\alpha_1\to 0}C\int_{X_0}^{x_1}f(y)\ln\left(\frac{x_1-y}{x_1-X_0}\right)\dd y .
	\end{align*}
	Now take $f(x)=\frac{\mathbbm{1}_{x\in[X_0, x_1)}}{(x_1-x)\big|\ln(x_1-x)\big|^2}$ so that $f\in L^1$ but $f(x)\ln(x-x_1)\not\in L^1_{per}$,  then the right-hand side  is $+\infty$.
	This shows \eqref{eq:251124a}.
	In particular,  the resolvent $R(\lambda; A)$ has a singularity at $\lambda=\lambda_1^m$; we deduce that $\lambda_1^m\in \sigma(A)$, the spectrum of $A$.
	Since  $A$ is the generator of a $C_0$-semigroup of positive operators by Theorem \ref{THEO1} and $\sigma(A)\neq \varnothing$,  it follows from \cite[Theorem 3.3]{Greiner-Voigt-Wolff-1981} that $s(A)\in\sigma(A)$ and we conclude that $s(A)=\lambda_1^m$.
\end{proof}

Next we consider the restriction of $A$ to the space $C^0\big([x_0, x_1]\big)$ and prove a similar result. 
\begin{proposition}\label{prop:resolvent-A-C0}
    Let Assumption \ref{as:interval} hold true. Call $r^m:=\max\big(r(x_0), r(x_1)\big)$ and let $\big(A_0, D(A_0)\big)$ be the operator defined by \eqref{defA0int}.
	\begin{enumerate}
	    \item The resolvent set of $A_0$, $\rho(A_0)$, contains the ray $\big(r^m, +\infty\big)$, and we have
			\begin{equation}\label{eq:A0-res-l>>1}
			    \begin{gathered}
				    R(\lambda; A_0)f(x)=(\lambda I -A_0)^{-1} f (x)= \dfrac{-1}{a(x)}\displaystyle \int_{x_1}^x f(y) e^{\int_{x}^y \frac{\sigma(z)-\lambda}{a(z)}\dd z}\dd y,\qquad  \text{ for any } x\not\in \{x_0, x_1\} \\
				(\lambda I -A_0)^{-1} f (x_0) = \dfrac{f(x_0)}{\lambda - r(x_0)},\qquad  
				(\lambda I -A_0)^{-1} f (x_1) = \dfrac{f(x_1)}{\lambda-r(x_1)},
			    \end{gathered}
			\end{equation}
			for any $f\in C^0\big([x_0, x_1]\big)$ and $\lambda>r^m$.
		\item When $\lambda\to r^m$ we have 
		    \begin{equation*}\label{eq:251124b}
			\lim_{\lambda\to (r^m)^+} \Vert (\lambda I-A_0)^{-1}\Vert_{\mathcal{L}(C^0)} = +\infty,
		    \end{equation*}
		    and the spectral bound of $A_0$ is $s(A_0)=r^m$.
	\end{enumerate}
\end{proposition}
\begin{proof}
	The formula \eqref{eq:A0-res-l>>1} for $(\lambda I-A_0)^{-1}f(x)$ when $x\not\in\{x_0, x_1\}$ follows from a similar analysis as in the proof of {Proposition \ref{prop:resolvent-A-L1}}. Let us check that \eqref{eq:A0-res-l>>1} is well-defined when $\lambda\in (r(x_1), \sigma(x_1))$.    
	Recall from Lemma \ref{lem:E_lambda} that $\mathcal{E}_\lambda(y)=e^{\int_{X_0}^y\frac{\sigma(z)-\lambda}{a(z)}\dd z}=(x_1-y)^{\frac{\sigma(x_1)-\lambda}{a'(x_1)}}e^{\mathcal{O}(1)}$, where the constant $\mathcal{O}(1)$ is locally bounded with respect to $\lambda$. Since $\lambda>r(x_1)$ we have that 
    \begin{equation*}
	\frac{\sigma(x_1)-\lambda}{a'(x_1)}=\frac{r(x_1)-a'(x_1)-\lambda}{a'(x_1)}= -1+\frac{r(x_1)-\lambda}{a'(x_1)}>-1,
    \end{equation*}
	hence $f(y)e^{\int_{X_0}^y\frac{\sigma(z)-\lambda}{a(z)}\dd z} = (x_1-y)^{\frac{\sigma(x_1)-\lambda}{a'(x_1)}}e^{\mathcal{O}(1)} $ is integrable for $y\in (x, x_1)$. Thus \eqref{eq:A0-res-l>>1} is well-defined for any $\lambda>r^m$ when $x\not\in \{x_0, x_1\}$ (recall that $a'(x_0)>0$ and therefore $r(x_0)>\sigma(x_0)=r(x_0)-a'(x_0)$).

    For $x< x_1$ but close to $x_1$, we have
    \begin{equation*}
	\varphi(x) = \dfrac{\int_{x}^{x_1} f(y) e^{\int_{X_0}^y\frac{\sigma(z)-\lambda}{a(z)}\dd z}\dd y}{a(x) e^{\int_{X_0}^x \frac{\sigma(z)-\lambda}{a(z)}\dd z}}=:\dfrac{N(x)}{D(x)}. 
    \end{equation*}
    Since $f(y)e^{\int_{X_0}^y\frac{\sigma(z)-\lambda}{a(z)}\dd z}$ is $L^1$ in the neighborhood of $x_1$, we have that $N(x)\to 0$ as $x\to x_1$. Similarly, $1+\frac{\sigma(x_1)-\lambda}{a'(x_1)}\geq 0$ and therefore $a(x)e^{\int_{X_0}^x\frac{\sigma(z)-\lambda}{a(z)}\dd z}=D(x)\to 0$ as $x\to x_1$. Both $N$ and $D$ are differentiable (even $C^1$) in the neighborhood of $x_1$ and 
    \begin{equation*}
	\dfrac{N'(x)}{D'(x)} =  \dfrac{-f(x) e^{\int_{X_0}^x\frac{\sigma(z)-\lambda}{a(z)}\dd z} }{\left(a'(x)+a(x)\frac{\sigma(x)-\lambda}{a(x)}\right)e^{\int_{X_0}^x\frac{\sigma(z)-\lambda}{a(z)}\dd z}}=\dfrac{-f(x)}{r(x)-\lambda} 
	\xrightarrow[x\to x_1]{}\dfrac{f(x_1)}{\lambda-r(x_1)}.
    \end{equation*}
    Hence by L'Hospital rule \cite{Taylor-1952} we conclude that 
    \begin{equation}\label{eq:251127b}
	\lim_{x\to x_1}\varphi(x_1) = \lim_{x\to x_1}\frac{N(x)}{D(x)} = \lim_{x\to x_1}\frac{N'(x)}{D'(x)} =\dfrac{f(x_1)}{\lambda-r(x_1)}.
    \end{equation}

    For $x> x_0$ but close to $x_0$, we have
    \begin{equation*}
	\varphi(x) = \dfrac{\int_{x}^{x_1} f(y) e^{\int_{X_0}^y\frac{\sigma(z)-\lambda}{a(z)}\dd z}\dd y}{a(x) e^{\int_{X_0}^x \frac{\sigma(z)-\lambda}{a(z)}\dd z}}=:\dfrac{N(x)}{D(x)}. 
    \end{equation*}
	Recall from Lemma \ref{lem:E_lambda} that $\mathcal{E}_\lambda(y)=e^{\int_{X_0}^y\frac{\sigma(z)-\lambda}{a(z)}\dd z}=(y-x_0)^{\frac{\sigma(x_0)-\lambda}{a'(x_0)}}e^{\mathcal{O}(1)}$, where the constant $\mathcal{O}(1)$ is locally bounded in $\lambda$. Since $\lambda>r(x_0)$ we have that 
    \begin{equation*}
	\frac{\sigma(x_0)-\lambda}{a'(x_0)}=\frac{r(x_0)-a'(x_0)-\lambda}{a'(x_0)}= -1+\frac{r(x_0)-\lambda}{a'(x_0)}<-1.
    \end{equation*}
    Hence $\frac{\sigma(x_0)-\lambda}{a'(x_0)}+1<0$ and therefore $D(x)\to +\infty$ as $x\to x_0$ (recall $a(x)=a'(x_0)(x-x_0)+o(x-x_0)$ has constant sign in the neighborhood of $x_0$). We have 
    \begin{equation*}
	\dfrac{N'(x)}{D'(x)} =  \dfrac{-f(x) e^{\int_{X_0}^x\frac{\sigma(z)-\lambda}{a(z)}\dd z} }{\left(a'(x)+a(x)\frac{\sigma(x)-\lambda}{a(x)}\right)e^{\int_{X_0}^x\frac{\sigma(z)-\lambda}{a(z)}\dd z}}=\dfrac{-f(x)}{r(x)-\lambda} 
	\xrightarrow[x\to x_0]{}\dfrac{f(x_0)}{\lambda-r(x_0)}, 
    \end{equation*}
    and since $D(x)\to \infty$ as $x\to x_0$, we apply L'Hospital rule and obtain 
    \begin{equation*}
	\lim_{x\to x_0}\varphi(x_0) = \lim_{x\to x_0}\frac{N(x)}{D(x)} = \lim_{x\to x_0}\frac{N'(x)}{D'(x)} =\dfrac{f(x_0)}{\lambda-r(x_0)}.
    \end{equation*}

    Finally we note that
    \begin{equation*}
	    |\varphi(x)|\leq \sup_{y\in{[x_0, x_1]}}|f(y)| \underset{g(x)}{\underbrace{\dfrac{\left|\int_x^{x_1} e^{\int_{X_0}^y\frac{\sigma(z)-\lambda}{a(z)}\dd z}\dd y\right| }{|a(x)|e^{\int_{X_0}^x\frac{\sigma(z)-\lambda}{a(z)}\dd z}}}}
    \end{equation*}
    where $g(x)$ is continuous on $(x_0, x_1)$ and  
	\begin{align}\label{eq:251126a}
	g(x) &\leq C\dfrac{ \frac{1}{1+\alpha_1}|x-x_1|^{1+\alpha_1}}{|x-x_1|^{1+\alpha_1}}\leq \frac{C}{-(1+\frac{\sigma(x_1)-\lambda}{a'(x_1)})}= \frac{C|a'(x_1)|}{\lambda-r(x_1)} , \text{ as } x\to x_1,
    \end{align}
    and 
    \begin{align*}
	g(x) &\leq C\dfrac{ \frac{1}{|1+\alpha_0|}\left||x-x_0|^{1+\alpha_0} - |x_1-x_0|^{1+\alpha_0}\right|}{|x-x_0|^{1+\alpha_0}}\leq \frac{C}{-(1+\frac{\sigma(x_0)-\lambda}{a'(x_0)})}=\frac{Ca'(x_0)}{\lambda-r(x_0)}  , \text{ as } x\to x_0,
    \end{align*}
    therefore $\Vert \varphi\Vert_{C^0}\leq C\Vert f\Vert_{C^0}$ for some constant $C$ independent of $f$. 
    This completes the proof of assertion 1.  The blow-up of the operator norm in assertion 2 follows immediately from the formula \eqref{eq:A0-res-l>>1} at $x=x_0$ and $x=x_1$ (take any $f$ with either $f(x_1)\neq 0$ or $f(x_0)\neq 0$. Thus $r^m\in\sigma(A)$, the spectrum of $A$, which is therefore non-empty; $A$ being the generator of a $C_0$-semigroup of positive operators,  it follows from \cite[Theorem 3.3]{Greiner-Voigt-Wolff-1981} that $s(A)\in\sigma(A)$ and we conclude that $s(A)=r^m$. Proposition \ref{prop:resolvent-A-C0} is proved. 
\end{proof}

In the next Proposition we show that the resolvent of $\lambda I-A$ can be extended outside of the resolvent set prescribed by {Proposition \ref{prop:resolvent-A-L1} or Proposition \ref{prop:resolvent-A-C0}} as a right-inverse of $\lambda I-A$. The quantity $\lambda_1^{ext}$ plays the role of the spectral bound but is ``extended'' in the sense that it is always below both $\lambda_1^m$ and $r^m$. Notice that the extension is not unique but forms an ordered family of operators.

\begin{proposition}[Extended resolvent family]\label{prop:resolvent-A-C0-L1}
	Let Assumption \ref{as:interval} hold. Recall $\lambda_1^{ext}=\max\big(\lambda_1^0, r(x_1)\big)$ and let $\big(A_0, D(A_0)\big)$ be the operator defined by \eqref{defA0int}. We assume that $\lambda_1^0 < \lambda_1^1$, so that the interval $(\lambda_1^{ext}, \lambda_1^m)$ is not empty.
	\begin{enumerate}
	    \item The resolvent of $A_0$ can be extended into a one-parameter family of continuous operators $ R_\gamma(\lambda; A_0):C^0\big([x_0, x_1]\big) \to L^1\big([x_0, x_1]\big)$ by the formula 
		    \begin{equation}\label{eq:formula-extres}
			    R_\gamma(\lambda; A_0) f(x):=\dfrac{\displaystyle -\int_{x_1}^x f(y)\mathcal{E}_\lambda(y)\dd y + \gamma \int_{X_1}^{x_1} f(y)\mathcal{E}_\lambda(y)\dd y}{a(x)\mathcal{E}_\lambda(x)},
		    \end{equation}
		    for any $\gamma\in\mathbb{R}$ and $\lambda\in (\lambda_1^{ext}, \lambda_1^m)$. $R_\gamma(\lambda; A_0)$ is a right-inverse of $\lambda I-A$, meaning that we have $(\lambda I-A)R_\gamma(\lambda; A_0)f=f$ for all $f\in C^0$.
	    \item The family $R_\gamma(\lambda; A_0)$ is continuous in the operator norm of $\mathcal{L}\big(C^0([x_0, x_1]), L^1([x_0, x_1])\big)$ as a function of the parameter $\lambda\in(\lambda_1^{ext}, \lambda_1^m)$ and $\gamma\in \mathbb{R}$; for each $\gamma\geq 0$, $R_\gamma(\lambda; A_0)$ is a positive operator. For any  $\lambda\in(\lambda_1^{ext}, \lambda_1^m)$, $[0, +\infty)\ni\gamma\mapsto R_\gamma(\lambda; A_0)$ is  increasing; when $\gamma= 0$, $(\lambda_1^{ext}, \lambda_1^m)\ni\lambda\mapsto R_0(\lambda; A_0)$ is  decreasing.
	    \item When $\lambda\to \lambda_1^{ext}$, we have  for any $\gamma>0$:
		    \begin{equation}\label{eq:251209a}
			    \lim_{\lambda\to (\lambda_1^{ext})^+} \Vert R_\gamma(\lambda; A_0)\Vert_{\mathcal{L}(C^0, L^1)} = +\infty,
		    \end{equation}
			thus $\lambda_1^{ext}$ is the minimal value for which $R_\gamma(\lambda; A_0)$ is defined and continuous on $\lambda\in(\lambda_1^{ext},+\infty)$. 
	    \item When $\lambda\to (\lambda_1^m)^-$, we have for any $\gamma>0$: 
		    \begin{equation}\label{eq:251209b}
			    \lim_{\lambda\to (\lambda_1^{m})^-} \Vert R_\gamma(\lambda; A_0)\Vert_{\mathcal{L}(C^0, L^1)} = +\infty.
		    \end{equation}
		On the other hand when $\gamma=0$, then $R_0(\lambda; A_0)$ is a continuation of $R(\lambda; A)$ for the operator topology of $\mathcal{L}(C^0, L^1)$,  i.e., $R_0(\lambda_1^m; A_0)$ is well-defined by \eqref{eq:formula-extres} and  
		    \begin{equation}\label{eq:251209c}
			    \lim_{\lambda\to(\lambda_1^{m})^-} \Vert R_0(\lambda; A_0)-R_0(\lambda_1^m; A)\Vert_{\mathcal{L}(C^0, L^1)} = 0 \text{ and } \lim_{\lambda\to(\lambda_1^{m})^+} \Vert R_0(\lambda_1^m; A)-R(\lambda; A)\Vert_{\mathcal{L}(C^0, L^1)} = 0.
		    \end{equation}
	    \item  Suppose $\lambda=\lambda_1^1$, and let $f\in D(A)$. If $(\lambda I-A)f=g$ for some $g\in C^0([x_0,x_1])$, then $f=R_0(\lambda; A_0)g$. In particular,  $R_0(\lambda_1^1; A_0)$ is also a left inverse of $\lambda_1^1 I - A_0$, meaning that $R_0(\lambda_1^1; A_0)(\lambda_1^1 I-A_0)f=f$ for any $f\in  D(A_0)$.
	\end{enumerate}
\end{proposition}
\begin{remark}
	When $\lambda=\lambda_1^1$ and $\lambda_1^0<\lambda_1^1<r(x_0)$, we have that $R_0(\lambda; A_0)\big(C^0([x_0, x_1])\big)\not\subset C^0([x_0, x_1])$; therefore $R_0(\lambda; A_0)$ is not a bounded inverse of $\lambda I-A_0$, even though $R_0(\lambda; A_0)(\lambda I-A_0)f=(\lambda I-A_0)R_0(\lambda ;A_0)f=f$ for any $f\in D(A_0)$.
\end{remark}
\begin{proof}[Proof of Proposition \ref{prop:resolvent-A-C0-L1}]
	We first consider assertion 1. 
	We note that, since $\lambda>r(x_1)$, Lemma \ref{lem:E_lambda} yields $\mathcal{E}_\lambda(x)=|x-x_1|^{\alpha_1}e^{\mathcal{O}(1)}\in L^1$ with  $\alpha_1=\frac{\sigma(x_1)-\lambda}{a'(x_1)}=-1+\frac{r(x_1)-\lambda}{a'(x_1)}>-1$, in the neighborhood of $x_1$. Thus we can rewrite \eqref{eq:resolvent-A-formula-0} as
	\begin{align*}
		\varphi(x)&=R_\gamma(\lambda; A_0)f(x) = \dfrac{-\int_{x_1}^xf(y)\mathcal{E}_\lambda(y)\dd y -\int_{X_1}^{x_1} f(y)\mathcal{E}_\lambda(y)\dd y + a(X_1)\varphi(X_1)\mathcal{E}_\lambda(X_1)}{a(x)\mathcal{E}_\lambda(x)}, 
	\end{align*}
	therefore by selecting 
	\begin{equation*}
		\varphi(X_1):= \dfrac{1+\gamma}{a(X_1)\mathcal{E}_\lambda(X_1)}\int_{X_1}^{x_1}f(y)\mathcal{E}_\lambda(y)\dd y
	\end{equation*}
	we obtain the formula \eqref{eq:formula-extres}. For any $\gamma\in\mathbb{R}$ and $f\in C^0([x_0, x_1])$, it is then clear that $(\lambda I-A)R_\gamma(\lambda; A_0)f=f$.

	Let us show that $R_\gamma(\lambda; A_0)$ is a bounded operator. Let $f\in C^0\big([x_0, x_1]\big)$ and $\varphi=R_\gamma(\lambda; A_0) f$, we have
	\begin{align*}
		|\varphi(x)|&\leq \Vert f\Vert_{C^0} \underset{g(x)}{\underbrace{\dfrac{\left|\int_x^{x_1} e^{\int_{X_0}^y\frac{\sigma(z)-\lambda}{a(z)}\dd z}\dd y\right| }{|a(x)|e^{\int_{X_0}^x\frac{\sigma(z)-\lambda}{a(z)}\dd z}}}}
		+|\gamma|\int_{X_1}^{x_1}\mathcal{E}_\lambda(y)\dd y\Vert f\Vert_{C^0}\underset{h(x)}{\underbrace{ \frac{1}{a(x)\mathcal{E}_\lambda(x)}}}.
	\end{align*}
	The functions $g(x)$ and $h(x)$ are continuous on $(x_0, x_1)$. 
	By applying L'Hospital rule (as in the proof of Proposition \ref{prop:resolvent-A-C0}), we have as $x\to x_1$:
    \begin{align*}
	g(x) &\leq C\dfrac{ \frac{1}{1+\alpha_1}|x-x_1|^{1+\alpha_1}}{|x-x_1|^{1+\alpha_1}}\leq \frac{C}{-(1+\frac{\sigma(x_1)-\lambda}{a'(x_1)})}= \frac{C|a'(x_1)|}{\lambda-r(x_1)} ,
    \end{align*}
	so $g(x)$ is bounded when $x$ is close to $x_1$. When $x$ is close to $x_0$, we have
	\begin{align}
		\label{eq:251125a}
		g(x)&=\dfrac{\left|\int_x^{x_1} e^{\mathcal{O}(1)}|x-x_0|^{\alpha_0}\dd y\right| }{|a(x)|e^{\mathcal{O}(1)}|x-x_0|^{\alpha_0}}.
	\end{align}
	Recall that $\alpha_0=\frac{\sigma(x_0)-\lambda}{a'(x_0)}\leq -1$. When $\lambda>r(x_0)$ then we conclude that $g(x)$ is bounded in the neighborhood of $x_0$ as in the proof of Proposition \ref{prop:resolvent-A-C0}; thus it is in $L^1$. If $\lambda = r(x_0)$ we have $\alpha_0=-1$, and \eqref{eq:251125a} can be estimated by 
	\begin{align*}
		g(x)&\leq C |\ln(x-x_0)-\ln(x_1-x_0)|,
	\end{align*}
	which is an integrable function in the neighborhood of $x_0$. If $\lambda_1^0<\lambda < r(x_0)$ we have $\alpha_0\in(-1, 0)$, and \eqref{eq:251125a} can be estimated by 
	\begin{align*}
		g(x)&\leq C \dfrac{\big||x-x_0|^{1+\alpha_0}-|x_1-x_0|^{1+\alpha_0}\big|}{|x-x_0|^{1+\alpha_0}} = C\left|1-\frac{|x_1-x_0|^{\alpha_0+1}}{|x-x_0|^{1+\alpha_0}}\right|
	\end{align*}
	which is an integrable function in the neighborhood of $x_0$. Finally since $r(x_1)<\lambda<\sigma(x_1)$ we have $-1<\alpha_1=\frac{\sigma(x_1)-\lambda}{a'(x_1)}<0$, and since $\lambda>\sigma(x_0)$ we have $\alpha_0=\frac{\sigma(x_0)-\lambda}{a'(x_0)}<0$, therefore
	\begin{equation*}
		h(x)=|x-x_1|^{-1-\alpha_1}e^{\mathcal{O}(1)}  \text{ as } x\to x_1,\quad h(x)=|x-x_0|^{-1-\alpha_0}e^{\mathcal{O}(1)}, 
	\end{equation*}
	which proves that $h\in L^1 $ in $[x_0, x_1]$.
	Thus we have the uniform bound
	\begin{equation*}
		\Vert (\lambda I-A_0)^{-1}_R f\Vert_{L^1} = \Vert \varphi\Vert_{L^1}\leq \Vert f\Vert_{C^0}\left( \Vert g\Vert_{L^1}+ |\gamma|\int_{X_1}^{x_1}\mathcal{E}_\lambda(y)\dd y\Vert h\Vert_{L^1} \right).
	\end{equation*}
	This finishes the proof of assertion 1.
	\medskip

	Next we consider assertion 2. Let $f\in C^0\big([x_0, x_1]\big)$, $f\geq 0$, $\gamma\geq 0$ and $\lambda_1^{ext}<\lambda'\leq \lambda <\lambda_1^m$  be given. We can immediately see from \eqref{eq:formula-extres} that $R_\gamma(\lambda; A_0)f\geq 0$, hence $R_\gamma(\lambda; A_0)$ is a positive operator.  We have, for $x\in (x_0, x_1)$:
	\begin{align}
		\nonumber
		\big[R_0(\lambda'; A_0)-&R_0(\lambda; A_0)\big] f(x) = \frac{1}{a(x)}  \int_x^{x_1} f(y) \left(e^{\int_x^y \frac{\sigma(z)-\lambda'}{a(z)}\dd z} - e^{\int_x^y \frac{\sigma(z)-\lambda}{a(z)}\dd z}\right)\dd y \\ 
		%\big[R_\gamma(\lambda'; A_0)-&R_\gamma(\lambda; A_0)\big] f(x) = \frac{1}{a(x)}  \int_x^{x_1} f(y) \left(e^{\int_x^y \frac{\sigma(z)-\lambda'}{a(z)}\dd z} - e^{\int_x^y \frac{\sigma(z)-\lambda}{a(z)}\dd z}\right)\dd y \\ 
		%\nonumber
		%&\quad + \gamma  \frac{1}{a(x)}  \int_{X_1}^{x_1} f(y) \left(e^{\int_x^y \frac{\sigma(z)-\lambda'}{a(z)}\dd z} - e^{\int_x^y \frac{\sigma(z)-\lambda}{a(z)}\dd z}\right)\dd y\\
		\label{eq:251125d}
		%&=\frac{1}{a(x)}  \left(\int_x^{x_1}+\gamma\int_{X_1}^{x_1}\right) f(y) e^{\int_x^y \frac{\sigma(z)-\lambda}{a(z)}\dd z}\left(e^{\int_x^y \frac{\lambda-\lambda'}{a(z)}\dd z} -1 \right)\dd y\geq 0,
		&=\frac{1}{a(x)}  \int_x^{x_1} f(y) e^{\int_x^y \frac{\sigma(z)-\lambda}{a(z)}\dd z}\left(e^{\int_x^y \frac{\lambda-\lambda'}{a(z)}\dd z} -1 \right)\dd y\geq 0,
	\end{align}
	This shows the monotony of $\lambda\mapsto R_\gamma(\lambda; A_0)$. The monotony of $\gamma\mapsto R_\gamma(\lambda; A_0)$ is clear.  Now, coming back to \eqref{eq:251125d}, we have
	\begin{align}
		\nonumber
		\big|\big[R_\gamma(\lambda'; A_0)&-R_\gamma(\lambda; A_0)\big] f(x)\big| \leq \frac{1}{a(x)}  \int_x^{x_1} |f(y)| \left|e^{\int_x^y \frac{\sigma(z)-\lambda'}{a(z)}\dd z} - e^{\int_x^y \frac{\sigma(z)-\lambda}{a(z)}\dd z}\right|\dd y \\ 
		\nonumber
		&\quad +\frac{\gamma}{a(x)} \int_{X_1}^{x_1}|f(y)|\left|e^{\int_x^y \frac{\sigma(z)-\lambda'}{a(z)}\dd z} - e^{\int_x^y \frac{\sigma(z)-\lambda}{a(z)}\dd z}\right|\dd y \\ 
		\label{eq:251209d}
		&\leq\Vert f\Vert_{C^0}\left( \underset{g_{\lambda, \lambda'}(x)}{\underbrace{\frac{1}{|a(x)|}  \int_x^{x_1} \left|e^{\int_x^y \frac{\sigma(z)-\lambda'}{a(z)}\dd z} - e^{\int_x^y \frac{\sigma(z)-\lambda}{a(z)}\dd z}\right|\dd y}} \right. \\ 
		\nonumber
		&\quad + \gamma \left.\underset{h_{\lambda, \lambda'}(x)}{\underbrace{\frac{1}{|a(x)|}  \int_{X_1}^{x_1} \left|e^{\int_x^y \frac{\sigma(z)-\lambda'}{a(z)}\dd z} - e^{\int_x^y \frac{\sigma(z)-\lambda}{a(z)}\dd z}\right|\dd y}} \right).
	\end{align}
	Freeze $x\in(x_0, x_1)$, clearly $\left|e^{\int_x^y \frac{\sigma(z)-\lambda'}{a(z)}\dd z} - e^{\int_x^y \frac{\sigma(z)-\lambda}{a(z)}\dd z}\right|\to 0$ as $\lambda\to \lambda'$ for any $y\in (x, x_1)$; moreover by Lemma \ref{lem:E_lambda} we have
	\begin{equation*} 
		\left|e^{\int_x^y \frac{\sigma(z)-\lambda'}{a(z)}\dd z} - e^{\int_x^y \frac{\sigma(z)-\lambda}{a(z)}\dd z}\right|\leq C(x_1-y)^{-1+\varepsilon}, 
	\end{equation*}
	and the right-hand side is integrable, where $0<\varepsilon< \min(\frac{\sigma(x_1)-\lambda}{a'(x_1)}, \frac{\sigma(x_1)-\lambda'}{a'(x_1)}) +1$. Thus by Lebesgue's dominated convergence Theorem, we have
	\begin{equation*}
		\lim_{\lambda'\to\lambda} g_{\lambda, \lambda'}(x)=0 \text{ and } \lim_{\lambda'\to\lambda} h_{\lambda, \lambda'}(x)=0 
	\end{equation*}
	for any $x\in (x_0, x_1)$. Next, remark that
	\begin{equation*}
		|g_{\lambda, \lambda'}(x)| \leq \big[R_\gamma(\lambda; A_0)\mathbbm{1} + R_\gamma(\lambda'; A_0)\mathbbm{1}\big](x)\leq 2R_\gamma(\bar{\lambda}; A_0)\mathbbm{1}, 
	\end{equation*}
	with $\lambda_1^{ext}<\bar{\lambda}<\min(\lambda, \lambda')$.
	It follows from the analysis in the proof of Proposition \ref{prop:resolvent-A-C0} (and more specifically, \eqref{eq:251126a}) that $R_\gamma(\bar{\lambda}; A_0)\mathbbm{1}(x)$ is uniformly bounded in the neighborhood of $x_1$ whenever $\bar{\lambda}>r(x_1)$, and from the analysis in the proof of Proposition \ref{prop:resolvent-A-L1} (more specifically, the computation leading to \eqref{eq:251126b}) that $ R_\gamma(\bar{\lambda};A_0)\mathbbm{1}(x)  $ is integrable in the neighborhood of $x_0$ whenever $\bar{\lambda}>\lambda_1^0$.  Thus $R_\gamma(\bar{\lambda}; A_0)\mathbbm{1}(x)$ is integrable on $(x_0, x_1)$ whenever $\bar{\lambda}>\max(\lambda_1^0, r(x_1))=\lambda_1^{ext}$. We conclude by Lebesgue's dominated convergence theorem that
	\begin{equation*}
		\lim_{\lambda'\to \lambda}\int_{x_0}^{x_1}g_{\lambda, \lambda'}(x)\dd x =0. 
	\end{equation*}
	Similarly, we have
	\begin{equation*}
		h_{\lambda, \lambda'}(x) \leq \dfrac{\int_{X_1}^{x_1}\mathcal{E}_{\lambda'}(y)\dd y}{a(x)\mathcal{E}_{\lambda'}(x)}+ \dfrac{\int_{X_1}^{x_1} \mathcal{E}_\lambda(y)\dd y}{a(x)\mathcal{E}_\lambda(x)} \leq C \dfrac{1}{a(x)\mathcal{E}_{\bar{\lambda}}(x)}\in L^1(x_0, x_1)
	\end{equation*}
	thus
	\begin{equation*}
		\lim_{\lambda'\to \lambda}\int_{x_0}^{x_1}h_{\lambda, \lambda'}(x)\dd x =0. 
	\end{equation*}
	Finally
	\begin{equation*}	
		\int_{x_0}^{x_1}\big|\big[R_\gamma(\lambda'; A_0)-R_\gamma(\lambda; A_0)\big] f(x)\big| \dd x \leq \Vert f\Vert_{C^0}\left(\int_{x_0}^{x_1} g_{\lambda, \lambda'}(x)\dd x + \gamma h_{\lambda, \lambda'}(x)\dd x\right) ,
	\end{equation*}
	which proves the continuity of $\lambda\mapsto R_\gamma(\lambda; A_0)$ in the operator norm. The continuity with respect to $\gamma$ is clear.  This completes the proof of assertion 2.
\medskip

	We next turn to assertion 3.  We have proved in Proposition \ref{prop:resolvent-A-L1} that $\Vert(\lambda I-A)^{-1}\mathbbm{1}\Vert_{L^1}\to +\infty $ when $\lambda\to\lambda_1^0\geq\lambda_1^1$, see \eqref{eq:251125c}; a similar computation proves that $\Vert R_\gamma(\lambda; A_0)\mathbbm{1}\Vert_{L^1}\to +\infty $ when $\lambda\to \lambda_1^0$ whenever $\lambda_1^0 \geq r(x_1)$. We are left with the case $r(x_1)>\lambda_1^0$ and $\lambda\to r(x_1)$. We conduct a similar computation in the neighborhood of $x_1$.
	 Choose $X_0\in (x_0, x_1)$ and suppose without loss of generality that $r(x_1)<\lambda<\sigma(x_1)=\lambda_1^1$ (so $\alpha_1=\frac{\sigma(x_1)-\lambda}{a'(x_1)}\in (-1, 0)$). We have, with $f\equiv 1$:
	\begin{align}
		\nonumber
		\int_{X_0}^{x_1} R_\gamma(\lambda; A_0) \mathbbm{1} \dd x&\geq \int_{X_0}^{x_1} \dfrac{\int_x^{x_1} e^{\int_{X_0}^y \frac{\sigma(z)-\lambda}{a(z)}\dd z}\dd y}{a(x)e^{\int_{X_0}^x \frac{\sigma(z)-\lambda}{a(z)}\dd z}}\dd x= \int_{X_0}^{x_1} e^{\int_{x_0}^y \frac{\sigma(z)-\lambda}{a(z)}\dd z} \int_{X_0}^{y}  \dfrac{e^{\mathcal{O}(1)}}{(x_1-x)^{(1+\alpha_1)}}\dd x\dd y\\
	    \nonumber
	    &\geq C\int_{X_0}^{x_1} e^{\int_{X_0}^y \frac{\sigma(z)-\lambda}{a(z)}\dd z} \left[- \frac{(x_1-x)^{-\alpha_1}}{-\alpha_1}\right]_{x=X_0}^{x=y}\dd y\\
	    \nonumber
		&\geq C\frac{-1}{\alpha_1}\int_{X_0}^{x_1}(x_1-y)^{\alpha_1}e^{\mathcal{O}(1)}\big[ (x_1-X_0)^{-\alpha_1} - (x_1-y)^{-\alpha_1}\big]\dd y \\ 
		\nonumber
		&\geq C\dfrac{-1}{\alpha_1}\left( (x_1-X_0)^{-\alpha_1}\left[-\frac{(x_1-y)^{\alpha_1+1}}{\alpha_1+1}\right]_{y=X_0}^{y=x_1} - (x_1-X_0)\right) \\
		\label{eq:251126c}
		&= C\frac{a'(x_1)(x_1-X_0)}{r(x_1)-\lambda} + \mathcal{O}(1)\xrightarrow[\lambda\to r(x_1)^+]{}+\infty, 
	\end{align}
	since $\alpha_1\to -1$ as $\lambda\to r(x_1)$.
	This proves the blow-up of $\Vert R_\gamma(\lambda; A_0)\Vert_{\mathcal{L}(C^0,L^1)}$ as $\lambda\to r(x_1)^+$. This completes the proof of \eqref{eq:251209a} and assertion 3.
	\medskip

	Next we prove assertion 4, starting with \eqref{eq:251209b}. Let $\gamma>0$ be given and recall that, under our assumptions, $\lambda_1^m=\lambda_1^1=\sigma(x_1)$. Then $\alpha_1=\alpha_1(\lambda)=\frac{\sigma(x_1)-\lambda}{a'(x_1)}\to 0$ as $\lambda\to \lambda_1^m$ and thus 
	\begin{equation*}
		\int_{x_0}^{x_1}R_\gamma(\lambda; A_0)\mathbbm{1}\dd x\geq \gamma \int_{X_1}^{x_1}\mathcal{E}_\lambda(y)\dd y\int_{x_0}^{x_1}\dfrac{1}{a(x)\mathcal{E}_\lambda(x)}\dd x\geq \frac{\gamma C}{-\alpha_1(\lambda)} \frac{1}{|x_1-x_0|^{\alpha_1(\lambda)}} \xrightarrow[\lambda\to (\lambda_1^m)^-]{}+\infty.
	\end{equation*}
	This proves \eqref{eq:251209b}. Next we turn to \eqref{eq:251209c}, the continuity of $R_0(\lambda; A_0)$ in the operator norm when $\lambda\to \lambda_1^m$. Recall the definition of $g_{\lambda, \lambda'}$ in \eqref{eq:251209d} and that $\alpha_1(\lambda_1^m)=\frac{\sigma(x_1)-\lambda_1^m}{a'(x_1)}=0$ under our assumptions; call $\alpha_1^+={\alpha}_1^+(\lambda)=\max\big(\alpha_1(\lambda),0\big)$. We have:
	\begin{align*}
		g_{\lambda, \lambda_1^m}(x) &= \frac{1}{a(x)}\int_{x}^{x_1} \left|e^{\int_x^y\frac{\sigma(z)-\lambda}{a(z)}\dd z}-e^{\int_x^y\frac{\sigma(z)-\lambda_1^m}{a(z)}\dd z}\right|\dd y\leq \frac{1}{a(x)}\int_{x}^{x_1}\int_{x}^y \frac{|\lambda-\lambda_1^m|}{a(z)}\dd z e^{\int_{x}^y\frac{\sigma(z)-\max(\lambda, \lambda_1^m)}{a(z)}\dd z}\dd y \\
		&=\frac{|\lambda-\lambda_1^m|}{a(x)} \int_{x}^{x_1} \big(\ln(x_1-x)-\ln(x_1-y)\big)\left(\dfrac{x_1-y}{x_1-x}\right)^{\alpha_1^+}e^{\mathcal{O}(1)} \dd y \\ 
		&\leq C\frac{|\lambda-\lambda_1^m|}{a(x)} \left(\dfrac{1}{1+\alpha_1^+}\right)^2 (x_1-x) = \mathcal{O}\big(|\lambda-\lambda_1^m|\big)
	\end{align*}
	which is uniformly bounded when $\lambda\to \lambda_1^m$. For the behavior near $x_0$, we proceed as in the proof of assertion 2. We obtain
	\begin{equation*}
		\lim_{\lambda\to (\lambda_1^m)^-}\int_{x_0}^{x_1}g_{\lambda, \lambda_1^m}(x)\dd x =0, 
	\end{equation*}
	and as a consequence
	\begin{equation*}	
		\int_{x_0}^{x_1}\big|\big[R_0(\lambda_1^m; A_0)-R_0(\lambda; A_0)\big] f(x)\big| \dd x \leq \Vert f\Vert_{C^0}\int_{x_0}^{x_1} g_{\lambda, \lambda_1^m}(x)\dd x=\Vert f\Vert_{C^0}  \mathcal{O}\big(|\lambda-\lambda_1^m|\big), \text{ if } \lambda<\lambda_1^m. 
	\end{equation*}
	By a similar argument,  
	\begin{equation*}	
		\int_{x_0}^{x_1}\big|\big[R_0(\lambda_1^m; A_0)-R(\lambda; A)\big] f(x)\big| \dd x \leq \Vert f\Vert_{C^0}\int_{x_0}^{x_1} g_{\lambda, \lambda_1^m}(x)\dd x= \Vert f\Vert_{C^0}  \mathcal{O}\big(|\lambda-\lambda_1^m|\big), \text{ if } \lambda>\lambda_1^m. 
	\end{equation*}
	This proves the continuity as $\lambda\to\lambda_1^m$, which completes the proof of \eqref{eq:251209c}, assertion 4.
	\medskip

	Finally we prove assertion 5. Let $\lambda =\lambda_1^1$, $f\in D(A)$ and $g=(\lambda I-A) f$. We assume that $g\in C^0([x_0, x_1])$; our goal is to prove $R_0(\lambda; A_0)g=f$. Recalling that $\mathcal{E}_\lambda(x)=e^{\mathcal{O}(1)}$ as $x\to x_1$ by Lemma \ref{lem:E_lambda}, we obtain by integrating $(\lambda I-A)f=g$ that  
	\begin{equation*}
		f(x) = \dfrac{-\int_{x_1}^xg(y)\mathcal{E}_\lambda(y)\dd y+ C}{a(x)\mathcal{E}_\lambda(x)} = R_0(\lambda; A_0) g(x)+  \dfrac{C}{a(x)\mathcal{E}_\lambda(x)},
	\end{equation*}
	for some constant $C$; hence 
	\begin{equation*}
		\dfrac{C}{a(x)\mathcal{E}_\lambda(x)} = f(x)- R_0(\lambda; A_0) g(x)\in L^1. 
	\end{equation*}
	But $\dfrac{C}{a(x)\mathcal{E}_\lambda(x)}=\dfrac{C}{x_1-x}e^{\mathcal{O}(1)}$ which cannot belong to $L^1$ when $C\neq 0$. We deduce that $C=0$ and therefore $f=R_0(\lambda; A_0)g$, which proves the claim.
	\medskip

	This ends the proof of Proposition \ref{prop:resolvent-A-C0-L1}
\end{proof}

We end with a Lemma that gives the local behavior of the extended resolvent applied to a continuous strictly positive function near the two ends of the interval. 
\begin{lemma}\label{lem:loc-est}
    Let $f\in C^0([x_0, x_1]) $ satisfy $f(x)>0$ for all $x\in [x_0, x_1]$, and $\lambda\in(\lambda_1^{ext}, +\infty)$. Call $\psi:=R_0(\lambda; A_0)f$.
\begin{enumerate}
    \item With $\alpha_0=\frac{\sigma(x_0)-\lambda}{a'(x_0)}$, we have in the neighborhood of $x_0$:
	    \begin{equation}\label{eq:260128f}
	    \psi(x) =  
	    \begin{cases}
		\frac{e^{\mathcal{O}(1)}}{|x-x_0|^{1+\alpha_0}} + o\left(\frac{1}{|x-x_0|^{1+\alpha_0}}\right), & \text{ if }\lambda\in\big(\lambda_1^0, r(x_0)\big)\Leftrightarrow \alpha_0\in (-1, 0), \\ 
		|\ln(|x-x_0|)|e^{\mathcal{O}(1)} + \mathcal{O}(1),%o\left(|\ln(|x-x_0|)|\right), 
		& \text{ if } \lambda= r(x_0)\Leftrightarrow \alpha_0=-1, \\ 
		\frac{K\star \varphi(x_0)}{\lambda-r(x_0)} + o(1), & \text{ if }  \lambda> r(x_0)\Leftrightarrow\alpha_0<-1.  
	    \end{cases}
	    \end{equation}
	\item $\psi(x)=e^{\mathcal{O}(1)}$ in the neighborhood of $x_1$, and more precisely, 
	    \begin{equation}\label{eq:260128g}
	    \psi(x)\xrightarrow[x\to x_1^-]{}\dfrac{f(x_1)}{\lambda-r(x_1)},
	    \end{equation}
\end{enumerate}
\end{lemma}
\begin{proof}
    1. We prove \eqref{eq:260128f}.
	First recall that $\alpha_0=\frac{\sigma(x_0)-\lambda}{a'(x_0)} = \frac{r(x_0)-a'(x_0)-\lambda}{a'(x_0)}=-1+\frac{r(x_0)-\lambda}{a'(x_0)}$.  If $\alpha_0\in(-1, 0) \Leftrightarrow \lambda\in \big(\lambda_1^0, r(x_0)\big)$, we have
	 \begin{align*}
		 \psi(x)&=\frac{1}{a(x)}\int_{x}^{x_1}f(y)e^{\int_{x}^y \frac{\sigma(z)-\lambda}{a(z)}\dd z}\dd y= \frac{1}{a(x)}\int_{x}^{X_0}e^{\mathcal{O}(1)} \left(\dfrac{y-x_0}{x-x_0}\right)^{\alpha_0}\dd y + C \\ 
		 &=\dfrac{e^{\mathcal{O}(1)}}{(x-x_0)} \left[\frac{1}{\alpha_0+1} \dfrac{(y-x_0)^{\alpha_0+1}}{(x-x_0)^{\alpha_0}}\right]_{y=x}^{y=X_0}+C = \dfrac{e^{\mathcal{O}(1)}}{(1+\alpha_0)(x-x_0)} \left(\dfrac{(X_0-x_0)^{\alpha_0+1}}{(x-x_0)^{\alpha_0}}- (x-x_0)\right)+C \\ 
		 &=\dfrac{e^{\mathcal{O}(1)}}{1+\alpha_0}\left(\dfrac{(X_0-x_0)^{\alpha_0+1}}{(x-x_0)^{\alpha_0+1}}-1\right)+C=\dfrac{e^{\mathcal{O}(1)}}{1+\alpha_0}\left(\left(\dfrac{X_0-x_0}{x-x_0}\right)^{\alpha_0+1}-1\right)+C.
	 \end{align*}
	 If $\alpha_0=-1 \Leftrightarrow \lambda=r(x_0)$, then by a similar computation 
	 \begin{equation*}
	     \psi(x)=e^{\mathcal{O}(1)} \left(\ln\left(\frac{X_0-x_0}{x-x_0}\right)+C\right).
	 \end{equation*}
	 Finally if $\alpha_0<-1\Leftrightarrow \lambda>r(x_0)$, we have $a(x)\mathcal{E}_\lambda(x) = |x-x_0|^{1+\alpha_0}e^{\mathcal{O}(1)}\to +\infty$ as $x\to x_0$, thus by L'Hospital rule
	 \begin{equation*}
		 \psi(x) \xrightarrow[x\to x_0]{} \dfrac{f(x_0)}{\lambda-r(x_0)}.
	 \end{equation*}
	 This completes the proof of \eqref{eq:260128f}.

	\medskip
	
	2. We prove \eqref{eq:260128g}.
	Remark that $\int_{x_1}^x f(y) \mathcal{E}_\lambda(y)\dd y\to 0$ and $a(x)\mathcal{E}_\lambda(x)=|x_1-x|^{1+\alpha_1}e^{\mathcal{O}(1)}\to 0$ as $x\to x_1$ (recall that $\lambda>\lambda_1^{ext} $ implies $\alpha_1> -1$). Then \eqref{eq:260128g} follows from L'Hospital rule.
	\medskip

    This completes the proof of Lemma \ref{lem:loc-est}. 
\end{proof}

%%%%%%%%%%%%%%%%%%%%%%%%%%%%%%%%%%%%%%%%%%%%%%%%%%%%%%%%%%%%%%%%%%%%%%%%%%%%%%%%%%%%%%%%%%%%%%%%%%%%%%%%%%%%%%%%%%%%%%%%%%%%%%%%%%%%%%%%%%%%%%%%%%%%%%%%%%%%%%%%%%%%%%%%%%%%
\subsection{Construction of principal eigenpairs}
\label{subsec:construction}
In this section we take advantage of our analysis in section \ref{sec:inversion-hyperbolic} to construct principal eigenvectors for the problem \eqref{eq:main}. {Throughout this section we will assume that Assumption \ref{ASS-20} holds true.}  We will frequently split the space $\mathbb{S}^1\backslash \{x_0, x_1\}$ into two connected subspaces, which we identify with $(x_0, x_1)$ and $(x_1, 1+x_0)$.
For $X_0\in (x_0, x_1)$ and $ X_0'\in (x_1, 1+x_0)$ we define recall the definition of $E_\lambda(x)$ by \eqref{eq:E_lambdaS1}:
\begin{equation*}
	\mathcal{E}_\lambda(x)= \mathcal{E}_\lambda^{X_0, X_0'}(x):=
	\begin{cases}
		\exp\left(\int_{X_0}^x \dfrac{\sigma(z)-\lambda}{a(z)}\dd z\right),  &\text{ if } x\in (x_0, x_1), \vspace{3pt}\\ 
		\exp\left(\int_{X_0'}^x \dfrac{\sigma(z)-\lambda}{a(z)}\dd z\right),  &\text{ if } x\in (x_1,1+ x_0), 
	\end{cases}
\end{equation*}
which is the transposition of \eqref{eq:E_lambda} in the context of a circle. When the context is clear, we will frequently write $\mathcal{E}_\lambda$ instead of $\mathcal{E}_\lambda^{X_0, X_0'}$, for simplicity. 
We first state our results on the operator on $\mathbb{S}^1$, which is a direct consequence of the results in section \ref{sec:inversion-hyperbolic}.
%%%%%%%%%%%%%%%%%%%%%%%%%%%%%%%%%%%%%%%%%%%%%%%%%%%%%%%%%%%%%%%%%%%%%%%%%%%%%%%%%%%%%%%%%%%%%%%%%%%%%%%%%%%%%%%
\begin{proposition}\label{prop:resolvent-circle}
Let $a(x)$ and $r(x)$ satisfy Assumption \ref{ASS-20}. 
	\begin{enumerate}
		\item Let $A$, $D(A)\subset L^1_{per}$ be the operator defined by \eqref{defA}. The resolvent set of $A$, $\rho(A)$, contains the ray $(\lambda_1^m, +\infty)$ and the formula
			\begin{equation}\label{eq:resolvent-L1}
				R(\lambda; A)f=(\lambda I -A)^{-1} f (x)= \dfrac{-1}{{a(x)\mathcal{E}_\lambda(x)}}\displaystyle \int_{x_1}^x f(y) \mathcal{E}_\lambda(y)\dd y
			\end{equation}
			holds true for any $f\in L^1_{per}$ and  $x\in \mathbb{S}^1\backslash\{x_0, x_1\}$. When $\lambda\to \lambda_1^m$ we have that $\Vert R(\lambda; A)\Vert_{\mathcal{L}(L^1_{per})} \to +\infty$, and the spectral bound of $A$ is $s(A)=\lambda_1^m$.
		\item Let $A_0$, $D(A_0)\subset C^0_{per}$ be the operator defined by \eqref{defA0}. The resolvent set of $A_0$, $\rho(A_0)$, contains the ray $(r^m, +\infty)$ and the formula 
			\begin{equation*}\label{eq:resolvent-C0}
			    \begin{aligned}
				    R(\lambda; A_0)f(x)=(\lambda I -A_0)^{-1} f (x)&:=
				    \begin{cases}
					    \dfrac{-1}{a(x)\mathcal{E}_\lambda(x)}\displaystyle \int_{x_1}^x f(y) \mathcal{E}_\lambda(y)\dd y,& \text{ if }  x\in \mathbb{S}^1\backslash\{x_0, x_1\} \vspace{3pt}\\
					    \dfrac{f(x_i)}{\lambda - r(x_i)}, &  \text{ if } x=x_i,\, i\in\{0, 1\},
				    \end{cases}
			    \end{aligned}
			\end{equation*}
 holds true for any $x\in \mathbb{S}^1$. When $\lambda\to r^m$ we have that $\Vert (\lambda I-A_0)^{-1}\Vert_{\mathcal{L}(C^0_{per})} \to +\infty$, and the spectral bound of $A_0$ is $s(A_0)=r^m$.
		\item Assume that $\lambda_1^0<\lambda_1^1 $ and let $\gamma_1, \gamma_2\in\mathbb{R}$ and $X_1\in (x_0, x_1)$, $X_1'\in (x_1, 1+x_0)$ be given. For $\lambda\in (\lambda_1^{ext}, \lambda_1^m)$, the two-parameter family of operators 
		    \begin{equation}\label{eq:resolvent-extended}
			    \resizebox{0.9\textwidth}{!}{$
			    R_{\gamma_1, \gamma_2}(\lambda; A_0) f(x):=\dfrac{\displaystyle -\int_{x_1}^x f(y)\mathcal{E}_\lambda(y)\dd y + \gamma_1 \int_{X_1}^{x_1} f(y)\mathcal{E}_\lambda(y)\dd y\mathbbm{1}_{(x_0, x_1)}-\gamma_2 \int^{X_1'}_{x_1} f(y)\mathcal{E}_\lambda(y)\dd y\mathbbm{1}_{(x_1, 1+x_0)} }{a(x)\mathcal{E}_\lambda(x)},
			    $}
		    \end{equation}
			is composed of bounded operators from $C^0_{per}$ into $L^1_{per}$ that are right-inverses for $(\lambda I-A)$, i.e. $(\lambda I-A)R_{\gamma_1, \gamma_2}(\lambda; A_0)f = f$ for all $f\in C^0_{per}$. $R_{\gamma_1, \gamma_2}(\lambda; A_0)$ depends continuously on $\lambda$ and $\gamma_1, \gamma_2$. If moreover $\gamma_1\geq 0$ and $\gamma_2\geq 0$ then $R_{\gamma_1, \gamma_2}(\lambda; A_0)$ is a positive operator, and the map $(\gamma_1, \gamma_2)\mapsto R_{\gamma_1, \gamma_2}(\lambda; A_0)$ is increasing for any $\lambda\in (\lambda_1^{ext}, \lambda_1^m)$. Finally, the map $\lambda\mapsto R_{0}(\lambda; A_0):=R_{0, 0}(\lambda; A_0)$ is defined for $\lambda\in(\lambda_1^{ext}, \lambda_1^m]$, decreasing,  and satisfies 
		    \begin{equation*}
			    \lim_{\lambda\to(\lambda_1^{m})^-} \Vert R_0(\lambda; A_0)-R_0(\lambda_1^m; A)\Vert_{\mathcal{L}(C^0_{per}, L^1_{per})} = 0
		    \end{equation*}
			 and 
			 \begin{equation*}
				 \lim_{\lambda\to(\lambda_1^{m})^+} \Vert R_0(\lambda_1^m; A)-R(\lambda; A)\Vert_{\mathcal{L}(C^0_{per}, L^1_{per})} = 0.
		    \end{equation*}
	    \item  Suppose $\lambda=\lambda_1^1$, and let $f\in D(A)$. If $(\lambda I-A)f=g$ for some $g\in C^0_{per}$, then $f=R_0(\lambda; A_0)g$. In particular,  $R_0(\lambda_1^1; A_0)$ is also a left inverse of $\lambda_1^1 I - A_0$, meaning that $R_0(\lambda_1^1; A_0)(\lambda_1^1 I-A_0)f=f$ for any $f\in  D(A_0)$.
	\end{enumerate}
\end{proposition}
%%%%%%%%%%%%%%%%%%%%%%%%%%%%%%%%%%%%%%%%%%%%%%%%%%%%%%%%%%%%%%%%%%%%%%%%%%%%%%%%%%%%%%%%%%%%%%%%%%%%%%%%%%%%%%%
\noindent The proof of Proposition \ref{prop:resolvent-circle} is a direct adaptation of Propositions \ref{prop:resolvent-A-L1}, \ref{prop:resolvent-A-C0} and \ref{prop:resolvent-A-C0-L1}, and we omit it. 
\medskip

With this definition of $R(\lambda; A) $ and $R_0(\lambda; A)$, we define the operator $T_\lambda:L^1_{per}\to L^1_{per}$ as 
\begin{subequations} \label{eq:Tlambda-abstract}
	\begin{align}
		T(\lambda)\varphi &:= R(\lambda; A) K\star \varphi , & \text{ if }& \lambda >\lambda_1^m, \\ 
		T(\lambda)\varphi &:= R_0(\lambda; A) K\star \varphi ,&  \text{ if }& \lambda \in (\lambda_1^{ext}, \lambda_1^m]. 
	\end{align}
\end{subequations}
 Note that $T_\lambda$ is well defined because the operator $B\varphi=K\star \varphi$ maps $L^1_{per}$ into $C^0_{per}$ continuously. We have the formula: 
    \begin{equation}\label{eq:Tlambda-explicit}
	    T(\lambda) \varphi= \dfrac{\displaystyle -\int_{x_1}^x (K\star \varphi)(y) \mathcal{E}_\lambda(y)\dd y}{\displaystyle a(x)\mathcal{E}_\lambda(x)}, \text{ for any } x\not\in \{x_0, x_1\} ,
    \end{equation} 
    as an immediate consequence of \eqref{eq:resolvent-L1} and \eqref{eq:resolvent-extended}; \eqref{eq:Tlambda-explicit} holds true for any $\lambda\in (\lambda_1^{ext}, +\infty)$. 
    Following \eqref{eq:resolvent-extended}, given $\gamma_1\geq 0$ and $\gamma_2\geq 0$,  we also define 
    \begin{equation} \label{eq:Tgammalambda}
	    T_{\gamma_1, \gamma_2}(\lambda)\varphi := R_{\gamma_1, \gamma_2}(\lambda; A) K\star \varphi ,  \text{ if } \lambda\in (\lambda_1^{ext}, \lambda_1^m). 
    \end{equation}

As we will see, these operators are bounded, strongly positive and compact; hence each has  a principal eigenpair. %Comparing this eigenvalue with  $1$ will decide whether \eqref{eq:system} has a regular or concentrated solution.

%We follow  Deimling \cite{Deimling-1985} for the notions of positivity, strong positivity etc. 
The following result is concerned with the operator $T(\lambda)$.
\begin{proposition}\label{prop:eigen-0}
	{Let Assumption \ref{ASS-20} hold true and}
	 $\lambda>\lambda_1^{ext}$ be given. The operator $T(\lambda):L^1_{per}\to L^1_{per}$ defined in \eqref{eq:Tlambda-abstract} is  bounded, irreducible and compact. The spectral radius 
	\begin{equation*}\label{eq:eigen-0}
		r\big(T(\lambda)\big)=:\kappa(\lambda)>0
	\end{equation*}
	is a simple eigenvalue associated with a positive eigenvector $\varphi_1^\lambda$ ; it is the only eigenvalue associated with a nonnegative eigenvector, and the eigenvector is a quasi-interior point of $L^1_{per,+}$. 

	The map $\lambda\mapsto \kappa(\lambda)$ is continuous and strictly decreasing on $\big(\lambda_1^{ext}, +\infty\big)$.  When $\lambda\to+\infty$ we have that $\kappa(\lambda)\to 0$, and when $\lambda\to \lambda_1^{ext}$ we have $\kappa(\lambda)\to+\infty$.
\end{proposition}
%%%%%%%%%%%%%%%%%%%%%%%%%%%%%%%%%%%%%%%%%%%%%%%%%%%%%%%%%%%%%%%%%%%%%%%%%%%%%%%%%%%%%%%%%%%%%%%%%%%%%%%%%%%%%%%
\begin{proof}
    The fact that $T(\lambda)$ defined by \eqref{eq:Tlambda-abstract} is a bounded operator and formula \eqref{eq:Tlambda-explicit} are a direct consequence of Proposition \ref{prop:resolvent-A-C0-L1}. We also get the continuity of the map $\lambda\mapsto T(\lambda)$ in the operator norm for $\lambda\in(\lambda_1^{ext}, +\infty)$ and the fact that it is monotone decreasing as a function of $\lambda$.
    \medskip

	We show that $T(\lambda)$ is irreducible. As we have seen before  $\varphi:=T(\lambda)(\psi)$ is continuous on $\mathbb{S}^1\backslash\{x_0\}$ and strictly positive by \eqref{eq:Tlambda-explicit} (recall that $a(x)>0$ for $x\in(x_0, x_1)$ and $a(x)<0$ for $x\not\in [x_0, x_1]$). Thus for any {nontrivial $f\in L^\infty_+(\mathbb{S}^1)$} we have 
	\begin{equation*}
		\langle f, T(\lambda)\psi\rangle=\int_{\mathbb{S}^1} f(x)\varphi(x)\dd x >0.
	\end{equation*}
	Thanks to Proposition \ref{prop:charac-irred}, this proves that $T(\lambda)$ is irreducible. 
    \medskip

	Next we prove that $T(\lambda)$ is compact. Let $\psi_n\in L^1$ be a bounded sequence, $\varphi_n:=T(\lambda)\psi_n$. We have:
	\begin{equation*}
	    \big|K\star \psi_n(x)-K\star \psi_n(y)\big|\leq \Vert \psi_n\Vert_{L^1} \sup_{z\in \mathbb{S}^1}\big|K(x, z)-K(y, z)|, 
	\end{equation*}
	hence the oscillation rate of $\psi_n$ is uniformly bounded; it follows from the Arzelà-Ascoli Theorem that $\Vert K\star \psi_n -\bar{\psi}\Vert_{C^0_{per}}\to 0$ for some $\bar\psi\in C^0_{per}$, up to the extraction of a subsequence (we implicitly change the sequence index in the rest of the argument). Let  
	\begin{equation*} 
	    \bar \varphi(x): = 
	    \dfrac{\displaystyle -\int_{x_1}^x \bar{\psi}(y) e^{\int_{X_0}^y \frac{\sigma(z)-\lambda}{a(z)}\dd z}\dd y}{\displaystyle a(x)e^{\int_{X_0}^x\frac{\sigma(z)-\lambda}{a(z)}\dd z}}, 
	\end{equation*}
	then
	\begin{align*}
	    \Vert \varphi_n-\bar{\varphi}\Vert_{L^1} &\leq \Vert K\star \psi_n -\bar{\psi}\Vert_{C^0_{per}}\int_0^1\underset{g(x)}{\underbrace{\dfrac{\displaystyle \int_{x_1}^x  e^{\int_{X_0}^y \frac{\sigma(z)-\lambda}{a(z)}\dd z}\dd y}{\displaystyle a(x)e^{\int_{X_0}^x\frac{\sigma(z)-\lambda}{a(z)}\dd z}}}}\dd x
	\end{align*}
	Remark that $g(x)=R(\lambda; A_0)\mathbbm{1}$; therefore $g(x) \in L^1_{per}$. This proves that $\varphi_n\to \bar \varphi$, hence $T(\lambda)$ is compact. 
	\medskip

	 From the above properties, it follows that $T(\lambda)$ satisfies the assumptions of the Jentzsch-Perron Theorem \cite[Corollary 4.2.14 p.273]{Meyer--Nieberg-1991}, which shows that $r\big(T(\lambda)\big)$ is an eigenvalue of algebraic multiplicity one, and the corresponding eigenspace is spanned by a nonnegative vector $\varphi_1^{\lambda}$ that is a quasi-interior point of $L^1_{per, +}$. By \cite[Corollary 4.2.15 p.273]{Meyer--Nieberg-1991}, we know moreover that $\varphi_1^\lambda$ is the only nonnegative eigenvector of $T(\lambda)$. 
	 It is then classical that $\lambda\mapsto r\big(T(\lambda)\big)$ is continuous (since $\lambda\mapsto T(\lambda)$ is continuous in operator norm). That $\lambda\mapsto \kappa(\lambda)=r\big(T(\lambda)\big)$ is nonincreasing, is a consequence of Gelfand's formula: since $T(\lambda)$ is a bounded operator,  
	 \begin{equation*}
		 \kappa(\lambda')=\lim_{n\to+\infty} \Vert T(\lambda')^n\Vert^{\frac{1}{n}}\geq \lim_{n\to +\infty} \Vert T(\lambda)^n\Vert^{\frac{1}{n}} = {\kappa(\lambda)}.
	 \end{equation*}
	 Let us show that $\kappa(\lambda)<\kappa(\lambda') $ if $\lambda'<\lambda$. Since $K\star \varphi_1^\lambda$ is a bounded continuous function it follows from \eqref{eq:251127b} in the proof of Proposition \ref{prop:resolvent-A-C0} that $ \varphi_1^\lambda(x)=e^{\mathcal{O}(1)}$ when $x$ is close to $x_1$, and we have (recalling $\alpha_0=\frac{\sigma(x_0)-\lambda}{a'(x_0)}<0$)
	 \begin{align*}
		 \varphi_1^\lambda(x)&=\frac{1}{\kappa(\lambda)a(x)}\int_{x}^{x_1}K\star \varphi_1^\lambda(y)e^{\int_{x}^y \frac{\sigma(z)-\lambda}{a(z)}\dd z}\dd y= \frac{1}{\kappa(\lambda)a(x)}\int_{x}^{X_0}e^{\mathcal{O}(1)} \left(\dfrac{y-x_0}{x-x_0}\right)^{\alpha_0}\dd y + C \\ 
		 &=\dfrac{e^{\mathcal{O}(1)}}{(x-x_0)} \left[\frac{1}{\alpha_0+1} \dfrac{(y-x_0)^{\alpha_0+1}}{(x-x_0)^{\alpha_0}}\right]_{y=x}^{y=X_0}+C = \dfrac{e^{\mathcal{O}(1)}}{(1+\alpha_0)(x-x_0)} \left(\dfrac{(X_0-x_0)^{\alpha_0+1}}{(x-x_0)^{\alpha_0}}- (x-x_0)\right)+C \\ 
		 &=\dfrac{e^{\mathcal{O}(1)}}{1+\alpha_0}\left(\dfrac{(X_0-x_0)^{\alpha_0+1}}{(x-x_0)^{\alpha_0+1}}-1\right)+C=\dfrac{e^{\mathcal{O}(1)}}{1+\alpha_0}\left(\left(\dfrac{X_0-x_0}{x-x_0}\right)^{\alpha_0+1}-1\right)+C,
	 \end{align*}
	 when $\alpha_0\neq -1$ and $x\in (x_0,X_0)$, where $C>0$ is a positive constant; if $\alpha_0=-1$ then  
	 \begin{equation*}
		 \varphi_1^\lambda(x)=e^{\mathcal{O}(1)} \ln\left(\frac{X_0-x_0}{x-x_0}\right)+C;
	 \end{equation*}
	 if $x\in (X_0', 1+x_0)$ (with $X_0'\in(x_1, 1+x_0)$) we have
	 \begin{equation*}
		 \varphi_1^\lambda(x)=\begin{cases}
			 \dfrac{e^{\mathcal{O}(1)}}{1+\alpha_0}\left(\left(\dfrac{1+x_0-X_0'}{1+x_0-x}\right)^{\alpha_0+1}-1\right)+C & \text{ if } \alpha_0\neq -1 \\
			 \varphi_1^\lambda(x)=e^{\mathcal{O}(1)} \ln\left(\frac{1+x_0-X_0}{1+x_0-x}\right)+C &\text{ if } \alpha_0=-1.
		 \end{cases}
	 \end{equation*}
	 Summarizing, we have 
	 \begin{equation}\label{eq:251127a}
		 \varphi_1^\lambda(x) = \mathcal{O}\big(\varphi_1^{\lambda'}(x)\big) \text{ in general,  and } \varphi_1^\lambda(x) = o\big(\varphi_1^{\lambda'}(x)\big) \text{ if } \varphi_1^{\lambda'}(x_0)=+\infty,
	 \end{equation}
	 for $x$ close to $x_0$ whenever $\lambda'<\lambda$. Thus there exists $\gamma>0$ such that $\gamma\varphi_1^\lambda\leq \varphi_1^{\lambda'}$. Choose
	 \begin{equation*}
		 \gamma:=\sup\{\gamma'>0\,:\, \gamma'\varphi_1^\lambda\leq \varphi_1^{\lambda'}\}\in(0, +\infty).
	 \end{equation*}
	 By \eqref{eq:251127a} there must exist $\bar{x}\in [x_0, 1+x_0]$ such that $\gamma\varphi_1^\lambda(\bar{x})=\varphi_1^{\lambda'}(\bar{x})$. Then we have
	 \begin{align*}
		 \kappa(\lambda')\varphi_1^{\lambda'}(\bar{x})-\kappa(\lambda)\gamma \varphi_1^\lambda(\bar{x}) &= {T(\lambda') \big(\varphi_1^{\lambda'}-\gamma\varphi_1^\lambda\big)(\bar{x})} + \gamma\underset{>0}{\underbrace{\big(T(\lambda')-T(\lambda)\big)\varphi_1^\lambda(\bar{x})}} >0, 
	 \end{align*}
	 indeed 
	 \begin{equation*}
		 \big(T(\lambda')-T(\lambda)\big)\varphi_1^{\lambda}(\bar x)=\frac{-1}{a(x)}\int_{x_1}^{\bar{x}} K\star \varphi_1^\lambda(y)\left(e^{-\int_y^{\bar{x}}\frac{\sigma(z)-\lambda'}{a(z)}\dd z}-e^{-\int_y^{\bar{x}}\frac{\sigma(z)-\lambda}{a(z)}\dd z}\right)\dd y>0.
	 \end{equation*}
	 This proves that $\kappa(\lambda')>\kappa(\lambda)$.
	\medskip

	We show the limit behavior of $\kappa(\lambda)$. Since $A$ generates a strongly continuous semigroup, there exists $(M, \omega)$ such that 
	\begin{equation*}
	    \Vert (\lambda I -A)^{-1}\Vert \leq \dfrac{M}{\lambda-\omega}, \qquad ^{\forall} \lambda>\omega,
	\end{equation*}
	hence
	\begin{equation*}
	    0\leq \lim_{\lambda\to+\infty} \kappa(\lambda) \leq \lim_{\lambda\to+\infty} \Vert (\lambda I-A)^{-1}\Vert \Vert B\Vert = 0.
	\end{equation*}
	That solves the question when $\lambda\to+\infty$. When $\lambda\to \lambda_1^{ext}$,  
 we distinguish two cases.

	a) $r(x_1)\leq\lambda_1^0=\sigma(x_0)$.  For $\lambda>\lambda_1^{ext}=\lambda_1^0$, we have:
	\begin{align*}
	    \kappa(\lambda) &= \int_{\mathbb{S}^1}T(\lambda) \varphi_1 \geq \underset{\int \varphi = 1}{\inf_{\varphi\geq 0}} \int T(\lambda)\varphi  = \underset{\int \varphi = 1}{\inf_{\varphi\geq 0}}
	    \int_{\mathbb{S}^1}\dfrac{\displaystyle -\int_{x_1}^x K\star \psi(y) e^{\int_{X_0}^y \frac{\sigma(z)-\lambda}{a(z)}\dd z}\dd y}{\displaystyle a(x)e^{\int_{X_0}^x\frac{\sigma(z)-\lambda}{a(z)}\dd z}}\dd x \\
	    &\geq \left(\inf_{y,z\in\mathbb{S}^1}K(y,z)\right)\int_{\mathbb{S}^1}\dfrac{\displaystyle -\int_{x_1}^x  e^{\int_{X_0}^y \frac{\sigma(z)-\lambda}{a(z)}\dd z}\dd y}{\displaystyle a(x)e^{\int_{X_0}^x\frac{\sigma(z)-\lambda}{a(z)}\dd z}}\dd x \xrightarrow[\lambda\to \lambda_1^{0}]{}+\infty,
	\end{align*}
	by the same computation as the one leading to \eqref{eq:251125c}.

b) $r(x_1)>\lambda_1^0=\sigma(x_0)$.  For $\lambda>\lambda_1^{ext}=r(x_1)$, we have, similarly:
	\begin{align*}
	    \kappa(\lambda) &= \int_{\mathbb{S}^1}T(\lambda) \varphi_1 \geq \underset{\int \varphi = 1}{\inf_{\varphi\geq 0}} \int T(\lambda)\varphi   \geq \left(\inf_{y,z\in\mathbb{S}^1}K(y,z)\right)\int_{\mathbb{S}^1}\dfrac{\displaystyle -\int_{x_1}^x  e^{\int_{X_0}^y \frac{\sigma(z)-\lambda}{a(z)}\dd z}\dd y}{\displaystyle a(x)e^{\int_{X_0}^x\frac{\sigma(z)-\lambda}{a(z)}\dd z}}\dd x \xrightarrow[\lambda\to r(x_1)]{}+\infty.
	\end{align*}
	by the same argument as in \eqref{eq:251126c}.
	\medskip

	This  finishes the proof of Proposition \ref{prop:eigen-0}
\end{proof}
%%%%%%%%%%%%%%%%%%%%%%%%%%%%%%%%%%%%%%%%%%%%%%%%%%%%%%%%%%%%%%%%%%%%%%%%%%%%%%%%%%%%%%%%%%%%%%%%%%%%%%%%%%%%%%%

Our next result is concerned with the properties of $T_{\gamma_1, \gamma_2}(\lambda)$ when $\lambda_1^0<\lambda_1^1$. 
%%%%%%%%%%%%%%%%%%%%%%%%%%%%%%%%%%%%%%%%%%%%%%%%%%%%%%%%%%%%%%%%%%%%%%%%%%%%%%%%%%%%%%%%%%%%%%%%%%%%%%%%%%%%%%%
\begin{proposition}\label{prop:eigen-extended}
	{Let Assumption \ref{ASS-20} hold true and} suppose that $\lambda_1^0<\lambda_1^1$;  let $\lambda>\lambda_1^{ext}$ and $\gamma_1\geq 0$, $\gamma_2\geq 0$ be given. The operator $T_{\gamma_1, \gamma_2}(\lambda):L^1_{per}\to L^1_{per}$ defined in \eqref{eq:Tgammalambda} is  bounded, irreducible and compact. The spectral radius 
	\begin{equation*}\label{eq:eigen-extended}
		r\big(T_{\gamma_1, \gamma_2}(\lambda)\big)=:\kappa(\lambda; \gamma_1, \gamma_2)>0
	\end{equation*}
	is a simple eigenvalue associated with a positive eigenvector $\varphi_1^{\lambda; \gamma_1, \gamma_2}$; it is the only eigenvalue associated with a nonnegative eigenvector, and the eigenvector is a quasi-interior point of $L^1_{per,+}$. 

	For fixed $\lambda\in (\lambda_1^{ext}, \lambda_1^1)$ the map $(\gamma_1, \gamma_2)\mapsto \kappa(\lambda; \gamma_1, \gamma_2)$ is {strictly increasing with respect to $\gamma_1$ and $\gamma_2$}. 
	For fixed $(\gamma_1, \gamma_2)\geq (0, 0)$, 
	the map $\lambda\mapsto \kappa(\lambda; \gamma_1, \gamma_2)$ is continuous and superconvex on {$\big(\lambda_1^{ext},\lambda_1^1\big)$}. When  $\lambda\to \lambda_1^{ext}$ we have $\kappa(\lambda; \gamma_1, \gamma_2)\to+\infty$, and if $\gamma_1>0$ or $\gamma_2>0$, we have $\kappa(\lambda; \gamma_1, \gamma_2)\to+\infty$ as {$\lambda\to (\lambda_1^{1})^-$}. The map $[0, +\infty)^2\ni(\gamma_1, \gamma_2)\mapsto \kappa(\lambda; \gamma_1, \gamma_2)$ is continuous. 

	Finally, for any fixed $\gamma_1\geq 0$ and $\lambda\in(\lambda_1^{ext}, \lambda_1^1)$  we have $\kappa(\lambda; \gamma_1, \gamma_2)\to +\infty$ as $\gamma_2\to+\infty$, and similarly for any fixed $\gamma_2\geq 0$ we have $\kappa(\lambda; \gamma_1, \gamma_2)\to +\infty$ as $\gamma_1\to+\infty$. 
\end{proposition}
\begin{remark}
    By \textit{superconvex} we mean that $\lambda\mapsto \ln \kappa(\lambda; \gamma_1, \gamma_2)$ is convex. 
\end{remark}
%%%%%%%%%%%%%%%%%%%%%%%%%%%%%%%%%%%%%%%%%%%%%%%%%%%%%%%%%%%%%%%%%%%%%%%%%%%%%%%%%%%%%%%%%%%%%%%%%%%%%%%%%%%%%%%
\begin{proof}
	The proofs of boundedness, positivity, irreducibility and compactness of $T_{\gamma_1, \gamma_2}(\lambda)$ follow essentially the same steps as in Proposition \ref{prop:eigen-0}; so does the existence and uniqueness of a principal eigenpair $\big(\kappa(\lambda; \gamma_1, \gamma_2), \varphi_1^{\lambda; \gamma_1, \gamma_2}\big)$, and the continuity of $\lambda\mapsto \kappa(\lambda; \gamma_1, \gamma_2)$. We omit the details for concision. The {strict} monotony of $(\gamma_1, \gamma_2)\mapsto \kappa(\lambda; \gamma_1, \gamma_2)$ is a consequence of the {strict} monotony of  $(\gamma_1, \gamma_2)\mapsto T_{\gamma_1, \gamma_2}(\lambda )$.

    Let us prove that $\lambda\mapsto\kappa(\lambda; \gamma_1, \gamma_2)$ is superconvex. Let $\varphi\in L^1_+$ be given; since $f(y):=K\star \varphi(y)$ is continuous, the integral $\int_{x_1}^x f(y)\mathcal{E}_\lambda(y)\dd y$ can be computed as an improper Riemann integral. We claim that the limit
    \begin{equation}\label{eq:260109a}
	T_1(x):=\dfrac{-\int_{x_1}^x f(y)\mathcal{E}_\lambda(y)\dd y}{a(x)\mathcal{E}_\lambda(x)}=\lim_{n\to+\infty} S_n(x) 
    \end{equation}
    where
    \begin{equation}\label{eq:260109b}
	S_n(x):= \frac{1}{a(x)\mathcal{E}_\lambda(x)}
	\begin{cases} 
	    \sum_{k=1}^{n-2} \mathbbm{1}_{ [x_k^n, x_{k+1}^n)}(x)\frac{x_1-x_0}{n}\sum_{i=k}^{n-2}f(x_{i+1}^n)\mathcal{E}_\lambda(x_{i+1}^n) & \text{ if } x\in [x_0, x_1), \\ 
	    -\sum_{k=1}^{n-2} \mathbbm{1}_{ [\bar{x}_k^n, \bar{x}_{k+1}^n)}(x)\frac{1+x_0-x_1}{n}\sum_{i=1}^{k}f(\bar{x}_i^n)\mathcal{E}_\lambda(\bar{x}_i^n) & \text{ if } x\in [x_1, x_0+1), 
	\end{cases}
    \end{equation}
	$x_k^n:=x_0+k(x_1-x_0)/n$ and $\bar{x}_n := x_1+k(x_0+1-x_1)/n$, holds in $L^1_{per}$. 
	Clearly $S_n(x)\xrightarrow[n\to+\infty]{}T_1(x)$ almost everywhere. Let us check that $S_n$ is dominated by a $L^1$ function. 
	The function $S_n(x)$ is uniformly bounded for $x\not\in\{x_0, x_1\}$; there remains to prove that $S_n(x)$ is locally in $L^1$ in a neighborhood of $x_0$ and $x_1$, respectively.
	By \eqref{eq:251204d} we have $\mathcal{E}_\lambda=(x_1-x)^{\alpha_1} e^{\mathcal{O}(1)}$ close to $x_1$, and $\alpha_1=\frac{\sigma(x_1)-\lambda}{a'(x_1)}\in (-1, 0)$; hence $\big(a(x)\mathcal{E}_\lambda(x)\big)^{-1}=(x_1-x)^{-1-\alpha_1}e^{\mathcal{O}(1)}\in L^1_{loc}$. 
	Similarly,  $f(y)\mathcal{E}_\lambda(y)=(x_1-x)^{\alpha_1}e^{\mathcal{O}(1)}\in L^1$, so that $\int_{x_1}^xf(y)\mathcal{E}_\lambda(y)\dd y\in L^\infty$. Thus $S_n(x)$ is bounded by $C\big(a(x)\mathcal{E}_\lambda(x)\big)^{-1}=(x_1-x)^{-1-\alpha_1}e^{\mathcal{O}(1)}$ for a sufficiently large $C$. In what follows the symbol $C$ will denote a large positive constant whose value may change from line to line but  that is independent of $n$. In the vicinity of $x_0$, we have for $x>x_0$:
    \begin{align*}
	S_n(x)& =\frac{\mathbbm{1}_{x\geq x_1^n}}{a(x)\mathcal{E}_\lambda(x)}\sum_{k=1}^{n-2} \mathbbm{1}_{ [x_k^n, x_{k+1}^n)}(x)\frac{x_1-x_0}{n}\sum_{i=k}^{n-2}f(x_{i+1}^n)\mathcal{E}_\lambda(x_{i+1}^n) \\ 
	&\leq C (x-x_0)^{-1-\alpha_0}\left(1+\sum_{k=1}^{n/2} \mathbbm{1}_{ [x_k^n, x_{k+1}^n)}(x)\frac{x_1-x_0}{n}\sum_{i=k}^{n/2}(x_{i+1}^n-x_0)^{\alpha_0}\right)\\
	&\leq C (x-x_0)^{-1-\alpha_0}\left(1+\int_{x_{k_0}^n}^{(x_1+x_0)/2}(y-x_0)^{\alpha_0}\dd y\right) , 
    \end{align*}
    where $k_0=k_0(x):=\left\lfloor n\frac{x-x_0}{x_1-x_0}\right\rfloor$ so that $k_0$ is the largest integer $k$ satisfying $x_{k}^n\leq x$.  Recall that  $\alpha_0=\frac{\sigma(x_0)-\lambda}{a'(x_0)}<0$; at this point we easily check that $S_n(x)$ is dominated by a $L^1$ function in the vicinity of $x_0$ if either $1+\alpha_0>0$ or $1+\alpha_0=0$. If $1+\alpha_0<0$, we continue the computation to find that
    \begin{align*}
	S_n(x)&\leq C(x-x_0)^{-1-\alpha_0}\left(1+ \frac{1}{\alpha_0+1}\left[\left(\frac{x_1+x_0}{2}-x_0\right)^{\alpha_0+1}-(x_{k_0}^n-x_0)^{\alpha_0+1}\right]\right) \\
	&\leq C\left((x-x_0)^{-1-\alpha_0} + \left(\frac{x_{k_0}^n-x_0}{x-x_0}\right)^{\alpha_0+1}\right)\leq C\left((x-x_0)^{-1-\alpha_0} + \left(\frac{x_{k_0}^n-x_0}{x_{k_0+1}^n-x_0}\right)^{\alpha_0+1}\right) \\ 
	&= C\left((x-x_0)^{-1-\alpha_0} + \left(\frac{k_0}{k_0+1}\right)^{\alpha_0+1}\right) = C\left((x-x_0)^{-1-\alpha_0}+1\right)
    \end{align*}
    which is in $L^1$. We obtain a similar estimate for $x<x_0$ close to $x_0$ by using the second line in \eqref{eq:260109b}. We conclude that $S_n$ is dominated by a $L^1$ function, and by Lebesgue's dominated convergence theorem $S_n\to T_1$ in $L^1_{per}$.

    Now $S_n(x)$ is a finite sum of superconvex functions, hence superconvex \cite[Lemma 2.2 (b)]{Kato-1982}, and \eqref{eq:260109a} proves that $\frac{-\int_{x_1}^x f(y)\mathcal{E}_\lambda(y)\dd y}{a(x)\mathcal{E}_\lambda(x)}$ is superconvex as a function of $\lambda$ \cite[Lemma 2.2 (c)]{Kato-1982}. We proceed similarly with 
    \begin{equation*}
	T_2(x):= \dfrac{\displaystyle  \gamma_1 \int_{X_1}^{x_1} f(y)\mathcal{E}_\lambda(y)\dd y\mathbbm{1}_{(x_0, x_1)}(x)}{a(x)\mathcal{E}_\lambda(x)}
	\text{ and }
	T_3(x):=\dfrac{\displaystyle -\gamma_2 \int^{X_1'}_{x_1} f(y)\mathcal{E}_\lambda(y)\dd y\mathbbm{1}_{(x_1, 1+x_0)}(x) }{a(x)\mathcal{E}_\lambda(x)}
    \end{equation*} 
    to reach the superconvexity of $T_{\gamma_1, \gamma_2}(\lambda)\varphi = T_1+T_2+T_3$. Since this is true for all $\varphi\in L^1_{per, +}$, the operator family $T_{\gamma_1, \gamma_2}(\lambda)$ is superconvex. Now, by the generalized Kingman's theorem \cite[Theorem 2.5]{Kato-1982}, we obtain that 
    \begin{equation*}
	\lambda\mapsto r\big(T_{\gamma_1, \gamma_2}(\lambda)\big) = \kappa(\lambda; \gamma_1, \gamma_2)
    \end{equation*}
    is superconvex, which is the desired conclusion. 
    \medskip

    That  $\kappa(\lambda; \gamma_1, \gamma_2)\to+\infty$ when  $\lambda\to \lambda_1^{ext}$ is a consequence of the monotony of $(\gamma_1, \gamma_2)\mapsto \kappa(\lambda; \gamma_1, \gamma_2)$ and the corresponding property in Proposition \ref{prop:eigen-0}. 
	Next, recall the definition of $\alpha_0(\lambda)=\frac{\sigma(x_0)-\lambda}{a'(x_0)}$ and $\alpha_1(\lambda)=\frac{\sigma(x_1)-\lambda}{a'(x_1)}$ from \eqref{eq:alphas}. 
	If $\gamma_1>0$, we have 
    \begin{align*}
	\Vert T_{\gamma_1, \gamma_2}(\lambda)\varphi\Vert_{L^1_{per}}&\geq \gamma_1\int_{\mathbb{S}^1}\dfrac{\int_{X_1}^{x_1} K\star \varphi(y)\mathcal{E}_\lambda(y)\dd y\mathbbm{1}_{(x_0, x_1)}(x)}{a(x)\mathcal{E}_\lambda(x)}\dd x\\
	&\geq C\int_{\mathbb{S}^1}\varphi(x)\dd x\int_{X_1}^{x_1}(x_1-x_0)^{-\alpha_1-1}\dd x \int_{X_1}^{x_1}(y-x_1)^{\alpha_1}\dd y\\
	    &\geq C \frac{-1}{\alpha_1}(X_1-x_1)^{-\alpha_1}\xrightarrow[\lambda\to \lambda_1^{1}]{} +\infty, 
    \end{align*}
	for any non-trivial $\varphi\in L^1_{per, +}$, because $\alpha_1=\alpha_1(\lambda)\to 0$ when $\lambda\to \lambda_1^{1}$. We proceed similarly when $\gamma_2>0$. We conclude that  whenever $\gamma_1>0$ or $\gamma_2>0$ we have
    \begin{equation*}
	    \kappa(\lambda; \gamma_1, \gamma_2)\geq \Vert T_{\gamma_1, \gamma_2}(\varphi_1^{\lambda; \gamma_1, \gamma_2}) \Vert_{L^1} \xrightarrow[\lambda\to {\lambda_1^{1}}]{}+\infty.
    \end{equation*}
    Finally, let $\gamma_2>0$ be fixed;  integrating the relation $T_{\gamma_1, \gamma_2}(\lambda)\varphi=\kappa(\lambda; \gamma_1, \gamma_2)\varphi$ with $\int\varphi=1$ we get
    \begin{align*}
	\kappa(\lambda; \gamma_1, \gamma_2)&=\int_{\mathbb{S}^1}T_{\gamma_1, \gamma_2}(\lambda)\varphi \geq \gamma_1 \int_{x_0}^{x_1} \frac{1}{a(x)\mathcal{E}_\lambda(x)}\dd x \int_{X_1}^{x_1}K\star\varphi(y) \mathcal{E}_\lambda(y)\dd y \\ 
	&\geq \gamma_1 \int_{x_0}^{x_1}\frac{1}{a(x)\mathcal{E}_\lambda(x)}\dd x \int_{X_1}^{x_1} \inf_{x\in\mathbb{S}^1}K(x, y)\mathcal{E}_\lambda(y)\dd y \xrightarrow[]{\gamma_1\to+\infty}+\infty.
    \end{align*}
    The above computation makes sense for $\lambda\in(\lambda_1^{ext}, \lambda_1^1)$ because  $\alpha_1<0$ and $\alpha_0<0$ so $\frac{1}{a(x)\mathcal{E}_\lambda(x)}$ is in $L^1(x_0,x_1)$;  and $\alpha_1> -1$ so $\mathcal{E}_\lambda(y)$ is in $L^1(X_1, x_1)$. A similar computation shows 	$\kappa(\lambda; \gamma_1, \gamma_2) \xrightarrow[]{\gamma_2\to+\infty}+\infty$ for fixed $\gamma_1\geq 0$.
    This concludes the proof of Proposition \ref{prop:eigen-extended}.
\end{proof}
The connection between the function $\kappa(\lambda)$ and the existence of an eigenpair (singular or not) for $A+B$ will be given in the next result. 

%%%%%%%%%%%%%%%%%%%%%%%%%%%%%%%%%%%%%%%%%%%%%%%%%%%%%%%%%%%%%%%%%%%%%%%%%%%%%%%%%%%%%%%%%%%%%%%%%%%%%%%%%%%%%%%
\begin{theorem}[Regular eigenpairs]\label{thm:eigenpairs-L1}
	{Let Assumption \ref{ASS-20} hold true and}  $\kappa(\lambda)$ be the function defined on $(\lambda_1^{ext}, +\infty)$ in Proposition \ref{prop:eigen-0}. The equation $\kappa(\lambda)=1$ has a unique solution $\lambda_1^*$.
    \begin{enumerate}
	\item Suppose that $\lambda_1^*\geq \lambda_1^1$,  then \eqref{eq:main} admits a unique normalized eigenpair $(\lambda, \varphi)$ satisfying $\varphi\in L^1_{per, +}$, which is given by $\lambda=\lambda_1^*$ and $\varphi=\varphi_1^{\lambda_1^*}$. If moreover $\lambda>\max(r(x_0), \lambda_1^1)$, then $\varphi_1^{\lambda_1^*}\in C^0_{per}$.
	\item Suppose that $\lambda_1^*<\lambda_1^1$. Then any principal eigenpair  $(\lambda, \varphi)$ of \eqref{eq:main} with $\varphi\in L^1_{per} $ satisfies $\lambda\in [\lambda_1^*, \lambda_1^1)$, and for each $\lambda\in [\lambda_1^*, \lambda_1^1)$, there exists at least one principal eigenpair $(\lambda, \varphi)$ of \eqref{eq:main}. 

	    More precisely, for $\lambda=\lambda_1^*$, there exists a unique $\varphi=\varphi_1^{\lambda_1^*}$ such that $(\lambda, \varphi)$ is a normalized principal eigenpair of \eqref{eq:main}; for any $\lambda\in (\lambda_1^*, \lambda_1^1)$, there exists a real number $\overline{\gamma}(\lambda)>0$ and a continuous one-parameter family of normalized principal eigenvectors given by $\varphi_1^{\lambda; \gamma, \gamma_2(\lambda, \gamma)}=:\varphi_1^{\lambda, \gamma}$, where $\gamma_2(\lambda, \gamma)$ is defined as the unique solution of 
	    \begin{equation*}
		\kappa\big(\lambda; \gamma, \gamma_2(\lambda, \gamma)\big)=1 \text{ for all } \gamma\in [0, \overline{\gamma}(\lambda)].
	    \end{equation*}
	    The family  $\varphi_1^{\lambda, \gamma}$ is composed of distinct eigenvectors: for $(\lambda, \gamma)\neq (\lambda', \gamma')$ we have $\varphi_1^{\lambda, \gamma}\neq \varphi_1^{\lambda', \gamma'}$.
	    Moreover, any normalized principal eigenvector $\varphi$ corresponding to the principal eigenvalue $\lambda$ satisfies $\varphi=\varphi_1^{\lambda, \gamma}$ for some $\gamma\in [0, \overline{\gamma}(\lambda)]$. Finally
	    \begin{equation}\label{eq:260126a}
		\overline{\gamma}(\lambda)\xrightarrow[\lambda\to\lambda_1^*]{}0 \text{ and } \overline{\gamma}(\lambda)\xrightarrow[\lambda\to\lambda_1^1]{}0.
	    \end{equation}
    \end{enumerate}
\end{theorem}
%%%%%%%%%%%%%%%%%%%%%%%%%%%%%%%%%%%%%%%%%%%%%%%%%%%%%%%%%%%%%%%%%%%%%%%%%%%%%%%%%%%%%%%%%%%%%%%%%%%%%%%%%%%%%%%
We start with some a priori estimates on the ``supersolutions'' of the principal eigenvalue problem. 
\begin{lemma}\label{lem:supercrit}
	{Suppose that Assumption \ref{ASS-20} holds true. Let} $(\lambda, \varphi)\in\mathbb{R}\times L^1_{per}$ with $\varphi\geq 0$ a.e. and $\varphi\in W^{1, 1}_{loc}\big(\mathbb{S}^1\backslash\{x_0, x_1\}\big)$ solve the differential inequality 
    \begin{equation}\label{eq:260128h}
	a(x)\varphi'+r(x)\varphi-\lambda \varphi + K\star \varphi \leq 0,
    \end{equation}
    then $\lambda> \lambda_1^{ext}$.
	If moreover $\varphi(x_1)>0$ and  there exists a solution $(\lambda_1^*, \varphi_1^*)$ to \eqref{eq:main} with $\varphi_1^*\geq 0$ and   
	\begin{equation}\label{eq:260129a}
	    {\varphi_1^*}(x)= \dfrac{-\int_{x_1}^xK\star {\varphi_1^*}(y) \mathcal{E}_{\lambda_1^*}(y)\dd y}{a(x)\mathcal{E}_{\lambda_1^*}(x)}, 
    \end{equation}
	then $\lambda\geq \lambda_1^*$. In the latter case, if $\lambda=\lambda_1^*$, then $\varphi\equiv k\varphi_1^*$ for some $k>0$.
\end{lemma}
\begin{remark}
	The existence of a principal eigenpair $(\lambda_1^*, \varphi_1^*)$ with $\varphi_1^*$ solving \eqref{eq:260129a} will follow from Theorem \ref{thm:eigenpairs-L1}. Since the proof of Theorem \ref{thm:eigenpairs-L1} is not complete yet, we still need to include the existence as an assumption of Lemma \ref{lem:supercrit} at this point.
\end{remark}
\begin{proof}[Proof of Lemma \ref{lem:supercrit}]
    We get from \eqref{eq:260128h} that
	\begin{equation}\label{eq:260129c}
		\big(a(x)\mathcal{E}_\lambda(x)\varphi(x)\big)'=\left(a(x)\varphi'+r(x)\varphi-\lambda \varphi\right)\mathcal{E}_\lambda(x) \leq -K\star \varphi(x)\mathcal{E}_\lambda(x),
    \end{equation}
    and therefore for $x_0<x<X_1<x_1$, 
    \begin{equation*}
	    a(x)\mathcal{E}_\lambda(x)\varphi(x)-a(X_1)\mathcal{E}_\lambda(X_1)\varphi(X_1) \geq -\int_{X_1}^{x} K\star \varphi(y)\mathcal{E}_\lambda(y)\dd y.
    \end{equation*}
    Finally
    \begin{equation}\label{eq:260129b}
	    \varphi(x)\geq \dfrac{-\int_{X_1}^{x} K\star \varphi(y) {\mathcal{E}_\lambda(y)}\dd y+a(X_1)\mathcal{E}_\lambda(X_1)\varphi(X_1)}{a(x)\mathcal{E}_\lambda(x)}=:\psi(x).
    \end{equation}
	On the other hand, for $x_0<X_1<x<x_1$, we integrate \eqref{eq:260129c} between $ X_1$ and $x$  to find that 
    \begin{equation}\label{eq:260129d}
	    \varphi(x)\leq \dfrac{-\int_{X_1}^{x} K\star \varphi(y){\mathcal{E}_\lambda(y)}\dd y+a(X_1)\mathcal{E}_\lambda(X_1)\varphi(X_1)}{a(x)\mathcal{E}_\lambda(x)}.
    \end{equation}

    Suppose by contradiction that $\lambda\leq \lambda_1^0=\sigma(x_0)$.
	By Lemma \ref{lem:E_lambda}, we have $\mathcal{E}_\lambda(x)=(x-x_0)^{\alpha_0}e^{\mathcal{O}(1)}$ with $\alpha_0=\frac{\sigma(x_0)-\lambda}{a'(x_0)}\geq 0$ when $x$ approaches $x_0$; therefore  the numerator of $\psi(x)$ in  \eqref{eq:260129b} approaches a limit as $x\to x_0^+$ and 
    \begin{equation*}
	    \psi(x)=\dfrac{-\int_{X_1}^{x_0}K\star \varphi(y)\mathcal{E}_\lambda(y)\dd y+a(X_1)\mathcal{E}_\lambda(X_1)\varphi(X_1)+o(1)}{{(x-x_0)^{1+\alpha_0}}e^{\mathcal{O}(1)}},
    \end{equation*}
	and since $\varphi(x)\in L^1$ we must impose  $0\leq a(X_1)\mathcal{E}_\lambda(X_1)\varphi(X_1)\leq \int_{X_1}^{x_0}K\star\varphi(y)\mathcal{E}_\lambda(y)\dd y<0$, which is obviously impossible.  We conclude that $\lambda>\lambda_1^0$.

	Next, suppose by contradiction that 
	$\lambda\leq r(x_1)$. 	Recall that   $\alpha_1= \frac{\sigma(x_1)-\lambda}{a'(x_1)}$, then $\alpha_1\leq -1$ ; moreover we have $\mathcal{E}_\lambda(y)=|x_1-y|^{\alpha_1} e^{\mathcal{O}(1)}$ (from Lemma \ref{lem:E_lambda}) when $y$ is in the neighborhood of $x_1$. Therefore:
	\begin{align*}
	    -\int_{X_1}^xK\star \varphi(y) \mathcal{E}_\lambda(y)\dd y&=
		\begin{cases}
			\frac{(x_1-x)^{\alpha_1+1}-(x_1-X_0)^{1+\alpha_1}}{\alpha_1+1}e^{\mathcal{O}(1)}, &   \text{ if } \alpha_1<-1, \\
			\big(\ln(x_1-x)-\ln(x_1-X_0)\big)e^{\mathcal{O}(1)}, & \text{ if } \alpha_1=-1, 
		\end{cases}\\
		&\xrightarrow[x\to x_1^-]{}-\infty,
	\end{align*}
	and we infer from \eqref{eq:260129d} that  $\varphi(x)<0$ in the neighborhood of $x_1$: a contradiction. We conclude that $\lambda>r(x_1)$. \medskip
	\medskip 

	Thus we have established that $\lambda>\lambda_1^{ext}=\max\big(\lambda_1^0, r(x_1)\big)$. In the rest of the proof we assume that $\varphi(x_1)>0$ and that there exists a normalized principal eigenpair $(\lambda_1^*, \varphi_1^*)$.
	Let us now assume by contradiction that  $\lambda_1^{ext}<\lambda< \lambda_1^*$. Then $\psi(x)\geq R_0(\lambda; A_0)K\star \varphi(x)$ when $x$ is close to $x_0$, and $\varphi_1^*(x)\leq R_0(\lambda; A_0)K\star \varphi(x)$.
	It follows from \eqref{eq:260128f}, $\lambda<\lambda_1^*$  and  $\alpha_0(\lambda)>\alpha_0(\lambda_1^*)$, that  $\varphi_1^*(x) = o\big(\psi(x)\big)$ as $x\to x_0$.  Then it follows from \eqref{eq:260128g} that $\varphi_1^*=\mathcal{O}(1)$ in the neighborhood of $x_1$; let us show that the infimum of $\varphi$ is strictly positive.

	Let $x_*$ realize the minimum of $\varphi$: $\varphi(x_*)=\inf \varphi$, and suppose by contradiction that $\varphi(x_*)=0$. Since $x_*\neq x_1$, then  there exists a sequence $x_n\to x_* $ such that  $a(x_n)\varphi'(x_n)\geq 0$. We have, for $n$ large enough, 
	\begin{align*}
		0&<\left(r(x_n)-\lambda + \int_{\mathbb{S}^1}K(x_n, y) \dd y\right)\varphi(x_n) +\int_{\mathbb{S}^1}K(x_n, y)\big( \varphi(y)-\varphi(x_n)\big)\dd y\\
		&\leq  a(x_n)\varphi'(x_n)+\big(r(x_n)-\lambda\big)\varphi(x_n)+K\star \varphi(x_n)\leq 0,
	\end{align*}
	which is a contradiction. We conclude that $\inf \varphi >0$.

	Thus there exists $k_0>0$ with $k_0\varphi_1^*(x)\leq \varphi(x)$ and the number
	\begin{equation*}
	    k:=\sup \{k'>0\,:\,k''\varphi_1^*(x)\leq \varphi(x) \text{ for all }x\in \mathbb{S}^1\backslash\{x_0\} \text{ and } 0<k''\leq k'\}
	\end{equation*}
	is well-defined and finite.  Since $\varphi_1^*(x) = o\big(\varphi(x)\big)$ in the neighborhood of $x_0$ and $\varphi_1^*$ and $\varphi$ are both continuous on $\mathbb{S}^1\backslash\{x_0\}$, there  exists $x_*\in\mathbb{S}^1\backslash\{x_0\}$ such that $k\varphi_1^*(x_*)=\varphi(x_*)$, while by definition $k\varphi_1^*(x)\leq \varphi(x)$ for all $x\neq x_0$. Thus $x_*$ is a minimum of $\varphi-k\varphi_1^*$ and
	\begin{align*}
		0&\leq K\star(\varphi-k\varphi_1^*)(x) {\leq} -a(x_*)\big(\varphi-k\varphi_1^*\big)_x(x_*)-r(x_*)\big(\varphi(x_*)-k\varphi(x_*)\big)
	    + \lambda \varphi(x_*)-\lambda_1^*k\varphi_1^*(x_*)  \\
	    &= (\lambda-\lambda_1^*)\varphi(x_*),
	\end{align*}
	which proves that $\lambda\geq \lambda_1^*$. Moreover, if $\lambda=\lambda_1^*$, we see that $K\star (\varphi-k\varphi_1^*)=0$, which implies that $\varphi\equiv k\varphi_1^*$. This concludes the proof of Lemma \ref{lem:supercrit}. 
\end{proof}
%%%%%%%%%%%%%%%%%%%%%%%%%%%%%%%%%%%%%%%%%%%%%%%%%%%%%%%%%%%%%%%%%%%%%%%%%%%%%%%%%%%%%%%%%%%%%%%%%%%%%%%%%%%%%%%
We continue with some a priori estimates that are valid for any principal eigenpair. 
\begin{lemma}\label{lem:estimates-princeig}
	{Suppose that Assumption \ref{ASS-20} holds true.} 
    Let $(\lambda, \varphi)$ be a solution of \eqref{eq:main} with $\varphi\in L^1_{per}$ and $\varphi(x)\geq 0$. Let $X_1\in (x_0, x_1)$ and $X_1'\in (x_1, 1+x_0)$ be given. The following assertions hold true.
    \begin{enumerate}
	\item $\lambda>\lambda_1^{ext}$,
	\item $\varphi(x)$ can be written as $\varphi(x)=\varphi_a(x)+\varphi_b(x)$, where  
	    \begin{subequations}\label{eq:251204a}
	\begin{gather}
	    \label{eq:251204aa}
	    \varphi_a(x)= \dfrac{-\int_{x_1}^xK\star \varphi(y) \mathcal{E}_\lambda(y)\dd y}{a(x)\mathcal{E}_\lambda(x)}, \\ 
	    \label{eq:251204ab}
		    \varphi_b(x)=\dfrac{a(X_*)\varphi(X_*)\mathcal{E}_\lambda(X_*) - \int_{X_*}^{x_1}K\star \varphi(y)\mathcal{E}_\lambda(y)\dd y}{a(x)\mathcal{E}_\lambda(x)}, 
	\end{gather}
		and $ X_*=X_1 $ when $ x\in (x_0, x_1) $ and $X_*=X_1'$ when $x\in (x_1, 1+x_0)$.
		Both $\varphi_a$ and $\varphi_b$ are nonnegative $L^1_{per}$ functions.
	    \end{subequations}
	\item Let $\alpha_0=\frac{\sigma(x_0)-\lambda}{a'(x_0)}$, then we have in the neighborhood of $x_0$:
	    \begin{equation}\label{eq:260110a}
	    \varphi_a(x) =  
	    \begin{cases}
		\frac{e^{\mathcal{O}(1)}}{|x-x_0|^{1+\alpha_0}} + o\left(\frac{1}{|x-x_0|^{1+\alpha_0}}\right), & \text{ if }\lambda\in\big(\lambda_1^0, r(x_0)\big)\Leftrightarrow \alpha_0\in (-1, 0), \\ 
		|\ln(|x-x_0|)|e^{\mathcal{O}(1)} + \mathcal{O}(1),%o\left(|\ln(|x-x_0|)|\right), 
		& \text{ if } \lambda= r(x_0)\Leftrightarrow \alpha_0=-1, \\ 
		\frac{K\star \varphi(x_0)}{\lambda-r(x_0)} + o(1), & \text{ if }  \lambda> r(x_0)\Leftrightarrow\alpha_0<-1, \\ 
	    \end{cases}
	    \end{equation}
	\item Let $\alpha_1=\frac{\sigma(x_1)-\lambda}{a'(x_1)}$, then $\varphi_a(x)=e^{\mathcal{O}(1)}$ in the neighborhood of $x_1$. More precisely, 
	    \begin{equation}\label{eq:260110b}
	    \varphi_a(x)\xrightarrow[x\to x_1]{}\dfrac{K\star\varphi(x_1)}{\lambda-r(x_1)}.
	    \end{equation}

	\end{enumerate}
\end{lemma}
%%%%%%%%%%%%%%%%%%%%%%%%%%%%%%%%%%%%%%%%%%%%%%%%%%%%%%%%%%%%%%%%%%%%%%%%%%%%%%%%%%%%%%%%%%%%%%%%%%%%%%%%%%%%%%%
\begin{proof}
	Let $(\lambda, \varphi)\in\mathbb{R}\times D(A)$ be any solution of \eqref{eq:main} with $\varphi(x)\geq 0$. Then it follows from elementary computations that $\varphi(x)$ satisfies \eqref{eq:resolvent-A-formula-0} for $x\in \mathbb{S}^1\backslash \{x_0, x_1\}$, where $X_0$ (resp. $X_0'$) and $X_1$ (resp. $X_1'$) are arbitrary elements of $(x_0, x_1)$ (resp. $(x_1, x+x_0)$). Thus we have, for $x\not\in \{x_0, x_1\}$:
	\begin{equation}
	    \varphi(x)= \dfrac{-\int_{X_*}^xK\star \varphi(y) \mathcal{E}_\lambda(y)\dd y+a(X_*)\mathcal{E}_\lambda(X_*)\varphi(X_*)}{a(x)\mathcal{E}_\lambda(x)} , \text{ with } X_*=
	    \begin{cases} 
		X_1& \text{  if } x\in (x_0, x_1), \\
		X_1' & \text{ if }x\in (x_1, 1+x_0).
	    \end{cases}
	\end{equation}

	1.
	This follows directly from Lemma \ref{lem:supercrit}, since any solution of \eqref{eq:main} also satisfies \eqref{eq:260128h}
\medskip

2.
	Since $\lambda>\lambda_1^{ext}\geq r(x_1)$ we have  $\alpha_1=-1+\frac{r(x_1)-\lambda}{a'(x_1)}>-1$  and thus $\mathcal{E}_\lambda(x)=|x_1-x|^{\alpha_1}e^{\mathcal{O}(1)}\in L^1$ when $x$ is close to $x_1$;   recalling \eqref{eq:251204a} we have for $x\in (x_0, x_1)$:
	\begin{align*}
	    \varphi(x)& = \dfrac{-\int_{x_1}^x K\star \varphi(y)\mathcal{E}_\lambda(y)\dd y+a(X_1)\varphi(X_1)\mathcal{E}_\lambda(X_1) - \int_{X_1}^{x_1}K\star \varphi(y)\mathcal{E}_\lambda(y)\dd y}{a(x)\mathcal{E}_\lambda(x)}.
	\end{align*}
	By a similar argument we find for $x\in(x_1, 1+x_0)$:
	\begin{align*}
	    \varphi(x)& = \dfrac{-\int_{x_1}^x K\star \varphi(y)\mathcal{E}_\lambda(y)\dd y+a(X_1')\varphi(X_1')\mathcal{E}_\lambda(X_1') - \int_{X_1'}^{x_1}K\star \varphi(y)\mathcal{E}_\lambda(y)\dd y}{a(x)\mathcal{E}_\lambda(x)}.
	\end{align*}
	Thus the decomposition $\varphi(x)=\varphi_a(x)+\varphi_b(x)$ satisfying \eqref{eq:251204a} holds true. Clearly $\varphi_a(x)\geq 0$ for all $x\not\in\{x_0, x_1\}$.  Next, multiplying the equality $\varphi(x)=\varphi_a(x)+\varphi_b(x)$ by $a(x)\mathcal{E}_\lambda(x)$, we obtain for all $x\in(x_0, x_1)$:
	\begin{align}
	    \nonumber
	    \varphi(x)a(x)\mathcal{E}_\lambda(x) &= -\int_{x_1}^x K\star \varphi(y)\mathcal{E}_\lambda(y) \dd y + a(X_1)\varphi(X_1)\mathcal{E}_\lambda(X_1)-\int_{X_1}^x K\star \varphi(y)\mathcal{E}_\lambda(y)\dd y \\ 
	    \label{eq:260127a}
	    &=a(X_1)\varphi(X_1)\mathcal{E}_\lambda(X_1) -  \int_{X_1}^x K\star \varphi(y)\mathcal{E}_\lambda(y)\dd y.
	\end{align}
	As remarked above,  $K\star\varphi(y)\mathcal{E}_\lambda(y)$ is locally in $L^1$ for $y$ close to $x_1$ and thus $\int_{X_1}^xK\star\varphi\mathcal{E}_\lambda\to \int_{X_1}^{x_1}K\star \varphi\mathcal{E}_\lambda$ as $x\to x_1^-$; hence \eqref{eq:260127a} yields 
	\begin{equation*}
	    a(X_1)\varphi(X_1)\mathcal{E}_\lambda(X_1)-\int_{X_1}^{x_1}K\star \varphi(y) \mathcal{E}_\lambda(y)\dd y\geq 0. 
	\end{equation*}
	This proves that $\varphi_b(x)\geq 0$ for $x\in (x_0, x_1)$. By using a similar argument based on the identity
	\begin{equation*}
	    \varphi(x)a(x)\mathcal{E}_\lambda(x) = a(X_1')\varphi(X_1')\mathcal{E}_\lambda(X_1')-\int_{X_1'}^x K\star \varphi(y)\mathcal{E}_\lambda(y)\dd y 
	\end{equation*}
	and recalling that $a(x)<0$ on $(x_1, 1+x_0)$, we conclude that $\varphi_b(x)\geq 0$ on $(x_1, 1+x_0)$ as well.

	The fact that $\varphi_b\in L^1_{per} $ will follow from assertions 3 and 4 and the fact that $\varphi_b=\varphi-\varphi_a$. 
	\medskip

	3. and 4. We remark that $\varphi_a=R_0(\lambda; A_0)K\star \varphi$ and $K\star \varphi$ is a strictly positive, continuous function; then \eqref{eq:260110a} and \eqref{eq:260110b} follow from \eqref{eq:260128f} and \eqref{eq:260128g} in Lemma \ref{lem:loc-est}. 
	\medskip

    This completes the proof of Lemma \ref{lem:estimates-princeig}. 
\end{proof}

%%%%%%%%%%%%%%%%%%%%%%%%%%%%%%%%%%%%%%%%%%%%%%%%%%%%%%%%%%%%%%%%%%%%%%%%%%%%%%%%%%%%%%%%%%%%%%%%%%%%%%%%%%%%%%%
\begin{proof}[Proof of Theorem \ref{thm:eigenpairs-L1}]
    We first prove the existence of a normalized principal eigenpair $\big(\lambda_1^*, \varphi_1^*(x)\big)$.
	Since $\kappa(\lambda_1^{ext})=+\infty$ and $\kappa(+\infty)=0$ by Proposition \ref{prop:eigen-0}, it is clear from the intermediate value theorem that there exists $\lambda_1^*\in(\lambda_1^{ext}, +\infty)$ such that $\kappa(\lambda_1^*)=1$. Then there is $\varphi_1^*\in D(A)$, $\varphi^*>0$ and $\int \varphi_1^*=1$,  such that 
	\begin{equation*}
		(\lambda_1^* I-A_0)^{-1}_RB\varphi_1^*=\varphi_1^* \quad \Rightarrow\quad  B\varphi_1^*=(\lambda_1^*I-A)\varphi_1^*\quad \Rightarrow \quad (A+B)\varphi_1^*=\lambda_1^*\varphi_1^*.
	\end{equation*}
	Thus $(\lambda_1^*, \varphi_1^*)$ is a solution of \eqref{eq:main}. Moreover, since $\lambda\mapsto\kappa(\lambda)$ is strictly decreasing, $\lambda_1^*$ is the unique value $\lambda\in(\lambda_1^{ext}, +\infty)$ such that $\kappa(\lambda)=1$; in the rest of the proof, we will denote this value by $\lambda_1^*$ and the corresponding normalized eigenvector by $\varphi_1^*$.

	Next we prove the existence of a family of principal eigenpairs when $\lambda_1^*<\lambda_1^1$. Fix $\lambda\in (\lambda_1^*, \lambda_1^1)$. Then $\kappa(\lambda; 0,0)=\kappa(\lambda)<\kappa(\lambda_1^*)=1$. Since $\gamma\mapsto  \kappa(\lambda; \gamma, 0)$ is strictly increasing and $\kappa(\lambda; \gamma, 0)\to+\infty$ as $\gamma\to +\infty$ by Proposition \ref{prop:eigen-extended},  there exists a unique $\overline{\gamma}(\lambda)$ such that 
	\begin{equation*}
	    \kappa(\lambda;  \overline{\gamma}(\lambda), 0)=1.
	\end{equation*}
	For any fixed $\gamma_1\in[0, \overline{\gamma}(\lambda))$ we have $\kappa(\lambda; \gamma_1, 0)<\kappa(\lambda; \overline{\gamma}(\lambda), 0)=1$, and therefore there exists a unique $\gamma_2=\gamma_2(\lambda, \gamma_1)$ satisfying 
	\begin{equation*}
	    \kappa(\lambda;  \gamma_1, \gamma_2(\lambda, \gamma_1))=1.
	\end{equation*}
	Then it follows from the right-inverse property of $R_{\gamma_1, \gamma_2}(\lambda; A_0)$ (Proposition \ref{prop:resolvent-circle}) that $(\lambda, \varphi_1^{\gamma_1, \gamma_2}(x))$ solves \eqref{eq:main} whenever $\lambda\in(\lambda_1^*, \lambda_1^1)$, $\gamma_1\in [0, \overline{\gamma}(\lambda)]$ and $\gamma_2=\gamma_2(\lambda, \gamma_1)$. The existence of a two-parameter family of normalized principal eigenpairs is done. Next we turn to the characterization of the principal eigenpair, i.e. we prove that any $(\lambda, \varphi)$ with $\varphi\geq 0$ necessarily satisfies $\lambda=\lambda_1^*$ and $\varphi=\varphi_1^*$, or $\lambda\in (\lambda_1^*, \lambda_1^1)$ and $\varphi=\varphi_1^{\gamma_1, \gamma_2(\lambda, \gamma_1)}$, for some $\gamma_1\in[0, \overline{\gamma}(\lambda)]$.

	\medskip

	Let $(\lambda, \varphi)$ be a normalized principal eigenpair of \eqref{eq:main}, that is to say, a solution of \eqref{eq:main} with $\varphi\geq 0$ and $\int\varphi=1$. 
	 In Lemma \ref{lem:estimates-princeig} we have proved that $\lambda\leq \lambda_1^{ext}$ is impossible. We will consider the cases $\lambda_1^{ext}<\lambda<\lambda_1^*$ and $\lambda>\max(\lambda_1^*, \lambda_1^1)$, and in each case we will derive a contradiction. Then, we will consider the case $\lambda=\lambda_1^*$ and prove that $\varphi = \varphi_1^*$; and finally, assuming that $\lambda_1^*<\lambda_1^1$ (which may only happen if $\lambda_1^0<\lambda_1^1$), we will identify $\varphi$ when $\lambda_1^*<\lambda<\lambda_1^1$ and derive a contradiction if $\lambda=\lambda_1^1$.
	\medskip

	$\bullet $ Let us first assume by contradiction that  $\lambda_1^{ext}<\lambda< \lambda_1^*$. Then $\varphi(x)$ satisfies \eqref{eq:251204a} and therefore 
	\begin{equation*}
	    \varphi(x)\geq \varphi_a(x):= \dfrac{-\int_{x_1}^xK\star \varphi(y) e^{\int_{X_0}^y\frac{\sigma(z)-\lambda}{a(z)}\dd z}\dd y  }{a(x)e^{\int_{X_0}^x\frac{\sigma(z)-\lambda}{a(z)}\dd z}},  \text{ for all } x\in(x_0, x_1).
	\end{equation*}
	It follows from \eqref{eq:260110b} that $\varphi(x)\geq e^{\mathcal{O}(1)}$ when $x\to x_1$, and similarly $\varphi_1^*(x)=e^{\mathcal{O}(1)}$ as $x\to x_1$. Then it follows from \eqref{eq:260110a}, $\lambda<\lambda_1^*$  and  $\alpha_0(\lambda)>\alpha_0(\lambda_1^*)$, that  $\varphi_1^*(x) = o\big(\varphi(x)\big)$ as $x\to x_0$. 

	Thus there exists $k_0>0$ with $k_0\varphi_1^*(x)\leq \varphi(x)$ and the number
	\begin{equation*}
	    k:=\sup \{k'>0\,:\,k''\varphi_1^*(x)\leq \varphi(x) \text{ for all }x\in \mathbb{S}^1\backslash\{x_0\} \text{ and } 0<k''\leq k'\}
	\end{equation*}
	is well-defined and finite.  Since $\varphi_1^*(x) = o\big(\varphi(x)\big)$ and $\varphi_1^*$ and $\varphi$ are both continuous on $\mathbb{S}^1\backslash\{x_0\}$, there exists $x_*\in\mathbb{S}^1\backslash\{x_0\}$ such that $k\varphi_1^*(x_*)=\varphi(x_*)$, while by definition $k\varphi_1^*(x)\leq \varphi(x)$ for all $x\neq x_0$. Thus $x_*$ is a minimum of $\varphi-k\varphi_1^*$ and
	\begin{align*}
	    0&\leq K\star(\varphi-k\varphi_1^*)(x) = -a(x_*)\big(\varphi-k\varphi_1^*\big)_x(x_*)-r(x_*)\big(\varphi(x_*)-k\varphi(x_*)\big)
	    + \lambda \varphi(x_*)-\lambda_1^*k\varphi_1^*(x_*)  \\
	    &= (\lambda-\lambda_1^*)\varphi(x_*)<0,
	\end{align*}
	a contradiction. We conclude that there cannot exist a solution $ (\lambda, \varphi)$ of \eqref{eq:main} with $\varphi(x)\geq 0$ and  $\lambda_1^{ext}<\lambda< \lambda_1^*$. 
	\medskip

	$\bullet $ Next  we suppose by contradiction that $\lambda> \max(\lambda_1^*, \lambda_1^1)$. Since $\lambda_1^*>\lambda_1^{ext}>\lambda_1^0$, we infer $\lambda>\lambda_1^m$, hence by Proposition \ref{prop:resolvent-circle} we deduce that $\lambda I-A$ is an invertible operator from $D(A)\to L^1_{per}$ with inverse $R(\lambda; A)$. Rewriting \eqref{eq:main} as 
	\begin{equation*}
	    B \varphi = \lambda\varphi - A\varphi, 
	\end{equation*}
	and applying $R(\lambda; A)$ on both sides, we get 
	\begin{equation*}
	    T(\lambda)\varphi = R(\lambda; A)B \varphi = \varphi.
	\end{equation*}
	Since $\varphi\geq 0$ and $\int \varphi(x)\dd x=1$ we deduce from the uniqueness of the principal eigenpair of $T(\lambda)$ (proved in Proposition \ref{prop:eigen-0}) that  $\kappa(\lambda)=1$. Thus $\lambda=\lambda_1^*$, and the strict inequality in the assumption cannot hold.  The contradiction proves that  $\lambda\leq  \max(\lambda_1^*, \lambda_1^1)$. 
	\medskip

	$\bullet $ Next we consider the case $\lambda= \lambda_1^*$ and suppose by contradiction that $\varphi(x)\not\equiv k\varphi_1^*(x)$ for any $k>0$. Then it follows from Lemma \ref{lem:estimates-princeig} that  $\varphi(x)=\varphi_a+\varphi_b$ where $\varphi_a$ satisfies \eqref{eq:260110a} and \eqref{eq:260110b}, and $\varphi_b\geq 0$; on the other hand, $\varphi_1^*$ satisfies \eqref{eq:260110a} and \eqref{eq:260110b}.
 Thus the number 
	\begin{equation*}
	    k:=\sup \{k'>0\,:\,k''\varphi_1^*(x)\leq \varphi(x) \text{ for all }x\in \mathbb{S}^1\backslash\{x_0\} \text{ and } 0<k''\leq k'\}
	\end{equation*}
	is well-defined and finite. 
Suppose that there exists a point $x_*\in\mathbb{S}^1\backslash\{x_0\}$ such that $k\varphi_1^*(x_*)=\varphi(x_*)$, while by definition $k\varphi_1^*(x)\leq \varphi(x)$ for all $x\neq x_0$. Thus $x_*$ is a minimum of $\varphi-k\varphi_1^*$ and, recalling $k \varphi_1^*\not\equiv \varphi$, we have
	\begin{align*}
	    0&< K\star(\varphi-k\varphi_1^*)(x) = -a(x_*)\big(\varphi-k\varphi_1^*\big)_x(x_*)-r(x_*)\big(\varphi(x_*)-k\varphi(x_*)\big)
	    + \lambda \varphi(x_*)-\lambda_1^*k\varphi_1^*(x_*)  \\
	    &= (\lambda-\lambda_1^*)\varphi(x_*)=0,
	\end{align*}
	a contradiction. We conclude that $k\varphi_1^*(x)<\varphi(x)$ for any $x\in \mathbb{S}^1\backslash\{x_0\}$.

	Next we remark that  
	\begin{equation*}
	    \frac{\varphi(x)}{k\varphi_1^*(x)} = \dfrac{\varphi_a(x)+\varphi_b(x)}{k\varphi_1^*(x)}\geq \dfrac{\int_x^{x_1}K\star \varphi(y) e^{\int_{X_0}^y\frac{\sigma(z)-\lambda}{a(z)}\dd z}\dd y }{\int_x^{x_1}K\star (k\varphi_1^*)(y) e^{\int_{X_0}^y\frac{\sigma(z)-\lambda}{a(z)}\dd z}\dd y } =: \frac{N(x)}{D(x)}.
	\end{equation*}
	Then $\dfrac{N(x)}{D(x)}$ is a continuous function on $\mathbb{S}^1\backslash\{x_0\}$ and we have proved that $\dfrac{N(x)}{D(x)}>1$ for all $x\in \mathbb{S}^1\backslash \{x_0\}$. Moreover 
	\begin{gather*}
		\lim_{x\to x_0^+}\frac{N(x)}{D(x)}= \dfrac{\int_{x_0}^{x_1}K\star \varphi(y) e^{\int_{X_0}^y\frac{\sigma(z)-\lambda}{a(z)}\dd z}\dd y }{\int_{x_0}^{x_1}K\star (k\varphi_1^*)(y) e^{\int_{X_0}^y\frac{\sigma(z)-\lambda}{a(z)}\dd z}\dd y }>1 ,\\ 
\lim_{x\to 1+x_0^-}\frac{N(x)}{D(x)}= \dfrac{\int_{x_1}^{1+x_0}K\star \varphi(y) e^{\int_{X_0}^y\frac{\sigma(z)-\lambda}{a(z)}\dd z}\dd y }{\int_{x_1}^{1+x_0}K\star (k\varphi_1^*)(y) e^{\int_{X_0}^y\frac{\sigma(z)-\lambda}{a(z)}\dd z}\dd y }>1. 
	\end{gather*}
	Thus 
	\begin{equation*}
	    \inf_{x\in\mathbb{S}^1\backslash\{x_0\}}\frac{\varphi(x)}{k\varphi_1^*(x)}\geq\inf_{x\in\mathbb{S}^1\backslash\{x_0\}}\frac{N(x)}{D(x)}>1, 
	\end{equation*}
	which contradicts the maximality of $k$. The contradiction proves that our assumption ($\varphi(x)\not\equiv k\varphi_1^*(x)$ for any $k>0$) cannot hold true. Therefore the unique possibility is $\varphi(x)\equiv k\varphi_1^*(x)$ for some $k>0$, and $\int \varphi(x)\dd x=1$ yields $k=1$.
	\medskip

	$\bullet$ Next we assume that $\lambda_1^*<\lambda_1^1$ and consider the case $\lambda= \lambda_1^1$. Recalling the decomposition $\varphi=\varphi_a+\varphi_b$ from Lemma \ref{lem:estimates-princeig} and \eqref{eq:251204a},  suppose by contradiction that $\beta :=a(X_1)\varphi(X_1)\mathcal{E}_\lambda(X_1)-\int_{X_1}^{x_1}K\star \varphi(y)\mathcal{E}_\lambda(y)\dd y>0$. Then $\varphi(x)\geq \varphi_b(x)$ with 
	\begin{equation*}
	    \varphi_b(x) = \dfrac{\beta}{a(x)\mathcal{E}_\lambda(x)} = \dfrac{e^{\mathcal{O}(1)}}{x_1-x} \text{ as } x\to x_1^-, 
	\end{equation*}
	because $\alpha_1=\frac{\sigma(x_1)-\lambda}{a'(x_1)}=0$; 
	this contradicts the fact that $\varphi\in L^1$. Hence $\beta=0$. Reasoning similarly with $\beta'=a(X_1')\varphi(X_1')\mathcal{E}_\lambda(X_1')-\int_{X_1'}^{x_1}K\star \varphi(y)\mathcal{E}_\lambda(y)\dd y$, we conclude that $\beta'=0$ and finally $\varphi_b(x)\equiv 0$. Thus 
	\begin{equation*}
	    \varphi(x)=\varphi_a(x) = \dfrac{-\int_{x_1}^x K\star \varphi(y)\mathcal{E}_\lambda(y)\dd y}{a(x)\mathcal{E}_\lambda(x)} = R_0(\lambda; A)B\varphi = T(\lambda)\varphi, 
	\end{equation*}
	from which we deduce $\kappa(\lambda)=\kappa(\lambda_1^1)=1$; this is in contradiction with $1=\kappa(\lambda_1^*)>\kappa(\lambda_1^1)$. We conclude that $\lambda\in [\lambda_1^*, \lambda_1^1) $ whenever $\lambda_1^*<\lambda_1^1$.
	\medskip

	$\bullet$ We now assume that $\lambda_1^*<\lambda_1^1$ and consider the case $\lambda_1^*<\lambda< \lambda_1^1$.
	Call 
	\begin{equation*}
	    \gamma_1:=\dfrac{a(X_1)\varphi(X_1)\mathcal{E}_\lambda(X_1)-\int_{X_1}^{x_1}K\star\varphi(y)\mathcal{E}_\lambda(y)\dd y}{\int_{X_1}^{x_1}K\star \varphi(y)\mathcal{E}_\lambda(y)\dd y}
	\end{equation*}
	and 
	\begin{equation*}
	    -\gamma_2:=\dfrac{a(X_1')\varphi(X_1')\mathcal{E}_\lambda(X_1')-\int_{X_1'}^{x_1}K\star\varphi(y)\mathcal{E}_\lambda(y)\dd y}{\int_{x_1}^{X_1'}K\star \varphi(y)\mathcal{E}_\lambda(y)\dd y}, 
	\end{equation*}
	then it follows from \eqref{eq:251204a} and \eqref{eq:resolvent-extended} that 
	\begin{equation*}
	    \varphi = R_{\gamma_1, \gamma_2}(\lambda; A_0)K\star \varphi = R_{\gamma_1, \gamma_2}(\lambda; A_0) B\varphi = T_{\gamma_1, \gamma_2}(\lambda)\varphi.
	\end{equation*}
	We deduce from the uniqueness of the normalized principal eigenpair of $T_{\gamma_1, \gamma_2}(\lambda)$ proved in  Proposition \ref{prop:eigen-extended}, that 
	\begin{equation*}
	    \kappa(\lambda; \gamma_1, \gamma_2)=1.
	\end{equation*}
	This proves that $\gamma_2=\gamma_2(\lambda, \gamma_1)$, where $\gamma_2(\cdot, \cdot)$ is the function defined in the statement of Theorem \ref{thm:eigenpairs-L1}. Thus we have identified $(\lambda, \varphi)$ as a principal eigenpair of an operator in the family $\{T_{\gamma, \gamma_2(\lambda, \gamma)}(\lambda)\}_{\gamma\in [0, \overline{\gamma}(\lambda)]}$.
	\medskip

	Now let us prove that the family $\varphi^{\lambda, \gamma}$ is composed of distinct eigenvectors. If $\lambda\neq \lambda'$, then by  Lemma  \ref{lem:estimates-princeig} we can write $\varphi^{\lambda, \gamma} = \varphi_a^{\lambda, \gamma}+\varphi_b^{\lambda, \gamma} $ and  $\varphi^{\lambda', \gamma'} = \varphi_a^{\lambda', \gamma'}+\varphi_b^{\lambda', \gamma'} $; we have $\varphi_a^{\lambda, \gamma} = \mathcal{O}(1)$ and $\varphi_a^{\lambda', \gamma'}=\mathcal{O}(1)$, and 
	\begin{equation*}
		\varphi_b^{\lambda, \gamma} (x) = (x-x_1)^{-(1+\alpha_1)} e^{\mathcal{O}(1)} \text{ and }\varphi_b^{\lambda', \gamma'} (x) = (x-x_1)^{-(1+\alpha_1')} e^{\mathcal{O}(1)},  
	\end{equation*}
	at least on one side of $x_1$ (i.e. either when $x\to x_1^+$ or $x\to x_1^-$), with $\alpha_1=\frac{\sigma(x_1)-\lambda }{a'(x_1)}$ and  $\alpha_1'=\frac{\sigma(x_1)-\lambda' }{a'(x_1)}$. Therefore $\varphi^{\lambda, \gamma}$ and $\varphi^{\lambda', \gamma'}$ have different behavior in the vicinity of $x_1$, and they cannot be equal. Suppose now that $\lambda=\lambda'$ and $\gamma\neq \gamma'$ and assume by contradiction that $\varphi^{\lambda, \gamma}=\varphi^{\lambda, \gamma'}=:\varphi$. Then $T_{\gamma, \gamma_2(\lambda, \gamma)}(\lambda)\varphi =\varphi =T_{\gamma', \gamma_2(\lambda, \gamma')}(\lambda)\varphi$ hence by \eqref{eq:resolvent-extended} we obtain 
	\begin{equation*}
	    (\gamma-\gamma') \int_{X_1}^{x_1}K\star\varphi(y)\mathcal{E}_\lambda(y)\dd y\mathbbm{1}_{(x_0, x_1)} =\big(\gamma_2(\lambda, \gamma)-\gamma_2(\lambda, \gamma')\big) \int_{x_1}^{X_1'}K\star\varphi(y)\mathcal{E}_\lambda(y)\dd y\mathbbm{1}_{(x_1, 1+x_0)}, 
	\end{equation*}
	which is a contradiction. We have proved that the family $\varphi^{\lambda, \gamma}$ is composed of distinct eigenvectors.\medskip

	Finally we prove \eqref{eq:260126a}. Fix any $\varepsilon>0$, then $\kappa(\lambda_1^*; \varepsilon, 0)>1$, and by the continuity of $\lambda\mapsto \kappa(\lambda; \varepsilon, 0)$ (which is a convex function of a single real variable) the strict inequality holds true in a neighborhood of $\lambda_1^*$; from which we deduce
	\begin{equation*}
	    \limsup_{\lambda\to(\lambda_1^*)^+}\overline{\gamma}(\lambda) \leq \varepsilon. 
	\end{equation*}
	Since the latter inequality holds for any $\varepsilon>0$, we have proved 
	\begin{equation*}
	    \lim_{\lambda\to(\lambda_1^*)^+}\overline{\gamma}(\lambda)=0. 
	\end{equation*}
	Now let us show that $\lim_{ \lambda\to(\lambda_1^1)^-}\overline{\gamma}(\lambda) =0$. For $\varepsilon>0$ sufficiently small, there exists $\lambda \in (\lambda_1^*, \lambda_1^1)$ such that $\kappa(\lambda; \varepsilon, 0)<1$. Since the map $\lambda \mapsto \kappa(\lambda; \varepsilon,0)$ is superconvex, the 
	set of solutions of the inequality $\kappa(\lambda; \varepsilon, 0)\leq 1$ is $[\lambda_\varepsilon^-, \lambda_\varepsilon^+]$ for some $\lambda_\varepsilon^-\leq \lambda_\varepsilon^+$;
	moreover since $\lim_{\lambda\to (\lambda_1^1)^-}\kappa(\lambda; \varepsilon, 0)=+\infty$, we have $\lambda_\varepsilon^+<\lambda_1^1$. Thus $\kappa(\lambda; \varepsilon, 0)>1 $ for $\lambda\in (\lambda_\varepsilon^+, \lambda_1^1)$, which proves
	\begin{equation*}
	    \limsup_{\lambda\to(\lambda_1^1)^-}\overline{\gamma}(\lambda)\leq \varepsilon.
	\end{equation*}
	Since $\varepsilon$ is arbitrarily small, the result is proved.

	The continuity of $\gamma_2(\lambda, \gamma_1)$, of $\bar{\gamma}(\lambda)$ and the continuous dependence of the normalized eigenvector with respect to $(\lambda, \gamma_1, \gamma_2)$ can be obtained by the uniqueness of the corresponding eigenproblem and a compactness argument. We omit the details.
	The proof of Theorem \ref{thm:eigenpairs-L1} is complete.
\end{proof}

\begin{theorem}[Singular eigenpair]\label{thm:eigenpair-singular}
	{Let Assumption \ref{ASS-20} hold true and} $\lambda_1^*$ be the unique solution of $\kappa(\lambda)=1$ as defined in Theorem \ref{thm:eigenpairs-L1}.
    \begin{enumerate}
	    \item  If {$\lambda_1^*<\lambda_1^{1}$}, there exists a unique singular normalized principal eigenpair $(\lambda_1^1, \varphi_1^1)$ with $\varphi_1^1:=\varphi_{ac}^1+\varphi_s^1\delta_{x_1}$, $\varphi^1_{ac}\in L^1_{per, +}$ and $\varphi^1_s\in \mathbb{R}^+\backslash\{0\}$.
	\item Suppose that $(\lambda,\varphi)$ with  $\varphi\in\mathcal{M}_+(\mathbb{S}^1)$ is a normalized solution of 
	    \eqref{eq:main} in the sense of distributions. Then either $\varphi\in L^1_{per}$ and in that case $\varphi\in D(A)$ and solves \eqref{eq:main} almost everywhere; or $\varphi=\varphi_1^1$.
	\item Suppose that $\lambda_1^*<\lambda_1^1$ and fix any sequence  $\lambda_n\to \lambda_1^1$ and $\gamma_n\in\big[0, \overline{\gamma}(\lambda_n)\big]$; then
	    \begin{equation*}
		\varphi_1^{\lambda_n, \gamma_n} \xrightarrow[n\to +\infty]{weak-*}\varphi_1^1 
	    \end{equation*}
	    in the sense of measures.
    \end{enumerate}
\end{theorem}
\begin{proof}
    \textbf{Proof of assertions 2 and 1}
    We begin with the proof of assertion 2; assertion 1 will be proved in the course of the following argument. 
    Let $(\lambda, \varphi)\in\mathbb{R}\times \mathcal{M}_+(\mathbb{S}^1)$ be a distributional solution of \eqref{eq:main}. Then $f(x):=K\star\varphi(x)$ is a  continuous function. Writing \eqref{eq:main} as
    \begin{equation}\label{eq:260128a}
	    a(x)\varphi'+r(x)\varphi - \lambda\varphi=-f(x), 
    \end{equation}
    it is easily seen that \eqref{eq:260128a} considered as an equation on $(x_0, x_1)$ with source $f(x)$ has a solution $\widetilde{\varphi}$ given as
    \begin{equation*}
	    \widetilde{\varphi}(x):=\frac{-1}{a(x)e^{\int_{X_0}^x\frac{r(z)-a'(z)-\lambda}{a(z)}\dd z}}\int_{X_0}^x f(y)e^{\int_{X_0}^y\frac{r(z)-a'(z)-\lambda}{a(z)}\dd z} \dd y, 
    \end{equation*}
    for all $x\in(x_0, x_1)$, where $X_0$ is an arbitrary element of $(x_0, x_1)$. Then we have, in the sense of distributions over $(x_0, x_1)$:
    \begin{equation*}
	\big(a(x)\mathcal{E}_\lambda(x)(\varphi-\widetilde{\varphi})\big)' = 0, 
    \end{equation*}
    which shows that there exists a constant $C$ such that 
    \begin{equation*}
	\varphi = \dfrac{C}{a(x)\mathcal{E}_\lambda(x)}+\widetilde{\varphi}(x).
    \end{equation*}
    Thus $\varphi\in W^{1, \infty}_{loc}(x_0, x_1)$. A similar argument shows $\varphi\in W^{1, \infty}_{loc}(x_1, 1+x_0)$. We conclude that the singularities of $\varphi$, if any, are concentrated in the set $\{x_0, x_1\}$. In particular, $\varphi$ can be written as 
    \begin{equation}\label{eq:260128b}
	\varphi = \varphi_{ac}+\varphi_0\delta_{0}+\varphi_1\delta_{1},
    \end{equation}
    where $\varphi_{ac}\in L^1_{per}$ and $\delta_0=\delta_{x_0}$ and $\delta_1=\delta_{x_1}$ are the Dirac measures concentrated at $x=x_0$ and $x=x_1$, respectively. 
    Plugging \eqref{eq:260128b} into \eqref{eq:main}, we obtain
    \begin{align*}
	0 &= a(x) \big[\varphi_{ac}'+\varphi_0\delta_0'+\varphi_1\delta_1'\big] + r(x) \big[\varphi_{ac}+\varphi_0\delta_0+\varphi_1\delta_1\big]+K\star\big[\varphi_{ac}+\varphi_0 \delta_0 + \varphi_1 \delta_1\big] \\ 
	&\quad - \lambda \big[\varphi_{ac}+\varphi_0\delta_0+\varphi_1\delta_1\big] \\ 
	&=a(x)\varphi_{ac}'+r(x)\varphi_{ac}+K\star \varphi_{ac} + \varphi_0 K(x, x_0)+\varphi_1 K(x, x_1) - \lambda \varphi_{ac} \\ 
	&\quad +\big[r(x_0)-a'(x_0)-\lambda\big] \varphi_0 \delta_0 + \big[r(x_1)-a'(x_1)-\lambda \big] \varphi_1\delta_1,  
    \end{align*}
    where we have used the relations $a(x)\delta_0'=-a'(x_0)\delta_0$ and $a(x)\delta_1'=-a'(x_1)\delta_1$.
    The latter equality can be rewritten as the following system 
    \begin{subequations}\label{eq:260128c}
	\begin{align}
	    \label{eq:260128ca}
	    a(x)\varphi_{ac}'+r(x)\varphi_{ac}+K\star \varphi_{ac} - \lambda \varphi_{ac} &=- \varphi_0 K(x, x_0)-\varphi_1 K(x, x_1) , \\ 
	    \label{eq:260128cb}
	    \big[r(x_0)-a'(x_0)-\lambda \big]\varphi_0 &= 0, \\ 
	    \label{eq:260128cc}
	    \big[r(x_1)-a'(x_1)-\lambda \big]\varphi_1&=0 . 
	\end{align}
    \end{subequations}
	It follows from \eqref{eq:260128ca} that $(\lambda, \varphi_{ac})$ is a solution of the differential inequality \eqref{eq:260128h}; therefore it follows from Lemma \ref{lem:supercrit} that $\lambda>\lambda_1^{ext}$. In particular $\lambda>\lambda_1^0$, and as a consequence \eqref{eq:260128cb} implies that $\varphi_0=0$.
    We now consider two different cases:  $\lambda=\lambda_1^1 = r(x_1)-a'(x_1)$ and $\lambda\neq\lambda_1^1$.
\medskip

$\bullet $ \textbf{Case 1:} $\lambda=\lambda_1^1=r(x_1)-a'(x_1)$. Here recall that $\lambda>\lambda_1^0$ and in particular  $\lambda_1^1\neq \lambda_1^0$. Solving \eqref{eq:260128ca} on $(x_0, x_1)$, we obtain 
    \begin{equation}\label{eq:260128d}
	    \varphi_{ac}(x) = \dfrac{-\int_{X_1}^{x} K\star \varphi(y) \mathcal{E}_\lambda(y)\dd y + C}{a(x)\mathcal{E}_\lambda(x)}.
    \end{equation}
    For $x$ close to $x_1$, it follows from Lemma \ref{lem:E_lambda} that  $\mathcal{E}_\lambda(x) = (x_1-x)^{\alpha_1} e^{\mathcal{O}(1)} $ with $\alpha_1=\frac{\sigma(x_1)-\lambda}{a'(x_1)}=0$ and thus
    \begin{equation*}
	    \varphi_{ac}(x)=\dfrac{-\int_{X_1}^{x_1} K\star \varphi(y) \mathcal{E}_\lambda(y)\dd y + C+o(1)}{a(x)\mathcal{E}_\lambda(x)}
    \end{equation*}
	which is not in $L^1$ close to $x_1$ unless  $C=\int_{X_1}^{x_1}K\star\varphi(y)\mathcal{E}_\lambda(y)\dd y$; we obtain the formula 
    \begin{equation}\label{eq:260128e}
	    \varphi_{ac}(x)=\dfrac{-\int_{x_1}^{x}K\star \varphi(y)\mathcal{E}_\lambda(y)\dd y}{a(x)\mathcal{E}_\lambda(x)}. 
    \end{equation}
	Using a similar argument in $(x_1, 1+x_0)$, the formula \eqref{eq:260128e} is actually valid for $x\in\mathbb{S}^1\backslash\{x_0, x_1\}$. Thus $\varphi_{ac}=R_0(\lambda; A_0)K\star \varphi$ and \eqref{eq:260128g} in Lemma \ref{lem:loc-est} leads to 
	\begin{equation*}
		\varphi_{ac}(x_1)=\dfrac{K\star\varphi(x_1)}{\lambda-r(x_1)}>0.
	\end{equation*}
	Then  Lemma \ref{lem:supercrit} leads to the following alternative: either $\lambda_1^*=\lambda_1^1=\lambda$, in which case $\varphi_{ac}\equiv \varphi_1^*$ and as a consequence $\varphi_1=0$; or $\lambda_1^*<\lambda_1^1$, and $\varphi_1>0$. In the former case, the aim is achieved; let us focus on the latter.  Since $\kappa(\lambda)$ is a decreasing function, we have
	\begin{equation}\label{eq:260129e}
		r\big(T(\lambda_1^1)\big)=\kappa(\lambda_1^1)<r(\lambda_1^*)=1. 
	\end{equation}
	Writing \eqref{eq:260128ca} in abstract form, we have
	\begin{equation*}
		-\varphi_1 K(x, x_1) = (A-\lambda I+B)\varphi_{ac} = -(\lambda I-A)\big[I-R_0(\lambda; A_0)B\big]\varphi_{ac}=-(\lambda I-A)\big[I-T(\lambda_1^1)\big]\varphi_{ac}, 
	\end{equation*}
	and we infer from \eqref{eq:260129e} that $I-T(\lambda_1^1)$ is invertible on $L^1_{per}$. Since $(\lambda I-A)\big[I-R_0(\lambda; A_0)B\big]\varphi_{ac}\in C^0_{per}$, it follows from  assertion 4 in {Proposition} \ref{prop:resolvent-circle} that 
	\begin{equation*}
		\big[I-R_0(\lambda; A_0)B\big]\varphi_{ac} = R_0(\lambda; A_0)\varphi_1 K(x, x_1);
	\end{equation*}
	moreover $\big[I-T(\lambda_1^1)\big]\varphi_{ac}\in D(A_0)$ because $K(x, x_1)\in C^0_{per}$. Therefore we have identified uniquely $\varphi_{ac}$ as 
	\begin{equation}\label{eq:260129f}
		\varphi_{ac} = \varphi_1 \big[I-T(\lambda_1^1)\big]^{-1}R_0(\lambda_1^1; A_0)K(x, x_1).
	\end{equation}
	Conversely it is clear that \eqref{eq:260129f} defines a singular principal eigenpair for \eqref{eq:main}  $\varphi=\varphi_{ac}+\varphi_1\delta_1$ \textbf{which proves assertion 1.}
	Hence we have proved that either $\lambda_1^*=\lambda_1^1=\lambda$, in which case $\varphi\equiv \varphi_1^*$; or $\lambda_1^*<\lambda_1^1$, in which case $(\lambda_1^1, \varphi_1^1=\varphi_{1, ac}^1+\varphi_1 \delta_1)$ belongs to {the span} of an identified vector. The proof of Theorem \ref{thm:eigenpair-singular} in Case 1 is complete. 
	\medskip

	$\bullet $ \textbf{Case 2:} $\lambda\neq \lambda_1^1=r(x_1)-a'(x_1)$. Then \eqref{eq:260128cc} implies that $\varphi_1=0$. It follows that $(\lambda, \varphi)$ is a solution of \eqref{eq:260128ca} with the right-hand side equal to zero, thus a solution of \eqref{eq:main}. Then the result follows from Theorem \ref{thm:eigenpairs-L1}.
	\medskip

	\textbf{Proof of 3.} Since $\int_{\mathbb{S}^1} \varphi_1^{\lambda_n, \gamma_n}(y)\dd y=1$, it follows from the {Prokhorov} Theorem that the sequence $\varphi_1^{\lambda_n, \gamma_n}$  converges up to the extraction of a subsequence, for the weak-$\star$ topology of measures. The limit of any subsequence $(\lambda_1^1, \widetilde{\varphi}_1)$  solves \eqref{eq:main} in the sense of distributions, hence it corresponds to the unique singular eigenpair $(\lambda_1^1, \varphi_1^1)$ obtained in assertion 1. This finishes the proof of assertion 3.
\medskip 

	This completes the proof of Theorem \ref{thm:eigenpair-singular}.
\end{proof}

\section{Application:  A KPP  equation with a nonlocal reaction term}
\label{sec:KPP}
In this section, as an application of our main results, we study the dynamics of the KPP equation with a nonlocal reaction term. A key feature of this equation is that its solution can be represented explicitly in terms of the solution of the associated linear equation through a scalar normalization. This observation allows us to derive the long-time behavior directly from the spectral properties of the linear operator $A+B$. 

Consider 
\begin{equation}\label{kppe}
     u_t=(A+B)u-\bar uu,
        \qquad
        u(0,x)=u_0(x)\ge0,
\end{equation}
where 
$
        \bar u:=\int_{\mathbb S^1}u(t,x)\,dx,
$
 and the linearized equation
\begin{equation}\label{le}
      v_t=(A+B)v
        \qquad
        v(0,x)=u_0(x)\ge0.
\end{equation}

\begin{proposition}[Normalization formula]
The formal solution $u(t,x)$ of \eqref{kppe} is given by
\begin{equation}\label{formula}
        u(t,x)=
        \frac{v(t,x)}
        {1+\displaystyle\int_0^t\int_{\mathbb S^1}v(s,x)\,dx\,ds}
\end{equation}
where $v(t,x)$ is the formal solution of \eqref{le}. 
\end{proposition}

\begin{proof}
 Write $u(t,x)=\alpha(t)v(t,x)$ formally. We shall determine the scalar function $\alpha(t)$.  Since
$$
        u_t=\alpha'v+\alpha(A+B)v,
        \qquad
        (A+B)u-\bar u u=\alpha(A+B)v-\alpha v\bar u.
$$
Thus $u$ solves \eqref{kppe} if and only if
$$
        \alpha'=-\alpha \bar u.
$$
Since $\bar u=\alpha\int_{\mathbb S^1}v(t,x)\,dx$, this is equivalent to
$$
        \alpha'=-\alpha^2\int_{\mathbb S^1}v(t,x)\,dx,
        \qquad
        \frac{d}{dt}\left(\frac1\alpha\right)
        =\int_{\mathbb S^1}v(t,x)\,dx.
$$
The initial condition gives $\alpha(0)=1$, hence
$$
        \frac1{\alpha(t)}
        =1+\int_0^t\int_{\mathbb S^1}v(s,x)\,dx\,ds. 
$$ This completes the proof.
\end{proof}

The normalization formula immediately reduces the nonlinear dynamics of \eqref{kppe} to the asymptotic behavior of the linear equation \eqref{le}. Consequently, the long-time dynamics of \eqref{kppe} follows directly from the spectral properties of $A+B$.

\begin{proposition}
	Let $(\lambda,{\varphi})$ be a {normalized} principal eigenpair of $A+B$ and let $u=u(t,x)$ be the solution of \eqref{kppe} . 

 \begin{enumerate}
	\item Assume that $\lambda\leq 0$. Then $\lim\limits_{t\to\infty}\int _{\mathbb S^1}|u(t,x)|dx=0$ for $u_0=s{\varphi}$ and any $s>0$. Furthermore, for a measurable initial function $u_0$ with $s_1{\varphi}<u_0<s_2{\varphi}$ and $0<s_1<s_2$,  $\lim\limits_{t\to\infty}\int _{\mathbb S^1}|u(t,x)|dx=0$.
\item Assume that $\lambda>0$. Then $\lambda{\varphi}$ is a steady state of \eqref{kppe}, $\lim\limits_{t\to\infty}\int _{\mathbb S^1}|u(t,x)-\lambda{\varphi}(x)|dx=0$ for $u_0=s{\varphi}$ and any $s>0$. Furthermore, for a measurable initial function $u_0$ with $s_1{\varphi}\leq u_0\leq s_2{\varphi}$ and $0<s_1<s_2$,  $\frac{s_1}{s_2}\lambda\leq \liminf\limits_{t\to\infty}\int _{\mathbb S^1}|u(t,x)|dx\leq \limsup\limits_{t\to\infty}\int _{\mathbb S^1}|u(t,x)|dx\leq \frac{s_2}{s_1}\lambda$.
	
 \end{enumerate}	
\end{proposition}
Combining the above proposition with Theorem~\ref{thm:existence}, we immediately obtain the following consequence in the case
$
\lambda_1^*<0<\lambda_1^1.
$
\begin{corollary}Assume that the assumptions of Theorem \ref{thm:existence} hold. Let $\lambda_1^*<0<\lambda_1^1$ and $(\lambda,{\varphi})$ is a {normalized} principal eigenpair of $A+B$ with $\lambda_1^*<\lambda<\lambda_1^1$. Then the following statements hold. \begin{enumerate}
	\item For any $0<\lambda< \lambda_1^1$,   $\lambda{\varphi}$ is steady state of \eqref{kppe}, and $\lim\limits_{t\to\infty}\int _{\mathbb S^1}|u(t,x)-\lambda{\varphi}(x)|dx=0$ for $u_0=s{\varphi}$ and any $s>0$.  
	\item For any $\lambda_1^*\leq \lambda\leq 0$,   $\lim\limits_{t\to\infty}\int _{\mathbb S^1}|u(t,x)|dx=0$ for $u_0=s{\varphi}$ and any $s>0$.   \end{enumerate}	
\end{corollary}

\begin{remark}
The above corollary illustrates that the multiplicity of principal eigenpairs may lead to substantially richer dynamics than in the classical local diffusion case. In particular, different principal eigenpairs may generate different asymptotic states of \eqref{kppe}. Moreover, when the principal eigenfunction is concentrated near the sink point $x_1$, an initial mass localized around $x_1$ may prevent extinction and drive the solution toward a positive steady state.

For comparison, consider the local diffusion equation
\begin{equation}\label{local}
u_t=u_{xx}+a(x)u-\bar uu,
 x\in\mathbb S^1,
\end{equation}
whose solution also admits a normalization formula analogous to \eqref{formula}.

The associated eigenvalue problem
$
u_{xx}+a(x)u=\lambda u,
x\in\mathbb S^1,
$
possesses a unique positive principal eigenpair $(\lambda,\phi)$ normalized by
$
\int_{\mathbb S^1}\phi(x)\,dx=1.
$
Consequently, for every positive initial function
$
u_0\in L^1(\mathbb S^1),
$
the solution of \eqref{local} converges to the unique positive steady state $\lambda\phi$ whenever $\lambda>0$, and converges to $0$ otherwise.

Therefore, unlike the local diffusion case, the nonlocal operator considered here may admit infinitely many principal eigenpairs, leading to multiple possible asymptotic states selected by the initial distribution.
\end{remark}
\medskip

\noindent \textbf{Acknowledgements.} 
X. Liang is supported by the National Natural Science Foundation of China (12331006,12531008)
and XDB0900100
\medskip

\noindent\textbf{Data availability statement.} No datasets were generated or analysed during the current study. 
\medskip

\noindent\textbf{Conflict of interest statement.} On behalf of all authors, the corresponding author states that there is no conflict of interest.

\bibliographystyle{plain}
\bibliography{biblio.bib}

\begin{thebibliography}{10}

\bibitem{Abramowitz-Stegun-1972}
Milton Abramowitz and Irene~A Stegun.
\newblock {\em Handbook of mathematical functions with formulas, graphs, and
  mathematical tables}, volume~55.
\newblock National Bureau of Standards Washington, DC, 1972.

\bibitem{Alfaro-Griette-2018}
Matthieu Alfaro and Quentin Griette.
\newblock Pulsating fronts for {F}isher--{KPP} systems with mutations as models
  in evolutionary epidemiology.
\newblock {\em Nonlinear Anal. Real World Appl.}, 42:255--289, 2018.

\bibitem{Aronson-Weinberger-1975}
Donald~G. Aronson and Hans~F. Weinberger.
\newblock Nonlinear diffusion in population genetics, combustion, and nerve
  pulse propagation.
\newblock In {\em Partial differential equations and related topics ({P}rogram,
  {T}ulane {U}niv., {N}ew {O}rleans, {L}a., 1974)}, pages 5--49. Lecture Notes
  in Math., Vol. 446. 1975.

\bibitem{Berestycki-Diekmann-Nagelkerke-Zegeling-2009}
H.~Berestycki, O.~Diekmann, C.~J. Nagelkerke, and P.~A. Zegeling.
\newblock Can a species keep pace with a shifting climate?
\newblock {\em Bull. Math. Biol.}, 71(2):399--429, 2009.

\bibitem{Berestycki-Nirenberg-Varadhan-1994}
H.~Berestycki, L.~Nirenberg, and S.~R.~S. Varadhan.
\newblock The principal eigenvalue and maximum principle for second-order
  elliptic operators in general domains.
\newblock {\em Comm. Pure Appl. Math.}, 47(1):47--92, 1994.

\bibitem{Berestycki-Fang-2018}
Henri Berestycki and Jian Fang.
\newblock Forced waves of the {F}isher-{KPP} equation in a shifting
  environment.
\newblock {\em J. Differential Equations}, 264(3):2157--2183, 2018.

\bibitem{Berestycki-Hamel-2002}
Henri Berestycki and Fran\c{c}ois Hamel.
\newblock Front propagation in periodic excitable media.
\newblock {\em Comm. Pure Appl. Math.}, 55(8):949--1032, 2002.

\bibitem{Berestycki-Hamel-Nadirashvili-2005}
Henri Berestycki, Fran\c{c}ois Hamel, and Nikolai Nadirashvili.
\newblock The speed of propagation for {KPP} type problems. {I}. {P}eriodic
  framework.
\newblock {\em J. Eur. Math. Soc. (JEMS)}, 7(2):173--213, 2005.

\bibitem{Berestycki-Rossi-2006}
Henri Berestycki and Luca Rossi.
\newblock On the principal eigenvalue of elliptic operators in {$\mathbb R^N$}
  and applications.
\newblock {\em J. Eur. Math. Soc. (JEMS)}, 8(2):195--215, 2006.

\bibitem{Birindelli-1995}
Isabeau Birindelli.
\newblock Hopf's lemma and anti-maximum principle in general domains.
\newblock {\em J. Differential Equations}, 119(2):450--472, 1995.

\bibitem{Bouhours-Giletti-2019}
Juliette Bouhours and Thomas Giletti.
\newblock Spreading and vanishing for a monostable reaction-diffusion equation
  with forced speed.
\newblock {\em J. Dynam. Differential Equations}, 31(1):247--286, 2019.

\bibitem{Coville-2010}
J\'er\^ome Coville.
\newblock On a simple criterion for the existence of a principal eigenfunction
  of some nonlocal operators.
\newblock {\em J. Differential Equations}, 249(11):2921--2953, 2010.

\bibitem{Coville-2013}
J\'er\^ome Coville.
\newblock Singular measure as principal eigenfunctions of some nonlocal
  operators.
\newblock {\em Appl. Math. Lett.}, 26(8):831--835, 2013.

\bibitem{Coville-Davila-Martinez-2013}
J\'{e}r\^{o}me Coville, Juan D\'{a}vila, and Salom\'{e} Mart\'{i}nez.
\newblock Pulsating fronts for nonlocal dispersion and {KPP} nonlinearity.
\newblock {\em Ann. Inst. H. Poincar\'{e} Anal. Non Lin\'{e}aire},
  30(2):179--223, 2013.

\bibitem{Coville-Hamel-2020}
J\'er\^ome Coville and Fran\c{c}ois Hamel.
\newblock On generalized principal eigenvalues of nonlocal operators with a
  drift.
\newblock {\em Nonlinear Anal.}, 193:111569, 20, 2020.

\bibitem{Diekmann-Heesterbeek-Metz-1990}
O.~Diekmann, J.~A.~P. Heesterbeek, and J.~A.~J. Metz.
\newblock On the definition and the computation of the basic reproduction ratio
  {$R_0$} in models for infectious diseases in heterogeneous populations.
\newblock {\em J. Math. Biol.}, 28(4):365--382, 1990.

\bibitem{Diekmann-1979}
Odo Diekmann.
\newblock Run for your life. {A} note on the asymptotic speed of propagation of
  an epidemic.
\newblock {\em J. Differential Equations}, 33(1):58--73, 1979.

\bibitem{Ding-Liang-2015}
Weiwei Ding and Xing Liang.
\newblock Principal eigenvalues of generalized convolution operators on the
  circle and spreading speeds of noncompact evolution systems in periodic
  media.
\newblock {\em SIAM J. Math. Anal.}, 47(1):855--896, 2015.

\bibitem{Ducrot-Griette-Liu-Magal-2022}
Arnaud Ducrot, Quentin Griette, Zhihua Liu, and Pierre Magal.
\newblock {\em Differential equations and population dynamics {I}.
  {I}ntroductory approaches}.
\newblock Lecture Notes on Mathematical Modelling in the Life Sciences.
  Springer, Cham, [2022] \copyright 2022.
\newblock With forewords by Jacques Demongeot and Glenn Webb.

\bibitem{Engel-Nagel-2000}
Klaus-Jochen Engel and Rainer Nagel.
\newblock {\em One-parameter semigroups for linear evolution equations}, volume
  194 of {\em Graduate Texts in Mathematics}.
\newblock Springer-Verlag, New York, 2000.
\newblock With contributions by S. Brendle, M. Campiti, T. Hahn, G. Metafune,
  G. Nickel, D. Pallara, C. Perazzoli, A. Rhandi, S. Romanelli and R.
  Schnaubelt.

\bibitem{Freidlin-Gertner-1979}
J\"{u}rgen Gertner and Mark~I. Fre\u\i{}dlin.
\newblock The propagation of concentration waves in periodic and random media.
\newblock {\em Dokl. Akad. Nauk SSSR}, 249(3):521--525, 1979.

\bibitem{Greiner-Voigt-Wolff-1981}
G\"unther Greiner, J\"urgen Voigt, and Manfred Wolff.
\newblock On the spectral bound of the generator of semigroups of positive
  operators.
\newblock {\em J. Operator Theory}, 5(2):245--256, 1981.

\bibitem{Griette-2019}
Quentin Griette.
\newblock Singular measure traveling waves in an epidemiological model with
  continuous phenotypes.
\newblock {\em Trans. Amer. Math. Soc.}, 371(6):4411--4458, 2019.

\bibitem{Jenkinson-2019}
Oliver Jenkinson.
\newblock Ergodic optimization in dynamical systems.
\newblock {\em Ergodic Theory Dynam. Systems}, 39(10):2593--2618, 2019.

\bibitem{Kato-1982}
Tosio Kato.
\newblock Superconvexity of the spectral radius, and convexity of the spectral
  bound and the type.
\newblock {\em Math. Z.}, 180(2):265--273, 1982.

\bibitem{Lam-Lou-2016}
King-Yeung Lam and Yuan Lou.
\newblock Asymptotic behavior of the principal eigenvalue for cooperative
  elliptic systems and applications.
\newblock {\em J. Dynam. Differential Equations}, 28(1):29--48, 2016.

\bibitem{Liang-Zhou-2022b}
Xing Liang and Tao Zhou.
\newblock Propagation of {KPP} equations with advection in one-dimensional
  almost periodic media and its symmetry.
\newblock {\em Adv. Math.}, 407:Paper No. 108568, 32, 2022.

\bibitem{Liang-Zhou-2022a}
Xing Liang and Tao Zhou.
\newblock Spreading speeds of nonlocal {KPP} equations in heterogeneous media.
\newblock {\em Acta Math. Sin. (Engl. Ser.)}, 38(1):161--178, 2022.

\bibitem{Magal-Ruan-2018}
Pierre Magal and Shigui Ruan.
\newblock {\em Theory and applications of abstract semilinear {C}auchy
  problems}, volume 201 of {\em Applied Mathematical Sciences}.
\newblock Springer, Cham, 2018.
\newblock With a foreword by Glenn Webb.

\bibitem{Meyer--Nieberg-1991}
Peter Meyer-Nieberg.
\newblock {\em Banach lattices}.
\newblock Universitext. Springer-Verlag, Berlin, 1991.

\bibitem{Morro-Sant'Anna-Varandas-2020}
Marcus Morro, Roberto Sant'Anna, and Paulo Varandas.
\newblock Ergodic optimization for hyperbolic flows and {L}orenz attractors.
\newblock {\em Ann. Henri Poincar\'e}, 21(10):3253--3283, 2020.

\bibitem{Nadin-2009}
Gr\'{e}goire Nadin.
\newblock The principal eigenvalue of a space-time periodic parabolic operator.
\newblock {\em Ann. Mat. Pura Appl. (4)}, 188(2):269--295, 2009.

\bibitem{Rawal-Shen-2012}
Nar Rawal and Wenxian Shen.
\newblock Criteria for the existence and lower bounds of principal eigenvalues
  of time periodic nonlocal dispersal operators and applications.
\newblock {\em J. Dynam. Differential Equations}, 24(4):927--954, 2012.

\bibitem{Schaefer-1974}
Helmut~H. Schaefer.
\newblock {\em Banach lattices and positive operators}, volume Band 215 of {\em
  Die Grundlehren der mathematischen Wissenschaften}.
\newblock Springer-Verlag, New York-Heidelberg, 1974.

\bibitem{Shen-Zhang-2012}
Wenxian Shen and Aijun Zhang.
\newblock Stationary solutions and spreading speeds of nonlocal monostable
  equations in space periodic habitats.
\newblock {\em Proc. Amer. Math. Soc.}, 140(5):1681--1696, 2012.

\bibitem{Taylor-1952}
A.~E. Taylor.
\newblock L'{H}ospital's rule.
\newblock {\em Amer. Math. Monthly}, 59:20--24, 1952.

\bibitem{Thieme-1977}
Horst~R. Thieme.
\newblock The asymptotic behaviour of solutions of nonlinear integral
  equations.
\newblock {\em Math. Z.}, 157(2):141--154, 1977.

\bibitem{Thieme-1998a}
Horst~R. Thieme.
\newblock Positive perturbation of operator semigroups: growth bounds,
  essential compactness, and asynchronous exponential growth.
\newblock {\em Discrete Contin. Dynam. Systems}, 4(4):735--764, 1998.

\bibitem{Thieme-1998b}
Horst~R. Thieme.
\newblock Remarks on resolvent positive operators and their perturbation.
\newblock {\em Discrete Contin. Dynam. Systems}, 4(1):73--90, 1998.

\bibitem{van_den_Driessche-Watmough-2002}
P.~van~den Driessche and James Watmough.
\newblock Reproduction numbers and sub-threshold endemic equilibria for
  compartmental models of disease transmission.
\newblock {\em Math. Biosci.}, 180:29--48, 2002.
\newblock John A.\ Jacquez memorial volume.

\bibitem{Webb-1987}
Glenn~F. Webb.
\newblock An operator-theoretic formulation of asynchronous exponential growth.
\newblock {\em Trans. Amer. Math. Soc.}, 303(2):751--763, 1987.

\bibitem{Weinberger-1982}
Hans~F. Weinberger.
\newblock Long-time behavior of a class of biological models.
\newblock {\em SIAM J. Math. Anal.}, 13(3):353--396, 1982.

\bibitem{Weinberger-2002}
Hans~F. Weinberger.
\newblock On spreading speeds and traveling waves for growth and migration
  models in a periodic habitat.
\newblock {\em J. Math. Biol.}, 45(6):511--548, 2002.

\bibitem{Xin-2000}
Jack Xin.
\newblock Front propagation in heterogeneous media.
\newblock {\em SIAM Rev.}, 42(2):161--230, 2000.

\end{thebibliography}

\appendix
\section{Irreducibility}
Here we recall a classical fact, for which we could not find a proper proof in the literature. 
\begin{definition}
    Let $E$ be a Banach lattice with positive cone $E_+$. 
    \begin{enumerate}
	\item An \textit{ideal} $I$ of $E$ is a subspace $I\subset E$ such that for any $x\in I$, $|y|\leq |x| \Rightarrow y\in I$. 
	\item Let $T:E \to E$ be a positive operator. An ideal $I$ is \textit{$T$-invariant} if $T(I)\subset I$. $T$ is \textit{irreducible} if there is no $T$-invariant ideal except $\{0\}$ and $E$.
    \end{enumerate}
\end{definition}
We refer to Schaefer \cite{Schaefer-1974}, Meyer-Nieberg \cite{Meyer--Nieberg-1991} for more information on irreducible operators. 

We have the following characterization. 
\begin{proposition}\label{prop:charac-irred}
    $T:E\to E$ is irreducible if, and only if, for any positive linear form $\varphi\in E^*_+$ with $\varphi\not\equiv 0$ and for any $x\in E_+$ with $x\neq 0$, we have
    \begin{equation}\label{eq:irred-pos-form}
	\exists n\in\mathbb{N}^*:\qquad \langle \varphi, T^nx\rangle>0. 
    \end{equation}
\end{proposition}
\begin{proof}
    We show that whenever $T$ is irreducible,  \eqref{eq:irred-pos-form} holds for any nontrivial $\varphi\in E^*_+$ and $x\in E_+$. We will prove the contrapositive of this statement. Suppose that \eqref{eq:irred-pos-form} does not hold, and let $\varphi\in E^*_+$ and $x\in E_+$ be nontrivial and such that 
    \begin{equation*}
	\langle\varphi, T^nx\rangle = 0 \text{ for all  }n\in\mathbb{N}^*.
    \end{equation*}
    Let $I$ be the ideal generated by $\{T^nx:n\in\mathbb{N}\}$, then obviously $T(I)\subset I$. Moreover if $z\in I$, then $z$ is the limit of a sequence $y_n$, where $y_n$ is a finite combination of elements of $I_{T^{n_k} x}$. Since, for any $n\in\mathbb{N}\backslash\{0\}$, we have $\langle \varphi, y_n\rangle = 0$ (indeed $I_{T^{n_k} x}$ is the closure of the space generated by all $f$ with $0\leq f\leq T^{n_k} x$, and $0\leq \langle \varphi, f\rangle \leq \langle \varphi, T^{n_k} x\rangle=0$), we conclude that $\langle \varphi, z\rangle = 0$. Hence $I\subset \ker \varphi$ which proves that $I\neq E$, and obviously $I\neq \{0\}$, thus we have proved that $T$ is not irreducible.

	Conversely, we show that whenever \eqref{eq:irred-pos-form} holds for any nontrivial $\varphi\in E^*_+$ and $x\in E_+$, then $T$ is irreducible. Again, we  prove the contrapositive. Assume that $T$ is not irreducible; and let $I$ be a closed ideal such that $I\not\in\{\{0\}, E\}$ and $T(I)\subset I$. First we remark that the function $p(x):=\Vert x_+\Vert $ is a sublinear (where $x_+=\max(x, 0)$).  Next we choose $x_0\in E_+$ such that $x_0\not\in I$, {$\Vert x_0 \Vert =1$}, and define the linear form $\varphi:I\oplus \text{vect}\{x_0\}\to \mathbb{R}$ by $\varphi(x)\equiv 0$ for $x\in I$ and $\varphi(\lambda x_0)=\lambda$. By the Hahn-Banach theorem there exist an extension $\varphi:E\to \mathbb{R}$ such that 
    \begin{equation*}
	\varphi(x)\leq p(x), \quad ^\forall x\in E.
    \end{equation*}
    Then for any $x\in E_+$ we have $\varphi(-x) \leq p(-x)=\Vert (-x)_+\Vert =0$, which shows that $\varphi(x)\geq 0$; hence $\varphi$ is positive; obviously $\varphi$ is bounded since it is bounded on $E_+$. Hence we have constructed a linear functional $\varphi\in E^*_+$ such that 
    \begin{equation*}
	\langle \varphi, Tx\rangle=0, \qquad ^\forall x\in I.
    \end{equation*}
    This proves that \eqref{eq:irred-pos-form} does not hold. 

    The proof of Proposition \ref{prop:charac-irred} is complete. 
\end{proof}

\end{document}